\documentclass[leqno,11pt,a4paper,twoside]{article}%
\usepackage{amssymb}
\usepackage{amsfonts}
\usepackage{amsmath}
\usepackage{graphicx}%
\providecommand{\U}[1]{\protect\rule{.1in}{.1in}}
\newtheorem{theorem}{Theorem}

\newtheorem{corollary}[theorem]{Corollary}

\newtheorem{definition}[theorem]{Definition}
\newtheorem{example}[theorem]{Example}

\newtheorem{lemma}[theorem]{Lemma}

\newtheorem{proposition}[theorem]{Proposition}
\newtheorem{remark}[theorem]{Remark}

\newenvironment{proof}[1][Proof]{\noindent\textbf{#1.} }{\ \rule{0.5em}{0.5em}}
\makeatletter
\newcommand\footnoteref[1]{\protected@xdef\@thefnmark{\ref{#1}}\@footnotemark}
\makeatother
\begin{document}

\title{A Piecewise Deterministic Markov Toy Model for Traffic/Maintenance and
Associated Hamilton-Jacobi Integrodifferential Systems on Networks}
\author{Dan Goreac\thanks{Universit\'{e} Paris-Est, LAMA, UMR8050, 5, boulevard
Descartes, Cit\'{e} Descartes, Champs-sur-Marne, 77454 Marne-la-Vall\'{e}e,
France} \thanks{Corresponding author, Email : Dan.Goreac@univ-mlv.fr, Tel. :
+33 (0)1 60 95 75 27, Fax : +33 (0)1 60 95 75 45}
\thanks{\textbf{Acknowledgement.} The work of the first author has been
partially supported by he French National Research Agency project PIECE,
number \textbf{ANR-12-JS01-0006.}}, Magdalena Kobylanski\footnotemark[1]
\thanks{Email : Magdalena.Kobylanski@univ-mlv.fr}, Miguel
Martinez\footnotemark[1] \thanks{Email : Miguel.Martinez@univ-mlv.fr}
\and \date{}}
\maketitle

\begin{abstract}
We study optimal control problems in infinite horizon when the dynamics belong
to a specific class of piecewise deterministic Markov processes constrained to
star-shaped networks (corresponding to a toy traffic model). We adapt the
results in \cite{Soner86_2} to prove the regularity of the value function and
the dynamic programming principle. Extending the networks and Krylov's
"shaking the coefficients" method, we prove that the value function can be
seen as the solution to a linearized optimization problem set on a convenient
set of probability measures. The approach relies entirely on viscosity
arguments. As a by-product, the dual formulation guarantees that the value
function is the pointwise supremum over regular subsolutions of the associated
Hamilton-Jacobi integrodifferential system. This ensures that the value
function satisfies Perron's preconization for the (unique) candidate to
viscosity solution.

\end{abstract}


\textbf{Acknowledgement.} The authors would like to thank the anonymous
referees for constructive remarks allowing to improve the manuscript.

\section{Introduction}

This paper aims at the study of optimal control problems in infinite horizon
when the dynamics belong to a specific class of piecewise deterministic Markov
processes constrained to networks. The starting point is a toy model inspired
by traffic. Our point of view is the one of a traffic regulator who observes
the generic traffic $X_{\cdot}$ and has the possibility to intervene in the
regulation by imposing speed limits via some (external) control. In this basic
model, the generic vehicle should remain on some star-shaped network
containing several edges bound to a common intersection. At the same time as
the traffic, the regulator should ensure the maintenance of the network by
observing a second (pure jump) component $\Gamma_{\cdot}$ (known as mode). The
functionality of the network evolves stochastically and damage to a specific
edge occurs exponentially distributed with a parameter $\lambda\left(
X,\Gamma,\alpha\right)  $ depending on the traffic, on the previous state of
the network and on regulator's control policy $\alpha$. In this context of
controlled switched Piecewise Deterministic Markov Processes (PDMP), the
regulator seeks to minimize its (discounted) operating cost
\[
v^{\delta}\left(  x,\gamma\right)  :=\inf_{\alpha,X_{\cdot}^{x,\gamma,\alpha
}\in network}\mathbb{E}\left[  \int_{0}^{\infty}e^{-\delta t}l_{\Gamma
_{t}^{x,\gamma,\alpha}}\left(  X_{t}^{x,\gamma,\alpha},\alpha_{t}\right)
dt\right]  .
\]
In this paper, we study the Hamilton-Jacobi integrodifferential systems on
networks associated to the previous control problem.

To our best knowledge, for deterministic dynamics, the constrained optimal
control problem with continuous cost was studied for the first time in
\cite{Soner86_1} (see also \cite{Soner86_2} for a stochastic framework). The
value function of an infinite horizon control problem with space constraints
was characterized as a continuous solution to a corresponding
Hamilton--Jacobi--Bellman equation. For discontinuous cost functionals, the
deterministic control problem with state constraints was studied in
\cite{frankowska_plaskacz_98}, \cite{frankowska_vinter_98},
\cite{plaskacz_quincampoix_00} using viability theory tools. However, the
results of these papers do not directly apply to (deterministic) control
problems on star-shaped networks. Several very recent results are available on
this subject when dealing with deterministic systems (cf.
\cite{Schieborn:2013aa}, \cite{AchdouCamilliCutri2013},
\cite{Imbert_Monneau_Zidani}, \cite{CamilliMarchi2013},
\cite{AchdouOudetTchou2013}, \cite{ImbertMoneau2015_Flux}). The cited papers
rely either on Bellman's approach or on Perron's method for the existence of
solutions of the associated Hamilton-Jacobi equation and propose several
methods for the uniqueness part. There is also an increasing literature on
problems inspired by stratified domains or interfaces and discontinuities that
partly share the same difficulties (e.g. \cite{Bresssan_Hong_2007},
\cite{Barnard_Wolenski_2013}, \cite{rao_siconolfi_zidani},
\cite{BarlesBrianiChasseigne2013}, \cite{BarlesBrianiChasseigneSIAM2014}).

Our control problem is governed by a switch PDMP with characteristic triple
$\left(  f,\lambda,Q\right)  $ (cf. \cite{davis_93}, see also Section
\ref{constr+ass} for the explicit construction). A switch process is often
used to model various aspects in biology (see \cite{cook_gerber_tapscott_98},
\cite{crudu_debussche_radulescu_09}, \cite{Wainrib_Thieullen_Pakdaman_2010},
\cite{crudu_Debussche_Muller_Radulescu_2012}, \cite{G8}), reliability or
storage (in \cite{Boxma_Kaspi_Kella_Perry_2005},
\cite{DufourDutuitGonzalesZhang2008}), finance (in \cite{Rolski_Schmidli_2009}%
), communication networks (\cite{graham2009},
\cite{Bardet_Christen_Guillin_Malrieu_Zitt_2012}). We proceed as follows. In
the first part, we prove that $v^{\delta}$ satisfies, in some generalized
viscosity sense the associated Hamilton-Jacobi integrodifferential equation.
As in the deterministic counterpart, we use Bellman's approach. We begin
Section \ref{Section_DPP_reg} with proving the regularity of the deterministic
value function and the dynamic programming principle (DPP) for this case. For
available (active) roads, the controllability assumptions are the same as
those in \cite{AchdouCamilliCutri2013}. However, entering inactive roads from
intersection should be prohibited and other assumptions must be made for this
case in order to guarantee the uniform continuity of the value function. Next,
we iterate the value functions and the DPP between jumps to prove the uniform
continuity of the (stochastic) value function and the DPP. As a by-product, we
prove that the value function satisfies in a (relaxed) viscosity sense the
associated Hamilton-Jacobi integrodifferential system (in Section
\ref{Section_Existence}).

We then focus on a different notion of uniqueness (in Section
\ref{Section_Uniqueness}): The well-known method of Perron consists in
proposing the supremum over regular subsolution as candidate to the viscosity
solution. Using this intuition, we proceed backward and prove that the value
function given in the previous section is the pointwise supremum over such
regular subsolutions (with a slightly modified notion). The major argument in
proving this result is to extend the intersection with some additional
directions and impose convenient extensions of the dynamics. Then, we adapt
Krylov's "shaking the coefficients" method (cf. \cite{krylov_00},
\cite{barles_jakobsen_02}) to exhibit a sequence of regular subsolutions of
our Hamilton-Jacobi system converging to the initial control problem. These
arguments allow the linearization of the value function. It is shown (in
Theorem \ref{TH_Lin_Network}) that the value function can be interpreted in
connection to an optimization problem set on a family of convenient
probability measures. This family is completely described by the Dynkin
operator of our process. Moreover, the dual value allows one to state that the
initial value function is, indeed, the pointwise supremum over regular subsolutions.

The paper is organized as follows. In Section \ref{constr+ass}, we recall the
basic construction of piecewise deterministic Markov switch processes and give
the main assumptions on the dynamics. We present our traffic model and
introduce the different types of admissible controls and the controllability
assumptions in Section 3. Section \ref{Section_DPP_reg} is dedicated to the
study of regularity of the value function and the dynamic programming
principles. The basic ingredient is the technical projection Lemma
\ref{Projection_Lemma} allowing to prove the uniform continuity of the value
function in the deterministic setting (in Theorem \ref{Th_Cont_0}). We proceed
as in \cite{Soner86_2} by iterating the value function and the dynamic
programming principle. In Section \ref{Section_Existence}, we introduce a
sequential relaxation of the dynamics and prove that the regular value
function exhibited before satisfies, in some generalized viscosity sense, the
associated Hamilton-Jacobi intergrodifferential system. Section
\ref{Section_Uniqueness} is dedicated to the linearization of our value
function. We begin with extending the graph and the dynamics by mirroring the
trajectories in the inactive case and using the inertia otherwise. We briefly
present the adaptation of Krylov's "shaking the coefficients" method and
exhibit a family of regular subsolutions converging to the initial value
function (in Theorem \ref{ThConvergence}). The main ingredients in proving the
convergence are successive projection arguments given by Lemmas
\ref{Projection_Lemma_eps} and \ref{Projection_Lemma_eps_Partii} (whose proofs
are postponed to the Appendix). The main result (Theorem \ref{TH_Lin_Network})
shows that the value function can be interpreted in connection to an
optimization problem set on a family of convenient probability measures.
Moreover, the dual of this problem allows one to characterize the value as the
pointwise supremum over regular subsolutions (as predicted by Perron's method).

\section{Standard construction of controlled switched PDMPs\label{constr+ass}}

We consider $A$ \ (the control space) to be a compact subspace of a metric
space $%
\mathbb{R}
^{d}$ and $%
\mathbb{R}
^{m}$ be the state space, for some $d,m\geq1.$ Moreover, we consider a finite
set $E.$

We summarize the construction of controlled piecewise deterministic Markov
processes (PDMP) of switch type (cf. \cite{Davis_84}, \cite{Davis_86},
\cite{davis_93}) having as characteristic triple $f_{\gamma}:%
\mathbb{R}
^{m}\times A\longrightarrow%
\mathbb{R}
^{m},$ for all $\gamma\in E,$ $\lambda:%
\mathbb{R}
^{m}\times E\times A\longrightarrow%
\mathbb{R}
_{+}$ and $Q:%
\mathbb{R}
^{m}\times E^{2}\times A\longrightarrow\left[  0,1\right]  .$ These functions
are assumed to satisfy some usual continuity conditions (to be made precise at
the end of the section). The switch PDMP is constructed on a space $\left(
\Omega,\mathcal{F},\mathbb{P}\right)  $ allowing to consider a sequence of
independent, $\left[  0,1\right]  $ uniformly distributed random variables
(e.g. the Hilbert cube starting from $\left[  0,1\right]  $ endowed with its
Lebesgue measurable sets and the Lebesgue measure for coordinate, see
\cite[Section 23]{davis_93}). We let $\mathbb{L}^{0}\left(
\mathbb{R}
_{+}\times%
\mathbb{R}
^{m}\times E;A\right)  $ denote the space of $A$-valued Borel measurable
functions defined on $%
\mathbb{R}
^{m}\times E\times%
\mathbb{R}
_{+}$. Whenever $\alpha_{1}\in\mathbb{L}^{0}\left(
\mathbb{R}
_{+}\times%
\mathbb{R}
^{m}\times E;A\right)  $ and $\left(  t_{0},x_{0},\gamma_{0}\right)  \in%
\mathbb{R}
_{+}\times%
\mathbb{R}
^{m}\times E,$ we consider the ordinary differential equation%
\[
\left\{
\begin{array}
[c]{l}%
dy_{\gamma_{0}}\left(  t;t_{0},x_{0},\alpha_{1}\right)  =f_{\gamma_{0}}\left(
y_{\gamma_{0}}\left(  t;t_{0},x_{0},\alpha_{1}\right)  ,\alpha_{1}\left(
t-t_{0};x_{0},\gamma_{0}\right)  \right)  dt,\text{ }t\geq t_{0},\\
y_{\gamma_{0}}\left(  t_{0};t_{0},x_{0};\alpha_{1}\right)  =x_{0.}%
\end{array}
\right.
\]
For the sake of simplicity, whenever $t_{0}=0,$ we denote by $y_{\gamma_{0}%
}\left(  t;x_{0},\alpha_{1}\right)  $ the solution of the previous ordinary
differential equation such that $y_{\gamma_{0}}\left(  0;x_{0},\alpha
_{1}\right)  =x_{0}$.

We pick the first jump time $\tau_{1}$ such that the jump rate is
$\lambda\left(  y_{\gamma_{0}}\left(  t;x_{0},\alpha_{1}\right)  ,\gamma
_{0},\alpha_{1}\left(  t;x_{0},\gamma_{0}\right)  \right)  $ i.e.%
\[
\mathbb{P}\left(  \tau_{1}\geq t\text{ }/\text{ }y_{\gamma_{0}}\left(
0;x_{0};\alpha_{1}\right)  =x_{0}\right)  =\exp\left(  -\int_{0}^{t}%
\lambda\left(  y_{\gamma_{0}}\left(  s;x_{0},\alpha_{1}\right)  ,\gamma
_{0},\alpha_{1}\left(  s;x_{0},\gamma_{0}\right)  \right)  ds\right)  .
\]
The controlled piecewise deterministic Markov processes (PDMP) is defined by
\[
\left(  X_{t}^{x_{0},\gamma_{0},\alpha},\Gamma_{t}^{x_{0},\gamma_{0},\alpha
}\right)  =\left(  y_{\gamma_{0}}\left(  t;x_{0},\alpha_{1}\right)
,\gamma_{0}\right)  ,\text{ if }t\in\left[  0,\tau_{1}\right)  .
\]
The post-jump location is denoted by $\left(  Y_{1},\Upsilon_{1}\right)  .$
Since we deal with continuous switching, $Y_{1}=y_{\gamma_{0}}\left(  \tau
_{1};x_{0},\alpha_{1}\right)  $ and $\Upsilon_{1}$ is a random variable who
has $Q\left(  y_{\gamma_{0}}\left(  \tau;x_{0},\alpha\right)  ,\gamma
_{0},\alpha_{1}\left(  \tau,x_{0},\gamma_{0}\right)  ,\cdot\right)  $ as
conditional distribution given $\tau_{1}=\tau.$ Starting from $\left(
Y_{1},\Upsilon_{1}\right)  $ at time $\tau_{1}$, we select the inter-jump time
$\tau_{2}-\tau_{1}$ such that
\[
\mathbb{P}\left(  \tau_{2}-\tau_{1}\geq t\text{ }/\text{ }\tau_{1},\left(
Y_{1},\Upsilon_{1}\right)  \right)  =\exp\left(  -\int_{\tau_{1}}^{\tau_{1}%
+t}\lambda\left(  y_{\Upsilon_{1}}\left(  s;\tau_{1},Y_{1},\alpha_{2}\right)
,\Upsilon_{1},\alpha_{2}\left(  s-\tau_{1};Y_{1},\Upsilon_{1}\right)  \right)
ds\right)  ,
\]
where $\alpha_{2}\in\mathbb{L}^{0}\left(
\mathbb{R}
_{+}\times%
\mathbb{R}
^{m}\times E;A\right)  $. We set
\[
\left(  X_{t}^{x_{0},\gamma_{0},\alpha},\Gamma_{t}^{x_{0},\gamma_{0},\alpha
}\right)  =\left(  y_{\Upsilon_{1}}\left(  t;\tau_{1},Y_{1},\alpha_{2}\right)
,\Upsilon_{1}\right)  ,\text{ if }t\in\left[  \tau_{1},\tau_{2}\right)  .
\]
The post-jump location $\left(  Y_{2},\Upsilon_{2}\right)  $ satisfies%
\[
\mathbb{P}\left(  \left(  Y_{2},\Upsilon_{2}\right)  \in\mathcal{Y\times
E}\text{ }/\text{ }\tau_{2},\tau_{1},Y_{1},\Upsilon_{1}\right)  =\mathbf{1}%
_{y_{\Upsilon_{1}}\left(  \tau_{2};\tau_{1},Y_{1},\alpha_{2}\right)
\in\mathcal{Y}}Q\left(  y_{\Upsilon_{1}}\left(  \tau_{2};\tau_{1},Y_{1}%
,\alpha_{2}\right)  ,\Upsilon_{1},\mathcal{E},\alpha_{2}\left(  \tau_{2}%
-\tau_{1};Y_{1},\Upsilon_{1}\right)  \right)  ,
\]
for all Borel sets $\mathcal{Y}\subset%
\mathbb{R}
^{m}$ and $\mathcal{E\subset}E.$ (Of course, the set $E$ is endowed with the
discrete topology.) And so on.

Throughout the paper, unless stated otherwise, we assume the following:

(A1) The functions $f_{\gamma}:%
\mathbb{R}
^{m}\times A\longrightarrow%
\mathbb{R}
^{m}$ are uniformly continuous on $%
\mathbb{R}
^{m}\times A$ and there exists a positive real constant $C>0$ such that
\begin{equation}
\left\langle f_{\gamma}\left(  x,a\right)  -f_{\gamma}\left(  y,a\right)
,x-y\right\rangle \leq C\left\vert x-y\right\vert ^{2},\text{ and }\left\vert
f_{\gamma}\left(  x,a\right)  \right\vert \leq C, \tag{A1}\label{A1}%
\end{equation}
for all $x,y\in%
\mathbb{R}
^{m}$ and all $a\in A.$

(A2) The function $\lambda:%
\mathbb{R}
^{m}\times E\times A\longrightarrow%
\mathbb{R}
_{+}$ is uniformly continuous on $%
\mathbb{R}
^{m}\times\left\{  \gamma\right\}  \times A$ and there exists a positive real
constant $C>0$ such that
\begin{equation}
\left\vert \lambda\left(  x,\gamma,a\right)  -\lambda\left(  y,\gamma
,a\right)  \right\vert \leq C\left\vert x-y\right\vert ,\text{ and }%
\lambda\left(  x,\gamma,a\right)  \leq C, \tag{A2}\label{A2}%
\end{equation}
for all $x,y\in%
\mathbb{R}
^{m},$ all $\gamma\in E$ and all $a\in A.$

(A3) The function $Q:%
\mathbb{R}
^{m}\times E^{2}\times A\longrightarrow\left[  0,1\right]  $ is a stochastic
matrix : i.e. $%
{\textstyle\sum\limits_{\gamma^{\prime}\in E}}
Q\left(  x,\gamma,\gamma^{\prime},a\right)  =1,$ for all $\gamma\in E$ and all
$\left(  x,a\right)  \in%
\mathbb{R}
^{m}\times A.$ Moreover, we assume that $Q\left(  x,\gamma,\gamma,a\right)
=0,$ for all $\gamma\in E$ and that there exists some positive real constant
$C>0$ such that
\begin{equation}
\sup_{\substack{a\in A\\\gamma,\gamma^{\prime}\in E}}\left\vert Q\left(
x,\gamma,\gamma^{\prime},a\right)  -Q\left(  y,\gamma,\gamma^{\prime
},a\right)  \right\vert \leq C\left\vert x-y\right\vert . \tag{A3}\label{A3}%
\end{equation}

(A4) The cost functions $l_{\gamma}:%
\mathbb{R}
^{m}\times A\longrightarrow%
\mathbb{R}
$ are uniformly continuous on $%
\mathbb{R}
^{m}\times A$ and there exists a positive real constant $C>0$ such that
\begin{equation}
\left\vert l_{\gamma}\left(  x,a\right)  -l_{\gamma}\left(  y,a\right)
\right\vert \leq C\left\vert x-y\right\vert ,\text{ and }\left\vert l_{\gamma
}\left(  x,a\right)  \right\vert \leq C, \tag{A4}\label{A4}%
\end{equation}
for all $x,y\in%
\mathbb{R}
^{m}$ and all $a\in A.$

\begin{remark}
(i) The assumptions (\textbf{A1-A4}) are quite standard when dealing with
viscosity theory in PDMP. They appear under this form in \cite{Soner86_2} and
are needed to infer the uniform continuity of the value function.

(ii) We have chosen this presentation in order to emphasize the continuity of
the $X$ component (continuous switch). Readers who are familiar with the
construction in \cite{Soner86_2}, may skip this subsection and just think of a
characteristic triple
\begin{align*}
\overline{f} &  :%
\mathbb{R}
^{m+d}\times A\longrightarrow%
\mathbb{R}
^{m+d},\text{ }\overline{f}\left(  \left(  x,\gamma\right)  ,a\right)
=\left(  f_{\gamma}\left(  x,a\right)  ,0_{%
\mathbb{R}
^{d}}\right)  ,\text{ }\overline{\lambda}=\lambda\\
\overline{Q} &  :%
\mathbb{R}
^{m+d}\times A\rightarrow\mathcal{P}\left(
\mathbb{R}
^{m+d}\right)  ,\text{ }\overline{Q}\left(  \left(  x,\gamma\right)
,a,dy,d\theta\right)  =\delta_{x}\left(  dy\right)  Q\left(  x,\gamma
,d\theta\right)  .
\end{align*}
Here, $\mathcal{P}\left(
\mathbb{R}
^{m+d}\right)  $ stands for the family of probability measures on $%
\mathbb{R}
^{m+d}.$
\end{remark}

\section{A traffic problem}

We consider a traffic problem on a network given by \ :

- a family of vertices $\left(  e_{j}\right)  _{j=1,2,...,N},$ for some $N\in%
\mathbb{N}
^{\ast}\smallsetminus\left\{  1\right\}  ,$

- a central intersection denoted by $O.$%

\[%
{\parbox[b]{2.1811in}{\begin{center}
\includegraphics[
height=1.9052in,
width=2.1811in
]%
{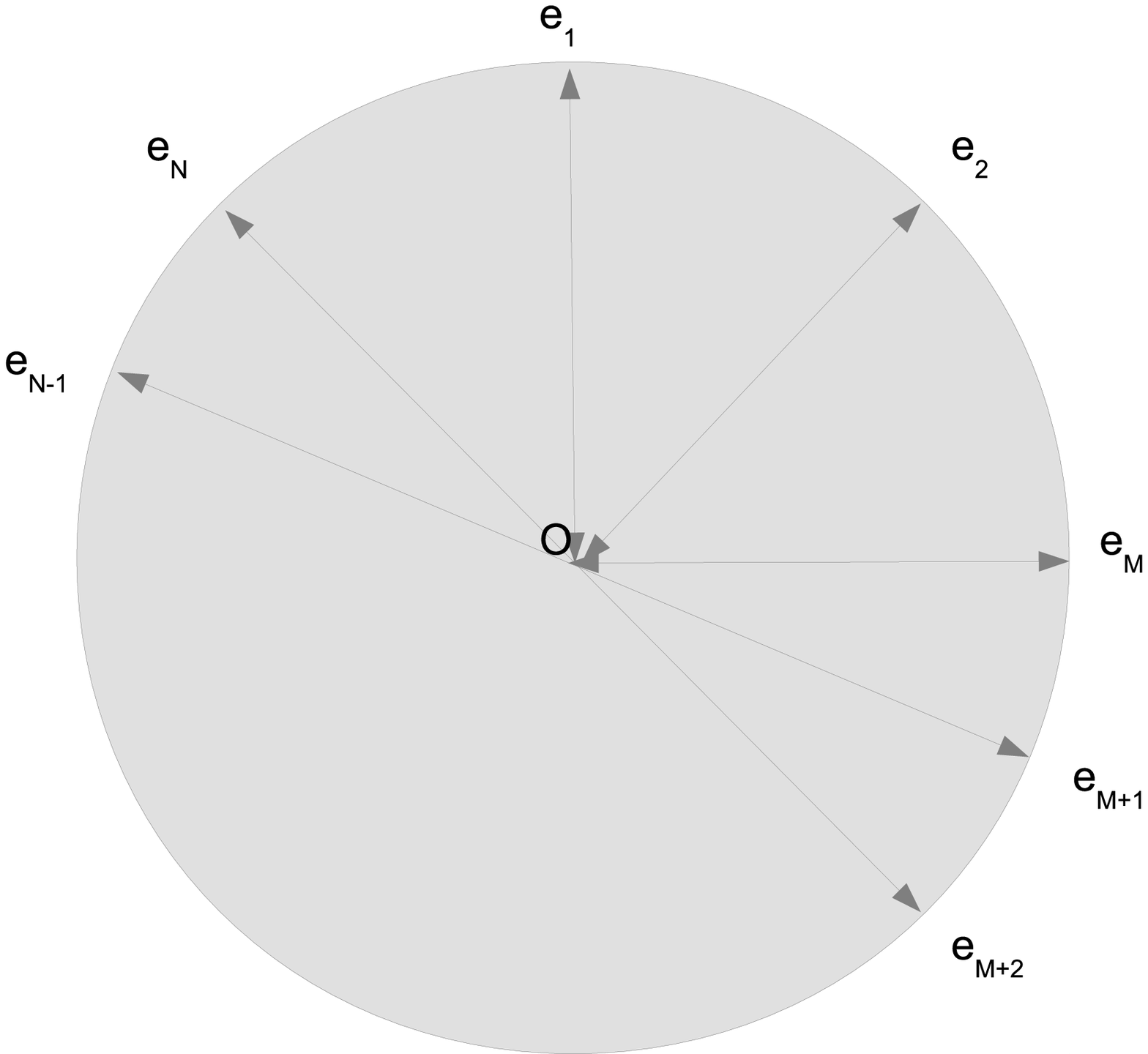}%
\\
Fig 1. Original intersection
\end{center}}}%
\]

We let $J_{j}:=\left(  0,1\right)  e_{j},$ for all $j=1,2,...,N,$
$\mathcal{G}:\mathcal{=}\underset{j=1,2,...,N}{\mathcal{\cup}}\left[
0,1\right)  e_{j}$ and $\overline{\mathcal{G}}:\mathcal{=}%
\underset{j=1,2,...,N}{\mathcal{\cup}}\left[  0,1\right]  e_{j}$.

Our point of view is the one of a traffic regulator who observes the generic
traffic and has the possibility to intervene in the regulation by imposing
speed limits via some (external control). Given an initial point
$x\in\overline{\mathcal{G}}$, the generic vehicle will move (in a continuous
trajectory $X_{t}$) on $\overline{\mathcal{G}}$. At the same time as the
actual traffic, the regulator observes the quality of the road ($\Gamma_{t}$)
and distinguishes between roads which are functional (active) and those which
need repairing (inactive). For functional roads, speeding up the traffic at
the intersection in both directions is possible. In the inactive case, the
road needs repairing and the vehicles that have just entered the road are
directed to the junction. We emphasize that we have a simplified toy model in
which the $x$ component stands for the position of a generic vehicle (in
opposition with the usual density component in traffic models).

This leads to controlled switch PDMP dynamics $\left(  X_{t}^{x,\gamma,\alpha
},\Gamma_{t}^{x,\gamma,\alpha}\right)  $ governed by the speed of the vehicle
$f,$ a jump parameter $\lambda$ depending on both the traffic and the quality
of the road $\lambda$ and a postjump transition $Q$ specifying functionality
of the network. We denote by $E$ the family of all possible functionality
variables (e.g. $\left\{  0,1\right\}  ^{N}$) and introduce, for all
$j=1,2,...,N$ a partition of $E=E_{j}^{active}\cup E_{j}^{inactive}.$

Given an initial couple describing the position and configuration $\left(
x,\gamma\right)  \in\overline{\mathcal{G}}\times E,$ we introduce the set of
feasible (network-constrained) controls for the deterministic framework by
setting%
\[
\mathcal{A}_{\gamma,x}:=\left\{  \alpha:%
\mathbb{R}
_{+}\longrightarrow A:\alpha\text{ is Borel measurable, }y_{\gamma}\left(
t;x,\alpha\right)  \in\overline{\mathcal{G}},\text{ for all }t\geq0,\text{
}\right\}  ,
\]
for all $\left(  x,\gamma\right)  \in\overline{\mathcal{G}}\times E.$ In
general, without further assumptions, these sets might be empty. We will
specify hereafter some controllability conditions that guarantee consistence
of these sets. We also introduce the set of constant, locally-admissible
controls for the deterministic problem by setting
\[
A_{\gamma,x}=\left\{  a\in A:y_{\gamma}\left(  t;x,a\right)  \in
\overline{\mathcal{G}},\text{ for some }\theta>0\text{ and all }t\in\left[
0,\theta\right]  \right\}  ,
\]
for all $\left(  x,\gamma\right)  \in\overline{\mathcal{G}}\times E.$

Unless stated otherwise, throughout the paper, we will use the following assumptions.

\textbf{(Aa)} There exist nonempty subsets $A^{\gamma,j}\subset A$ such that%
\begin{align*}
A_{\gamma,x}  &  =A^{\gamma,j},\text{ if }x\in J_{j},\\
A_{\gamma,O}  &  =\underset{j=1,2,...,N}{\cup}\left\{  a\in A^{\gamma
,j}:f\left(  O,a\right)  \in%
\mathbb{R}
_{+}e_{j}\right\}  ,\\
A_{\gamma,e_{j}}  &  =\left\{  a\in A^{\gamma,j}:\left\langle f_{\gamma
}\left(  e_{j},a\right)  ,e_{j}\right\rangle \leq0\right\}  \neq
\emptyset\text{,}%
\end{align*}
for all $\gamma\in E$ and all $j=1,2,...,N.$ Moreover, we assume that, for
every $\gamma\in E$ and every $j=1,2,...,N,$ either $A_{\gamma,e_{j}%
}=A^{\gamma,j}$ or, otherwise, there exists some $\beta>0$ and some
$a_{\gamma,j}\in A^{\gamma,j}$ satisfying%
\[
\left\langle f_{\gamma}\left(  e_{j},a_{\gamma,j}\right)  ,e_{j}\right\rangle
<-\beta.
\]

\textbf{(Ab)}

\textit{The active case :}

For all $\gamma\in E_{j}^{active},$ there exist some $\beta>0$ and some
$a_{\gamma,j}^{+},a_{\gamma,j}^{-}\in A^{\gamma,j}$ such that
\[
\left\langle f_{\gamma}\left(  O,a_{\gamma,j}^{+}\right)  ,e_{j}\right\rangle
>\beta\text{ and }\left\langle f_{\gamma}\left(  O,a_{\gamma,j}^{-}\right)
,e_{j}\right\rangle <-\beta.
\]

\textit{The inactive case :}

For $\gamma\in E_{j}^{inactive},$ there exist some $\beta>0,$ $1>\eta>0$,
$\kappa\in\left[  0,1\right)  $ and $a_{\gamma,j}^{-},a_{\gamma,j}^{0}\in
A^{\gamma,j}$ such that
\[
\left\langle f_{\gamma}\left(  x,a_{\gamma,j}^{-}\right)  ,e_{j}\right\rangle
\leq-\beta\left\langle x,e_{j}\right\rangle ^{\kappa},
\]
for all $x\in J_{j},$ $\left\vert x\right\vert \leq\eta$ and $f_{\gamma
}\left(  O,a_{\gamma,j}^{0}\right)  =0.$ Moreover,
\[
\left\langle f_{\gamma}\left(  x,a\right)  ,e_{j}\right\rangle \leq0,
\]
for all $a\in A^{\gamma,j}$ and all $x\in J_{j},$ $\left\vert x\right\vert
\leq\eta.$

\textbf{(Ac) }Whenever $\gamma\in E_{j}^{inactive},$ $l_{\gamma}\left(
O,a\right)  =l_{\gamma}\left(  O\right)  ,$ for all $a\in A^{\gamma,j}.$

\begin{remark}
\label{RemarkAss}(i) The condition (\textbf{Ab}) states that if the road is
functional (active), then one has a behavior similar to the one introduced in
\cite{AchdouCamilliCutri2013} (speeding up the traffic at the intersection in
both directions is possible).

If the road is inactive, then, again according to (\textbf{Ab}), for the cars
that have "just" entered the road, the only possibility is to move back into
the intersection (the road needs clearing up for repairing). A measure
($a_{\gamma,j}^{-}$) is possible to get them off this inactive road within a
controlled time and, eventually, they are allowed to stay in $O$ (due to the
control $a_{\gamma,j}^{0}$) until the road is repaired.

The condition (\textbf{Ac}) is intended for technical reasons. It can be
interpreted as : if the road is inactive, the presence of vehicles at the
entrance of the road prevents the authority to intervene and repair the road
and thus involves a certain cost. For vehicles that intend to get to $e_{j},$
there is a global "waiting" cost at junction. However, if $\left\{  a\in
A^{\gamma,j}:f\left(  O,a\right)  \in%
\mathbb{R}
_{+}e_{j}\right\}  =A^{\gamma,j}$, then (\textbf{Ac}) is no longer necessary.

(ii) Under the assumption (\textbf{Aa}), if $A_{\gamma,e_{j}}\neq A^{\gamma
,j},$ then there exists $\frac{1}{2}>\eta>0$ such that
\[
\left\langle f_{\gamma}\left(  x,a_{\gamma,j}\right)  ,e_{j}\right\rangle
<-\beta,
\]
whenever $\left\vert x-e_{j}\right\vert \leq\eta.$ Similarly, under the
assumption (\textbf{Ab}), for every $\gamma\in E_{j}^{active}$ and some
$\eta>0,$
\[
\left\langle f_{\gamma}\left(  x,a_{\gamma,j}^{-}\right)  ,e_{j}\right\rangle
<-\beta,\text{ }\left\langle f_{\gamma}\left(  x,a_{\gamma,j}^{+}\right)
,e_{j}\right\rangle >\beta,
\]
whenever $\left\vert x\right\vert \leq\eta.$
\end{remark}

Our assumptions guarantee the following.

\begin{proposition}
\label{Prop_nonempty}Under the assumptions (\textbf{Aa}) and (\textbf{Ab}),
the set $\mathcal{A}_{\gamma,x}$ is nonempty for all $\left(  x,\gamma\right)
\in\overline{\mathcal{G}}\times E.$
\end{proposition}

\begin{proof}
If $\gamma\in E_{j}^{active}$ and $x\in\left[  0,0.5\right)  e_{j},$ we
define
\[
t_{x,\gamma,e_{j}}^{+}:=\inf\left\{  t>0:y_{\gamma}\left(  t;x,a_{\gamma
,j}^{+}\right)  =e_{j}\right\}  ,
\]
If $x\in\left[  0.5,1\right]  e_{j},$ we let
\[
t_{x,\gamma,O}^{-}:=\inf\left\{  t>0:y_{\gamma}\left(  t;x,a\right)
=O\right\}  ,
\]
where $a$ is any point of $A_{\gamma,e_{j}}$. One notices that $t_{x,\gamma
,e_{j}}^{+}\geq\frac{0.5}{\max\left(  \left\vert f\right\vert _{0},1\right)
}$ and $t_{x,\gamma,O}^{-}\geq\frac{0.5}{\max\left(  \left\vert f\right\vert
_{0},1\right)  },$ where $\left\vert f\right\vert _{0}=\underset{\gamma\in
E,\text{ }x\in\overline{\mathcal{G}},\text{ }a\in A}{\max}\left\vert
f_{\gamma}\left(  x,a\right)  \right\vert $. For $x\in\left[  0,0.5\right)
e_{j},$ we set
\[
\alpha_{x,\gamma}^{0}\left(  t\right)  :=\left\{
\begin{array}
[c]{c}%
a_{\gamma,j}^{+},\text{ if }t\in\left[  0,t_{x,\gamma,e_{j}}^{+}\right)
\cup\left[  t_{x,\gamma,e_{j}}^{+}+t_{e_{j},\gamma,O}^{-},t_{x,\gamma,e_{j}%
}^{+}+t_{e_{j},\gamma,O}^{-}+t_{O,\gamma,e_{j}}^{+}\right)  \cup...,\\
a,\text{ otherwise.}%
\end{array}
\right.
\]
The estimates on $t^{+,-}$ imply that $\alpha_{x,\gamma}^{0}$ is defined on $%
\mathbb{R}
_{+}.$ Moreover, it is clear that $\alpha_{x,\gamma}^{0}\in A_{\gamma,x}.$
Similar construction holds true for $x\in\left[  0.5,1\right]  e_{j}.$ If
$\gamma\in E_{j}^{inactive},$ one gets similar results by replacing
$a_{\gamma,j}^{+}$ with $a_{\gamma,j}^{0}$. (In fact, in this case, if
$t_{x,\gamma,O}^{-}$ is finite, then the solution stays at $O$ after the time
$t_{x,\gamma,O}^{-}$). This concludes the proof of our assertion.
\end{proof}

We introduce the set $\mathcal{A}_{ad}$ given by%
\begin{equation}
\mathcal{A}_{ad}:=\left\{
\begin{array}
[c]{c}%
\alpha:%
\mathbb{R}
_{+}\times\overline{\mathcal{G}}\times E\longrightarrow A:\alpha\text{ is
Borel measurable,}\\
X_{t}^{x_{0},\gamma_{0},\alpha}\in\overline{\mathcal{G}},\text{ for all }%
t\geq0,\mathbb{P-}a.s.,\text{ for all }\left(  x_{0},\gamma_{0}\right)
\in\overline{\mathcal{G}}\times E\text{ }%
\end{array}
\right\}  . \label{Aad}%
\end{equation}
Here, $X_{t}^{x_{0},\gamma_{0},\alpha}$ is the continuous component of our
PDMP constructed as in Section \ref{constr+ass} by using $\alpha_{i}=\alpha,$
for all $i\geq1$.

\begin{remark}
\label{RemarkAad}(a) Under the assumptions \textbf{(Aa, Ab) }it is clear that
$\mathcal{A}_{ad}$ is nonempty. In fact, it suffices to note that all the
times $t^{+},t^{-}$ in the previous proposition are measurable functions of
$\left(  x,\gamma\right)  .$

(b) The set $\mathcal{A}_{\gamma,x}$ can be seen as a subset of $\mathcal{A}%
_{ad}$ by choosing some $\overline{\alpha}_{0}\in\mathcal{A}_{ad}$ and setting%
\[
\overline{\alpha}\left(  t,y,\eta\right)  :=\left\{
\begin{array}
[c]{c}%
\alpha\left(  t\right)  ,\text{ if }\left(  y,\eta\right)  =\left(
x,\gamma\right)  ,\\
\overline{\alpha}_{0}\left(  t;y,\eta\right)  ,\text{ otherwise,}%
\end{array}
\right.
\]
for all $\alpha\in\mathcal{A}_{\gamma,x}.$
\end{remark}

\begin{example}
Let us exhibit a simple example for which the previous assumptions
(particularly (A1), (Aa-Ab)) are satisfied. We consider $N=3$ and
$e_{1}:=\left(
\begin{array}
[c]{c}%
0\\
1
\end{array}
\right)  ,$ $e_{2}:=\left(
\begin{array}
[c]{c}%
1\\
0
\end{array}
\right)  =:-e_{3},$ $A:=\left[  -1,1\right]  e_{1}\cup\left[  -1,1\right]
e_{2},$ $E:=\left\{
\begin{array}
[c]{c}%
\left(  0,0,0\right)  ,\left(  0,1,1\right)  ,\\
\left(  1,0,0\right)  ,\left(  1,1,1\right)
\end{array}
\right\}  \subset\left\{  0,1\right\}  ^{3},$%
\[
f_{\gamma}\left(  x,a\right)  :=\gamma_{1}\left\langle a,e_{1}\right\rangle
e_{1}+\gamma_{2}\left\langle a,e_{2}\right\rangle e_{2}-\left\vert
a\right\vert \left[
\begin{array}
[c]{c}%
\left(  1-\gamma_{1}\right)  \left\langle x,e_{1}\right\rangle ^{\frac{1}{2}%
}e_{1}+\left(  1-\gamma_{2}\right)  \left(  \left\langle x,e_{2}\right\rangle
^{+}\right)  ^{\frac{1}{2}}e_{2}\\
+\left(  1-\gamma_{2}\right)  \left(  \left\langle x,e_{3}\right\rangle
^{+}\right)  ^{\frac{1}{2}}e_{3}%
\end{array}
\right]  ,\text{ }%
\]
for $x\in%
\mathbb{R}
\times%
\mathbb{R}
_{+},$ $\gamma=\left(  \gamma_{1},\gamma_{2},\gamma_{3}\right)  \in E.$ Here,
$z^{+}=\max\left(  z,0\right)  ,$ for $z\in%
\mathbb{R}
$. Then $E_{1}^{inactive}:=\left\{  \gamma\in E:\gamma_{1}=0\right\}  $ and
$E_{2}^{inactive}=E_{3}^{inactive}:=\left\{  \gamma\in E:\gamma_{2}=0\right\}
.$ The reader is invited to note that $f_{\gamma}$ is Lipschitz-continuous for
active configurations. Also, we wish to note that, for this particular case,
whenever $J_{2}$ is inactive (i.e. $\gamma\in E_{2}^{inactive}$)$,$
$f_{\gamma}\left(  re_{2},a\right)  =-f_{\gamma}\left(  -re_{2},a\right)  $,
for all $r\in%
\mathbb{R}
.$ The intersection acts as a mirror in the inactive case.

The cost $l$ can be chosen increasing with the speed, very high as one reaches
the intersection and null at the destination vertex. Moreover, it can be
chosen decreasing with respect to the number of available/ active roads

$l_{\gamma}\left(  \left(
\begin{array}
[c]{c}%
x_{1}\\
0
\end{array}
\right)  ,a\right)  :=l_{0}+\frac{1}{\gamma_{1}+\gamma_{2}+1}\left(
1-\left\vert x_{1}\right\vert \right)  ^{2}+\left\vert a\right\vert \left(
\left\vert x_{1}\right\vert -\left\vert x_{1}\right\vert ^{2}\right)  $ and
symmetrically for $\left(
\begin{array}
[c]{c}%
0\\
x_{2}%
\end{array}
\right)  $. Here, $l_{0}>0$ is some minimal cost.

The rate $\lambda$ can be chosen in a similar way as a propensity function :
we define $\widetilde{\lambda}_{\gamma}\left(  x,a\right)  :=\lambda
_{0}l_{\gamma}\left(  x,a\right)  $ for some $\lambda_{0}>0,$ then
$\lambda_{\gamma}\left(  x,a\right)  :=\underset{\gamma^{\prime}\in
E\smallsetminus\left\{  \gamma\right\}  }{\sum}\widetilde{\lambda}%
_{\gamma^{\prime}}\left(  x,a\right)  .$ The jump measure $Q$ can be chosen
proportional to the relative contribution to the propensity function
\[
Q\left(  x,\gamma,\gamma^{\prime},a\right)  :=\left\{
\begin{array}
[c]{c}%
\frac{\widetilde{\lambda}_{\gamma^{\prime}}\left(  x,a\right)  }%
{\lambda_{\gamma}\left(  x,a\right)  },\text{ if }\gamma^{\prime}\in
E\smallsetminus\left\{  \gamma\right\}  ,\\
0,\text{ if }\gamma^{\prime}=\gamma.
\end{array}
\right.
\]

\end{example}

\section{The dynamic programming principle and the regularity of the value
function(s)\label{Section_DPP_reg}}

The aim of the traffic regulator will be to minimize the expectation of the
(infinite horizon, discounted) operating cost $l$ satisfying (for the time
being and unless stated otherwise), the assumption \textbf{(A4)}%
\[
\inf_{\alpha}\mathbb{E}\left[  \int_{0}^{\infty}e^{-\delta t}l_{\Gamma
_{t}^{x,\gamma,\alpha}}\left(  X_{t}^{x,\gamma,\alpha},\alpha_{t}\right)
dt\right]  .
\]
The discount parameter $\delta>0$ will be fixed throughout the paper. The set
of control policies (keeping the vehicle on the network) as well as the
meaning of $\alpha_{t}$ will be given later on.

The program of this first part relies on the paper \cite{Soner86_2} : we study
the regularity properties in the deterministic setting via some projection
argument, then define some iterated value functions. Next, we prove the
uniform continuity of these iterates and the dynamic programming principles
(DPP). This leads to a regular limit function satisfying a DPP.

Throughout the paper, if $\phi$ is a bounded real-valued function on some set
$\mathcal{X\times}F,$ where $\mathcal{X\subset%
\mathbb{R}
}^{M}$ and $F$ is compact such that $\phi\left(  \cdot,\varsigma\right)  $ is
Lipschitz-continuous for all $\varsigma\in F,$ we set
\[
\left\vert \phi\right\vert _{0}:=\sup_{\left(  y,\varsigma\right)
\in\mathcal{X\times}F}\left\vert \phi\left(  y,\varsigma\right)  \right\vert
\text{ and }Lip\left(  \phi\right)  :=\sup_{\varsigma\in F}\sup
_{\substack{y,y^{\prime}\in\mathcal{X}\\y\neq y^{\prime}}}\frac{\left\vert
\phi\left(  y,\varsigma\right)  -\phi\left(  y^{\prime},\varsigma\right)
\right\vert }{\left\vert y-y^{\prime}\right\vert }.
\]
Whenever $f$ is not Lipschitz continuous (recall that (A1) is weaker than
Lipschitz-continuity), by abuse of notation, we let
\[
Lip\left(  f\right)  :=\sup_{\left(  \gamma,a\right)  \in E\times A}%
\sup_{\substack{y,y^{\prime}\in\mathcal{%
\mathbb{R}
}^{m}\\y\neq y^{\prime}}}\frac{\left\langle f_{\gamma}\left(  y,a\right)
-f_{\gamma}\left(  y^{\prime},a\right)  ,y-y^{\prime}\right\rangle
}{\left\vert y-y^{\prime}\right\vert ^{2}}.
\]
Of course, whenever the function $f$ \ is only defined and satisfies the
regularity assumptions on $\overline{\mathcal{G}}$, the supremum can be taken
over $j=1,2,...,N$ and $y,y^{\prime}$ which are colinear with $e_{j}$ and
$a\in A^{\gamma,j}$.

\subsection{A projection argument}

Whenever $\varepsilon>0$ is small enough, we let%
\[
t_{\varepsilon}:=-\frac{1}{\delta}\ln\left(  \frac{\varepsilon\delta
}{2\left\vert f\right\vert _{0}}\right)  ,\rho_{\varepsilon}:=\frac{\eta}%
{4}e^{-Lip\left(  f\right)  t_{\varepsilon}}.
\]

We will make extensive use of the following result.

\begin{lemma}
\label{Projection_Lemma}We assume (\textbf{Aa-Ac}) and (\textbf{A1-A4}) to
hold true.

(i) There exists some $C>0$ such that, for every $\varepsilon>0$, every
$\gamma\in E$, $x,y\in J_{1}\cup\left\{  O,e_{1}\right\}  $ satisfying
$\left\vert x-y\right\vert \leq\rho_{\varepsilon}^{\frac{2}{1-\kappa}}$ and
every $\alpha\in\mathcal{A}_{\gamma,x}$, there exists $\mathcal{P}%
_{x,y}\left(  \alpha\right)  \in\mathcal{A}_{\gamma,y}$ such that
\begin{equation}
\left\vert y_{\gamma}\left(  t;y,\mathcal{P}_{x,y}\left(  \alpha\right)
\right)  -y_{\gamma}\left(  t;x,\alpha\right)  \right\vert \leq C\left\vert
x-y\right\vert ^{\frac{1-\kappa}{2}}, \label{Projection_a}%
\end{equation}
and
\begin{equation}
\left\vert \int_{0}^{t}e^{-\delta s}l_{\gamma}\left(  y_{\gamma}\left(
s;y,\mathcal{P}_{x,y}\left(  \alpha\right)  \right)  ,\mathcal{P}_{x,y}\left(
\alpha\right)  \left(  s\right)  \right)  ds-\int_{0}^{t}e^{-\delta
s}l_{\gamma}\left(  y_{\gamma}\left(  s;x,\alpha\right)  ,\alpha\left(
s\right)  \right)  ds\right\vert \leq C\left\vert x-y\right\vert
^{\frac{1-\kappa}{2}}, \label{Projection_b}%
\end{equation}
for all $t\leq t_{\varepsilon}$.

(ii) Moreover, if $\alpha\in\mathcal{A}_{ad},$ then, for every $\varepsilon>0$
and every $\left(  \gamma,x\right)  \in E\times\left(  J_{1}\cup\left\{
O,e_{1}\right\}  \right)  ,$ there exists $\mathcal{P}_{\left(  x,\gamma
\right)  }\left(  \alpha\right)  \in\mathcal{A}_{ad}$ such that the previous
inequalities are satisfied with $\mathcal{P}_{\left(  x,\gamma\right)
}\left(  \alpha\right)  \left(  \cdot;y,\gamma\right)  $ replacing
$\mathcal{P}_{x,y}\left(  \alpha\right)  \left(  \cdot\right)  ,$ for all
$y\in J_{1}\cup\left\{  O,e_{1}\right\}  $ satisfying $\left\vert
x-y\right\vert \leq\rho_{\varepsilon}^{\frac{2}{1-\kappa}}.$
\end{lemma}

\begin{remark}
\label{Uniform_remark_one} (i) A brief look at the proof shows that the
constant $C$ in the previous lemma only depends on $Lip\left(  l\right)
,\left\vert l\right\vert _{0},Lip\left(  f\right)  ,\left\vert f\right\vert
_{0}$ and $\beta$ but not of the actual coefficient $f$ nor of the actual cost
function $l.$

(ii) The assumption (\textbf{Ac}) is only needed if $\left\{  a\in
A^{\gamma,1}:f\left(  O,a\right)  \in%
\mathbb{R}
_{+}e_{1}\right\}  \neq A^{\gamma,1}.$ Otherwise, both the cases (b1) and the
analogous (c3.2) in the proof need not being treated as special cases.
\end{remark}

At this point, we introduce the value function for the deterministic case
($\lambda=0$, or, equivalently, the road functionality $\gamma$ is immutable)
by setting%
\[
v_{0}^{\delta}\left(  x,\gamma\right)  =\inf_{\alpha\in\mathcal{A}_{\gamma,x}%
}\int_{0}^{\infty}e^{-\delta t}l_{\gamma}\left(  y_{\gamma}\left(
t;x,\alpha\right)  ,\alpha\left(  t\right)  \right)  dt,
\]
for all $x\in\overline{\mathcal{G}}$ and all $\gamma\in E.$

As a consequence of our projection lemma, we get the following continuity
result :

\begin{theorem}
\label{Th_Cont_0}The deterministic value functions $v_{0}^{\delta}\left(
\cdot,\gamma\right)  $ are bounded and uniformly continuous on $\overline
{\mathcal{G}}.$
\end{theorem}

\begin{proof}
Since the domain $\overline{\mathcal{G}}$ is compact, it suffices to prove
that $v^{\delta}\left(  \cdot,\gamma\right)  $ is continuous. Let us fix
$x\in\overline{\mathcal{G}}\setminus\left\{  O\right\}  $ and consider
$\varepsilon>0.$ Without loss of generality, we assume that $x\in J_{1}%
\cup\left\{  e_{1}\right\}  .$ Then, there exists some $\alpha\in
\mathcal{A}_{\gamma,x}$ such that
\[
v_{0}^{\delta}\left(  x,\gamma\right)  +\varepsilon\geq\int_{0}%
^{t_{\varepsilon}}e^{-\delta t}l_{\gamma}\left(  y_{\gamma}\left(
t;x,\alpha\right)  ,\alpha\left(  t\right)  \right)  dt-\frac{1}{\delta
}e^{-\delta t_{\varepsilon}}\left\vert l\right\vert _{0}.
\]
Hence, for every $y\in J_{1}\cup\left\{  e_{1},O\right\}  $ such that
$\left\vert x-y\right\vert \leq\rho_{\varepsilon}^{\frac{2}{1-\kappa}},$ using
the previous lemma, there exists $\mathcal{P}_{x,y}\left(  \alpha\right)
\in\mathcal{A}_{\gamma,y}$ such that
\begin{align*}
v_{0}^{\delta}\left(  x,\gamma\right)  +\varepsilon &  \geq\int_{0}%
^{t_{\varepsilon}}e^{-\delta t}l_{\gamma}\left(  y\left(  t;x,\mathcal{P}%
_{x,y}\left(  \alpha\right)  \right)  ,\mathcal{P}_{x,y}\left(  \alpha\right)
\left(  t\right)  \right)  dt-C\rho_{\varepsilon}-\frac{1}{\delta}e^{-\delta
t_{\varepsilon}}\left\vert l\right\vert _{0}\\
&  \geq\int_{0}^{\infty}e^{-\delta t}l_{\gamma}\left(  y\left(
t;x,\mathcal{P}_{x,y}\left(  \alpha\right)  \right)  ,\mathcal{P}_{x,y}\left(
\alpha\right)  \left(  t\right)  \right)  dt-C\rho_{\varepsilon}-\frac
{2}{\delta}e^{-\delta t_{\varepsilon}}\left\vert l\right\vert _{0}\\
&  \geq v_{0}^{\delta}\left(  y,\gamma\right)  -C\rho_{\varepsilon}%
-\frac{\left\vert l\right\vert _{0}}{\left\vert f\right\vert _{0}}\varepsilon.
\end{align*}
The continuity property follows by recalling that $\varepsilon>0$ is arbitrary
and $\underset{\varepsilon\rightarrow0}{\lim}\rho_{\varepsilon}=0.$ In the
case when $x=O,$ the same arguments yield
\[
\lim_{\substack{y\rightarrow O\\y\in J_{j}}}v_{0}^{\delta}\left(
y,\gamma\right)  =v_{0}^{\delta}\left(  O,\gamma\right)  ,
\]
for every $j=$ $1,2,...,N.$ The proof of our theorem is now complete.
\end{proof}

\begin{remark}
\label{Uniform_remark_two} The reader is invited to note that the continuity
modulus of $v_{0}^{\delta}$ depends only on $Lip\left(  l\right)  ,$
$\left\vert l\right\vert _{0},$ $Lip\left(  f\right)  ,$ $\left\vert
f\right\vert _{0}$ and $\beta$ but not of the actual coefficient $f$ nor of
the actual cost function $l.$
\end{remark}

\subsection{Iterated value function}

Following the ideas of \cite{Soner86_2}, we introduce the iterated value
functions $v_{m}^{\delta}$ defined by%
\[
v_{m}^{\delta}\left(  x,\gamma\right)  :=\inf_{\alpha\in\mathcal{A}_{ad}}%
J_{m}\left(  x,\gamma,\alpha\right)  ,
\]
where
\[
J_{m}\left(  x,\gamma,\alpha\right)  :=\mathbb{E}\left[  \int_{0}^{\tau_{1}%
}e^{-\delta t}l_{\gamma}\left(  X_{t}^{x,\gamma,\alpha},\alpha\left(
t;x,\gamma\right)  \right)  dt+e^{-\delta\tau_{1}}v_{m-1}^{\delta}\left(
Y_{1},\Upsilon_{1}\right)  \right]  .
\]
We recall that $\left(  Y_{1},\Upsilon_{1}\right)  $ are the post-jump
locations at the first jump time $\tau_{1}$ depending on $x,\gamma,\alpha,$
(cf. Section \ref{constr+ass}). Hence, we have $\left(  Y_{1},\Upsilon
_{1}\right)  =\left(  X_{\tau_{1}}^{x,\gamma,\alpha},\Gamma_{\tau_{1}%
}^{x,\gamma,\alpha}\right)  $ and $\tau_{1}=\tau_{1}^{x,\gamma,\alpha}.$ The
process is constructed as in section \ref{constr+ass} using $\alpha_{i}%
=\alpha\in\mathcal{A}_{ad},$ for all $i\geq1.$ The reader is invited to note
that a simple recurrence argument yields
\begin{equation}
\left\vert v_{m}^{\delta}\left(  x,\gamma\right)  \right\vert \leq
\frac{\left\vert l\right\vert _{0}}{\delta},\text{ for all }\left(
x,\gamma\right)  \in\overline{\mathcal{G}}\times E.
\end{equation}

Throughout the section, unless stated otherwise, we assume (\textbf{Aa-Ac})
and (\textbf{A1-A4}) to hold true. In order to simplify our presentation, we
assume that $\lambda$ and $Q$ are independent of the control parameter $a.$
The general case follows from similar arguments as those of Lemma
\ref{Projection_Lemma} (the estimates on $l$) if one assumes

\textbf{(Ac') }Whenever $\gamma\in E_{j}^{inactive},$ $Q\left(  O,\gamma
,\gamma^{\prime},a\right)  =Q\left(  O,\gamma,\gamma^{\prime}\right)  $ and
$\lambda\left(  O,\gamma,a\right)  =\lambda\left(  O,\gamma\right)  .$

Again, \textbf{(Ac') }is only needed for those $j$ such that $\gamma\in
E_{j}^{inactive}$ and $\left\{  a\in A^{\gamma,j}:f\left(  O,a\right)  \in%
\mathbb{R}
_{+}e_{j}\right\}  \neq A^{\gamma,j}.$

The same arguments as those employed in \cite[Lemma 3.1]{Soner86_2} yield

\begin{lemma}
\label{LemmaDPPItterate}Let us assume that $v_{m-1}^{\delta}\left(
\cdot,\gamma\right)  $ is continuous on $\overline{\mathcal{G}}$. Then, for
every $T>0,$ one has%
\[
v_{m}^{\delta}\left(  x,\gamma\right)  =\inf_{\alpha\in\mathcal{A}_{ad}%
}\mathbb{E}\left[
\begin{array}
[c]{c}%
\int_{0}^{\tau_{1}\wedge T}e^{-\delta t}l_{\gamma}\left(  y_{\gamma}\left(
t;x,\alpha\right)  ,\alpha\left(  t;x,\gamma\right)  \right)  dt\\
+e^{-\delta\tau_{1}}v_{m-1}^{\delta}\left(  Y_{1},\Upsilon_{1}\right)
\mathbf{1}_{\tau_{1}\leq T}+e^{-\delta T}v_{m}^{\delta}\left(  y_{\gamma
}\left(  T;x,\alpha\right)  ,\gamma\right)  \mathbf{1}_{\tau_{1}>T}%
\end{array}
\right]  ,
\]
for all $\left(  x,\gamma\right)  \in\overline{\mathcal{G}}\times E.$
\end{lemma}

The proof is identical (no changes needed) to the one in \cite[Lemma
3.1]{Soner86_2} and will be omitted from our (already long enough) presentation.

\begin{theorem}
The functions $v_{m}^{\delta}\left(  \cdot,\gamma\right)  $ are uniformly
continuous on $\overline{\mathcal{G}},$ for all $m\geq0$ and uniformly with
respect to $\gamma\in E.$
\end{theorem}

\begin{proof}
We prove our theorem by recurrence over $m.$ For $m=0,$ we invoke theorem
\ref{Th_Cont_0}. Let us assume that $v_{m-1}^{\delta}\left(  \cdot
,\gamma^{\prime}\right)  $ is continuous for all $\gamma^{\prime}\in E$. We
let $\omega_{m-1}$ be the continuity modulus
\[
\omega_{m-1}\left(  r\right)  :=\sup\left\{  \left\vert v_{m-1}^{\delta
}\left(  x,\gamma^{\prime}\right)  -v_{m-1}^{\delta}\left(  y,\gamma^{\prime
}\right)  \right\vert :\left\vert x-y\right\vert \leq r,\text{ }\gamma
^{\prime}\in E\right\}  .
\]
We also introduce
\[
\omega_{m}\left(  \gamma,r\right)  :=\sup\left\{  \left\vert v_{m}^{\delta
}\left(  x,\gamma\right)  -v_{m}^{\delta}\left(  y,\gamma\right)  \right\vert
:\left\vert x-y\right\vert \leq r\right\}  ,
\]
for all $r>0.$ Obviously, $\omega_{m}\left(  r\right)  =\underset{\gamma\in
E}{\sup}$ $\omega_{m}\left(  \gamma,r\right)  $. It is straightforward that
$\omega_{m}\left(  r\right)  \leq2\frac{\left\vert l\right\vert _{0}}{\delta}%
$. Let us fix, for the time being, $\left(  \gamma,x,y\right)  \in
E\times\overline{\mathcal{G}}^{2}$, $\varepsilon>0$ and assume that
$\left\vert x-y\right\vert \leq\rho_{\varepsilon}^{\frac{2}{1-\kappa}}$. Then,
due to the previous lemma, there exists some admissible control process
$\alpha\in\mathcal{A}_{ad}$ such that
\[
v_{m}^{\delta}\left(  x,\gamma\right)  \geq-\varepsilon+\mathbb{E}\left[
\begin{array}
[c]{c}%
\int_{0}^{\tau_{1}\wedge t_{\varepsilon}}e^{-\delta t}l_{\gamma}\left(
y_{\gamma}\left(  t;x,\alpha\right)  ,\alpha\left(  t;x,\gamma\right)
\right)  dt\\
+e^{-\delta\tau_{1}}v_{m-1}^{\delta}\left(  Y_{1},\Upsilon_{1}\right)
\mathbf{1}_{\tau_{1}\leq t_{\varepsilon}}+e^{-\delta t_{\varepsilon}}%
v_{m}^{\delta}\left(  y_{\gamma}\left(  t_{\varepsilon};x,\alpha\right)
,\gamma\right)  \mathbf{1}_{\tau_{1}>t_{\varepsilon}}.
\end{array}
\right]
\]
We denote by $\widetilde{\alpha}$ the admissible control process
$\mathcal{P}_{\left(  x,\gamma\right)  }\left(  \alpha\right)  \in
\mathcal{A}_{ad}$ given by the assertion (ii) in Lemma \ref{Projection_Lemma}.
Moreover, we let $\widetilde{\tau}_{1}$ be the first jump time starting from
$\left(  y,\gamma\right)  $ and using the control $\widetilde{\alpha}.$ We
introduce the following notations :
\begin{align*}
y(t)  &  :=y_{\gamma}\left(  t;x,\alpha\right)  ,\text{ }\alpha(t):=\alpha
\left(  t;x,\gamma\right)  ,\text{ }\lambda\left(  t\right)  :=\lambda\left(
y\left(  t\right)  ,\gamma\right)  ,\text{ }\Lambda\left(  t\right)
=\exp\left(  -\int_{0}^{t}\lambda\left(  s\right)  ds\right)  ,\\
\widetilde{y}(t)  &  :=y_{\gamma}\left(  t;y,\widetilde{\alpha}\right)
,\text{ }\widetilde{\alpha}(t):=\widetilde{\alpha}\left(  t;y,\gamma\right)
,\text{ }\widetilde{\lambda}\left(  t\right)  :=\lambda\left(  \widetilde{y}%
(t),\gamma\right)  ,\text{ }\widetilde{\Lambda}\left(  t\right)  =\exp\left(
-\int_{0}^{t}\widetilde{\lambda}\left(  s\right)  ds\right)  .
\end{align*}
Then%
\[
v_{m}^{\delta}\left(  y,\gamma\right)  \leq%
\begin{array}
[c]{c}%
\mathbb{E}\left[  \int_{0}^{\widetilde{\tau}_{1}\wedge t_{\varepsilon}%
}e^{-\delta t}l_{\gamma}\left(  \widetilde{y}\left(  t\right)
,\widetilde{\alpha}\left(  t\right)  \right)  dt\right] \\
+\mathbb{E}\left[  e^{-\delta\widetilde{\tau}_{1}}v_{m-1}^{\delta}\left(
\widetilde{Y}_{1},\widetilde{\Upsilon}_{1}\right)  \mathbf{1}_{\tau_{1}\leq
t_{\varepsilon}}+e^{-\delta t_{\varepsilon}}v_{m}^{\delta}\left(
\widetilde{y}\left(  t_{\varepsilon}\right)  ,\gamma\right)  \mathbf{1}%
_{\widetilde{\tau}_{1}>t_{\varepsilon}}.\right]
\end{array}
\]
The right-hand member can be written as%
\begin{align}
I_{m}\left(  y,\gamma,\widetilde{\alpha}\right)   &  =\int_{0}^{t_{\varepsilon
}}\widetilde{\lambda}\left(  t\right)  \widetilde{\Lambda}\left(  t\right)
\int_{0}^{t}e^{-\delta s}l_{\gamma}\left(  \widetilde{y}\left(  s\right)
,\widetilde{\alpha}\left(  s\right)  \right)  dsdt\label{I_m}\\
&  +\int_{0}^{t_{\varepsilon}}\widetilde{\lambda}\left(  t\right)
\widetilde{\Lambda}\left(  t\right)  e^{-\delta t}%
{\textstyle\sum\limits_{\gamma^{\prime}\in E\setminus\left\{  \gamma\right\}
}}
v_{m-1}^{\delta}\left(  \widetilde{y}\left(  t\right)  ,\gamma^{\prime
}\right)  Q\left(  \widetilde{y}\left(  t\right)  ,\gamma,\gamma^{\prime
}\right)  dt\nonumber\\
&  +\widetilde{\Lambda}\left(  t_{\varepsilon}\right)  \int_{0}%
^{t_{\varepsilon}}e^{-\delta t}l_{\gamma}\left(  \widetilde{y}\left(
t\right)  ,\widetilde{\alpha}\left(  t\right)  \right)  dt+\widetilde{\Lambda
}\left(  t_{\varepsilon}\right)  e^{-\delta t_{\varepsilon}}v_{m}^{\delta
}\left(  \widetilde{y}\left(  t_{\varepsilon}\right)  ,\gamma\right)
.\nonumber
\end{align}
Then, using the estimates (\ref{Projection_a}) in Lemma \ref{Projection_Lemma}
and recalling that (\textbf{A2}) holds true, one has%
\begin{align}
I_{m}\left(  y,\gamma,\widetilde{\alpha}\right)   &  \leq C\left\vert
x-y\right\vert ^{\frac{1-\kappa}{2}}+\int_{0}^{t_{\varepsilon}}\lambda\left(
t\right)  \Lambda\left(  t\right)  \int_{0}^{t}e^{-\delta s}l_{\gamma}\left(
\widetilde{y}\left(  s\right)  ,\widetilde{\alpha}\left(  s\right)  \right)
dsdt\label{Ineq0}\\
&  +\int_{0}^{t_{\varepsilon}}\lambda\left(  t\right)  \Lambda\left(
t\right)  e^{-\delta t}%
{\textstyle\sum\limits_{\gamma^{\prime}\in E\setminus\left\{  \gamma\right\}
}}
v_{m-1}^{\delta}\left(  \widetilde{y}\left(  t\right)  ,\gamma^{\prime
}\right)  Q\left(  \widetilde{y}\left(  t\right)  ,\gamma,\gamma^{\prime
}\right)  dt\nonumber\\
&  +\Lambda\left(  t_{\varepsilon}\right)  \int_{0}^{t_{\varepsilon}%
}e^{-\delta t}l_{\gamma}\left(  \widetilde{y}\left(  t\right)
,\widetilde{\alpha}\left(  t\right)  \right)  dt+\Lambda\left(  t_{\varepsilon
}\right)  e^{-\delta t_{\varepsilon}}v_{m}^{\delta}\left(  \widetilde{y}%
\left(  t_{\varepsilon}\right)  ,\gamma\right)  ,\nonumber
\end{align}
for some generic constant $C>0$ independent of $\varepsilon$, $\gamma
,y,x,\alpha$ which may change from one line to another. This constant only
depends on the supremum norm and the Lipschitz constants of $\lambda,Q,f$ and
$l.$ Again by (\ref{Projection_a}) and using the assumption (\textbf{A3}), we
get%
\begin{align}
&
{\textstyle\sum\limits_{\gamma^{\prime}\in E\setminus\left\{  \gamma\right\}
}}
v_{m-1}^{\delta}\left(  \widetilde{y}\left(  t\right)  ,\gamma^{\prime
}\right)  Q\left(  \widetilde{y}\left(  t\right)  ,\gamma,\gamma^{\prime
}\right)  -%
{\textstyle\sum\limits_{\gamma^{\prime}\in E\setminus\left\{  \gamma\right\}
}}
v_{m-1}^{\delta}\left(  y\left(  t\right)  ,\gamma^{\prime}\right)  Q\left(
y\left(  t\right)  ,\gamma,\gamma^{\prime}\right) \label{Ineq0.1}\\
&  \leq\omega_{m-1}\left(  C\left\vert x-y\right\vert ^{\frac{1-\kappa}{2}%
}\right)  +%
{\textstyle\sum\limits_{\gamma^{\prime}\in E\setminus\left\{  \gamma\right\}
}}
\left\vert v_{m-1}^{\delta}\left(  y\left(  t\right)  ,\gamma^{\prime}\right)
\right\vert \left\vert Q\left(  \widetilde{y}\left(  t\right)  ,\gamma
,\gamma^{\prime}\right)  -Q\left(  y\left(  t\right)  ,\gamma,\gamma^{\prime
}\right)  \right\vert \nonumber\\
&  \leq\omega_{m-1}\left(  C\left\vert x-y\right\vert ^{\frac{1-\kappa}{2}%
}\right)  +C\left\vert x-y\right\vert ^{\frac{1-\kappa}{2}},\nonumber
\end{align}
for all $t\leq t_{\varepsilon}$. Moreover%
\[
e^{-\delta t_{\varepsilon}}v_{m}^{\delta}\left(  \widetilde{y}\left(
t_{\varepsilon}\right)  ,\gamma^{\prime}\right)  \leq e^{-\delta
t_{\varepsilon}}v_{m}^{\delta}\left(  y\left(  t_{\varepsilon}\right)
,\gamma^{\prime}\right)  +e^{-\delta t_{\varepsilon}}\omega_{m}\left(
C\left\vert x-y\right\vert ^{\frac{1-\kappa}{2}}\right)  .
\]
Returning to (\ref{Ineq0}) and using (\ref{Ineq0.1}) and the previous
relation, we get%
\[
v_{m}^{\delta}\left(  y,\gamma\right)  \leq v_{m}^{\delta}\left(
x,\gamma\right)  +\varepsilon+C\left\vert x-y\right\vert ^{\frac{1-\kappa}{2}%
}+\omega_{m-1}\left(  C\left\vert x-y\right\vert ^{\frac{1-\kappa}{2}}\right)
+e^{-\delta t_{\varepsilon}}\omega_{m}\left(  C\left\vert x-y\right\vert
^{\frac{1-\kappa}{2}}\right)  .
\]
Hence, whenever $\left\vert x-y\right\vert \leq r\leq\rho_{\varepsilon}%
^{\frac{2}{1-\kappa}}$,
\[
\omega_{m}\left(  r,\gamma\right)  \leq\varepsilon+Cr^{\frac{1-\kappa}{2}%
}+\omega_{m-1}\left(  Cr^{\frac{1-\kappa}{2}}\right)  +e^{-\delta
t_{\varepsilon}}\omega_{m}\left(  Cr^{\frac{1-\kappa}{2}}\right)  .
\]
Taking the supremum over $\gamma\in E,$ we can replace $\omega_{m}\left(
r,\gamma\right)  $ with $\omega_{m}\left(  r\right)  .$ We can assume, without
loss of generality, that $C>1$ and the conclusion follows (similar to Lemma
3.3 in \cite{Soner86_2}). Indeed, one considers $r=C^{-\frac{2}{1+\kappa
}\left[  \left(  \frac{1-\kappa}{2}\right)  ^{-n}-1\right]  }$ and iterates in
the previous inequality to get
\begin{align*}
\omega_{m}\left(  C^{-\frac{2}{1+\kappa}\left[  \left(  \frac{1-\kappa}%
{2}\right)  ^{-n}-1\right]  }\right)   &  =\varepsilon\frac{1}{1-e^{-\delta
t_{\varepsilon}}}+e^{-\delta t_{\varepsilon}\left(  n-1\right)  }%
{\textstyle\sum\limits_{k=0}^{n-1}}
\omega_{m-1}\left(  C^{-\frac{2}{1+\kappa}\left[  \left(  \frac{1-\kappa}%
{2}\right)  ^{-k}-1\right]  }\right)  e^{\delta kt_{\varepsilon}}\\
&  +e^{-\delta t_{\varepsilon}\left(  n-1\right)  }%
{\textstyle\sum\limits_{k=0}^{n-1}}
\left(  C^{-\frac{2}{1+\kappa}\left[  \left(  \frac{1-\kappa}{2}\right)
^{-k}-1\right]  }\right)  e^{\delta kt_{\varepsilon}}+2e^{-\delta
t_{\varepsilon}n}\frac{\left\vert l\right\vert _{0}}{\delta},
\end{align*}
for $n$ large enough and recall that $\varepsilon>0$ is arbitrary. Then, by
the recurrence assumption and allowing $n\rightarrow\infty,$ one gets
\[
\omega_{m}\left(  0\right)  \leq\varepsilon\frac{1}{1-e^{-\delta
t_{\varepsilon}}}=\frac{\varepsilon}{1-\frac{\varepsilon\delta}{2\left\vert
f\right\vert _{0}}}.
\]
To complete the proof, one only needs to recall that this inequality holds
true for arbitrary $\varepsilon>0.$
\end{proof}

\begin{remark}
\label{Uniform_remark_three} In fact, all these continuity moduli depend only
the supremum norm and the Lipschitz constants of $\lambda,Q,f$ and $l$ but the
particular choice of the coefficients is irrelevant (see also Remark
\ref{Uniform_remark_two}).
\end{remark}

As a corollary, using the same proof as in the first part of Theorem 3.4 in
\cite{Soner86_2}, we get

\begin{corollary}
Under our assumptions (\textbf{A1-A4, Aa-Ac, Ac')}, the value function
$v^{\delta}\left(  \gamma,\cdot\right)  $ given by%
\[
v^{\delta}\left(  x,\gamma\right)  :=\inf_{\alpha\in\mathcal{A}_{ad}^{%
\mathbb{N}
}}\mathbb{E}\left[
{\textstyle\sum\limits_{n\geq0}}
\int_{\tau_{n}}^{\tau_{n+1}}e^{-\delta t}l_{\Gamma_{\tau_{n}}^{x,\gamma
,\alpha}}\left(  y_{\Gamma_{\tau_{n}}^{\gamma,x,\alpha}}\left(  t;X_{\tau_{n}%
}^{x,\gamma,\alpha},\alpha_{n+1}\right)  ,\alpha_{n+1}\left(  t-\Gamma
_{\tau_{n}}^{x,\gamma,\alpha};X_{\tau_{n}}^{x,\gamma,\alpha},\Gamma_{\tau_{n}%
}^{x,\gamma,\alpha}\right)  \right)  \right]
\]
is bounded and uniformly continuous on $\overline{\mathcal{G}},$ for all
$\gamma\in E.$ Moreover, it satisfies the following Dynamic Programming
Principle :%
\[
v^{\delta}\left(  x,\gamma\right)  =\inf_{\alpha\in\mathcal{A}_{ad}}%
\mathbb{E}\left[
\begin{array}
[c]{c}%
\int_{0}^{T\wedge\tau_{1}}e^{-\delta t}l_{\gamma}\left(  y_{\gamma}\left(
t;x,\alpha\right)  ,\alpha\left(  t;x,\gamma\right)  \right)  dt\\
+e^{-\delta\left(  T\wedge\tau_{1}\right)  }v^{\delta}\left(  y_{\gamma
}\left(  T\wedge\tau_{1};x,\alpha\right)  ,\Gamma_{T\wedge\tau_{1}}%
^{x,\gamma,\alpha}\right)
\end{array}
\right]  ,
\]
for all $T>0$ and all $\left(  \gamma,x\right)  \in E\times\overline
{\mathcal{G}}.$
\end{corollary}

Again, once we have established the ingredients of uniform continuity in the
previous theorem, the proof is identical with the first part of Theorem 3.4 in
\cite{Soner86_2} and will be omitted from our (long enough) paper. One
iterates Lemma \ref{LemmaDPPItterate} to get $v_{m}^{\delta}$ and recalls that
$\lambda$ is bounded and, thus, the jumping times cannot accumulate.

\section{Existence of the viscosity solution\label{Section_Existence}}

At this point, we introduce the following Hamilton-Jacobi integrodifferential
system%
\begin{equation}
\delta v^{\delta}\left(  x,\gamma\right)  +\sup_{a\in A_{\gamma,x}}\left\{
\begin{array}
[c]{c}%
-\left\langle f_{\gamma}\left(  x,a\right)  ,Dv^{\delta}\left(  x,\gamma
\right)  \right\rangle -l_{\gamma}\left(  x,a\right) \\
-\lambda\left(  x,\gamma,a\right)
{\textstyle\sum\limits_{\gamma^{\prime}\in E}}
Q\left(  x,\gamma,\gamma^{\prime},a\right)  \left(  v^{\delta}\left(
x,\gamma^{\prime}\right)  -v^{\delta}\left(  x,\gamma\right)  \right)
\end{array}
\right\}  =0. \label{HJ}%
\end{equation}

\subsection{Relaxing the dynamics}

In addition to the standard assumptions \textbf{(Aa-Ac)}, we will need the following.

\textbf{(Ad) }For every $1\leq j\leq N,$ every $\gamma\in E$ and every $x\in
J_{j},$ there exists $\theta>0$ such that, whenever $\alpha\in\mathcal{A}%
_{ad},$ one has $\alpha\left(  t;x,\gamma\right)  \in A^{\gamma,j}$ for almost
all $t\in\left[  0,\theta\right]  .$

For every $x\mathcal{\in}\overline{\mathcal{G}}$, we let $\mathcal{T}%
_{x}\left(  \overline{\mathcal{G}}\right)  $ denote the set of tangent
directions to $\overline{\mathcal{G}}$ at $x$ : $\mathcal{T}_{x}\left(
\overline{\mathcal{G}}\right)  =%
\mathbb{R}
e_{j}$ if $x\in J_{j},$ $\mathcal{T}_{e_{j}}\left(  \overline{\mathcal{G}%
}\right)  =%
\mathbb{R}
_{-}e_{j}$ and $\mathcal{T}_{O}\left(  \overline{\mathcal{G}}\right)
=\underset{1\leq j\leq N}{\cup}%
\mathbb{R}
_{+}e_{j}$. The set $\mathcal{M}_{+}\left(  E\right)  $ denotes the family of
(positive) measures $\zeta=\left(  \zeta\left(  \gamma\right)  \right)
_{\gamma\in E}$ $\in%
\mathbb{R}
_{+}^{E}.$ The following standard notations will be employed throughout the
section.%
\begin{align*}
\overline{FL}\left(  x,\gamma\right)   &  :=\left\{
\begin{array}
[c]{c}%
\left(  \xi,\zeta,\eta\right)  \in\mathcal{T}_{x}\left(  \overline
{\mathcal{G}}\right)  \times\mathcal{M}_{+}\left(  E\right)  \times%
\mathbb{R}
:\exists\left(  \alpha_{n}\right)  _{n}\subset\mathcal{A}_{ad},\text{ }\left(
t_{n}\right)  _{n}\subset%
\mathbb{R}
_{+},\text{ s.t.}\\
\underset{n\rightarrow\infty}{\lim}t_{n}=0,\underset{n\rightarrow\infty}{\lim
}\frac{1}{t_{n}}\int_{0}^{t_{n}}f_{\gamma}\left(  x,\alpha_{n}\left(
s;x,\gamma\right)  \right)  ds=\xi,\\
\underset{n\rightarrow\infty}{\lim}\frac{1}{t_{n}}\int_{0}^{t_{n}}%
\lambda\left(  x,\gamma,\alpha_{n}\left(  s;x,\gamma\right)  \right)  Q\left(
x,\gamma,\alpha_{n}\left(  s;x,\gamma\right)  \right)  ds=\zeta,\\
\underset{n\rightarrow\infty}{\lim}\frac{1}{t_{n}}\int_{0}^{t_{n}}l_{\gamma
}\left(  x,\alpha_{n}\left(  s;x,\gamma\right)  \right)  ds=\eta
\end{array}
\right\}  ,\\
\overline{F}\left(  x,\gamma\right)   &  :=\left\{
\begin{array}
[c]{c}%
\left(  \xi,\zeta\right)  \in\mathcal{T}_{x}\left(  \mathcal{G}\right)
\times\mathcal{M}_{+}\left(  E\right)  :\exists\left(  \alpha_{n}\right)
_{n}\subset\mathcal{A}_{ad},\text{ }\left(  t_{n}\right)  _{n}\subset%
\mathbb{R}
_{+},\text{ s.t.}\\
\underset{n\rightarrow\infty}{\lim}t_{n}=0,\underset{n\rightarrow\infty}{\lim
}\frac{1}{t_{n}}\int_{0}^{t_{n}}f_{\gamma}\left(  x,\alpha_{n}\left(
s;x,\gamma\right)  \right)  ds=\xi,\\
\underset{n\rightarrow\infty}{\lim}\frac{1}{t_{n}}\int_{0}^{t_{n}}%
\lambda\left(  x,\gamma,\alpha_{n}\left(  s;x,\gamma\right)  \right)  Q\left(
x,\gamma,\alpha_{n}\left(  s;x,\gamma\right)  \right)  ds=\zeta
\end{array}
\right\}  ,\\
\overline{fl}\left(  x,\gamma,a\right)   &  :=\left(  f_{\gamma}\left(
x,a\right)  ,\lambda\left(  x,\gamma,a\right)  Q\left(  x,\gamma,a\right)
,l_{\gamma}\left(  x,a\right)  \right)  .
\end{align*}

\begin{remark}
\label{Rem_relax}(a) The reader is invited to note that, in the previous
notations, $"\left(  \alpha_{n}\right)  _{n}\subset\mathcal{A}_{ad}"$ (resp.
$"\alpha\left(  s;x,\gamma\right)  "$) and can be replaced by $"\left(
\alpha_{n}\right)  _{n}\subset\mathcal{A}_{\gamma,x}"$ (resp. $"\alpha\left(
s\right)  "$, see also the second part of Remark \ref{RemarkAad}).

(b) Also, the assumptions on the coefficients imply that
\[
\underset{n\rightarrow\infty}{\lim}\frac{1}{t_{n}}\int_{0}^{t_{n}}f_{\gamma
}\left(  x,\alpha_{n}\left(  s;x,\gamma\right)  \right)
ds=\underset{n\rightarrow\infty}{\lim}\frac{1}{t_{n}}\int_{0}^{t_{n}}%
f_{\gamma}\left(  y_{\gamma}\left(  s;x,\alpha_{n}\left(  s;x,\gamma\right)
\right)  ,\alpha_{n}\left(  s;x,\gamma\right)  \right)  ds,
\]
and similar assertions hold true in the definition of $\eta$ and $\zeta.$

(c) Finally, we have dropped the dependency on $\lambda Q$ in these terms for
the sake of simplicity. One should have written $\overline{f\left(  \lambda
Q\right)  l}$, etc.
\end{remark}

We begin with the following technical result.

\begin{lemma}
\label{Lemma_interior}We assume (\textbf{Aa-Ad}) and (\textbf{A1-A4}) to hold
true. For every $x\in\mathcal{G\smallsetminus}\left\{  O\right\}  ,$ the
following equality holds true%
\[
\overline{FL}(x,\gamma)=\overline{co}\overline{fl}\left(  x,\gamma\right)
:=\overline{co}\left\{  \overline{fl}\left(  x,\gamma,a\right)  :a\in
A_{\gamma,x}\right\}  .
\]
Moreover, for every $j\leq N,$%
\[
\overline{FL}(e_{j},\gamma)\subset\overline{co}\overline{fl}\left(
e_{j},\gamma\right)  :=\overline{co}\left\{  \overline{fl}\left(  e_{j}%
,\gamma,a\right)  :a\in A^{\gamma,j}\right\}  \cap\left(  \mathcal{%
\mathbb{R}
}_{-}e_{j}\times\mathcal{M}_{+}\left(  E\right)  \times%
\mathbb{R}
\right)  .
\]

\end{lemma}

\begin{proof}
Without loss of generality, we first assume that $x\in J_{1}$. It is clear
that
\[
\overline{FL}(x,\gamma)\subset\overline{co}\left\{  \overline{fl}\left(
x,\gamma,a\right)  :a\in A_{\gamma,x}\right\}  .
\]
Indeed, it suffices to use the assumption \textbf{(Ad) }to get the existence
of some $\theta>0$ such that whenever $\alpha\in\mathcal{A}_{ad},$ one has
$\alpha\left(  t;x,\gamma\right)  \in A^{\gamma,1}$ for almost all
$t\in\left[  0,\theta\right]  .$ Then, for every $\left(  \alpha_{n}\right)
_{n}\subset\mathcal{A}_{ad},$ and every sequence $\left(  t_{n}\right)
_{n}\subset%
\mathbb{R}
_{+}$ such that $t_{n}\leq\theta,$ one has%
\[
\left(
\begin{array}
[c]{c}%
\frac{1}{t_{n}}\int_{0}^{t_{n}}f_{\gamma}\left(  x,\alpha_{n}\left(
s;x,\gamma\right)  \right)  ds\\
\frac{1}{t_{n}}\int_{0}^{t_{n}}\lambda\left(  x,\gamma,\alpha_{n}\left(
s;x,\gamma\right)  \right)  Q\left(  x,\gamma,\alpha_{n}\left(  s;x,\gamma
\right)  \right) \\
\frac{1}{t_{n}}\int_{0}^{t_{n}}l_{\gamma}\left(  x,\alpha_{n}\left(
s;x,\gamma\right)  \right)
\end{array}
\right)  \in\overline{co}\left\{  \overline{fl}\left(  x,\gamma,a\right)
:a\in A_{\gamma,x}\right\}  .
\]

If $x=e_{1},$ then
\[
\frac{1}{t_{n}}\int_{0}^{t_{n}}f_{\gamma}\left(  y_{\gamma}\left(
s;e_{1},\alpha_{n}\right)  ,\alpha_{n}\left(  s;e_{1},\gamma\right)  \right)
ds=\frac{y_{\gamma}\left(  t_{n};e_{1},\alpha_{n}\right)  -e_{1}}{t_{n}}\in%
\mathbb{R}
_{-}e_{1}.
\]
Hence, invoking part (b) of the Remark \ref{Rem_relax}, it follows that
\[
\overline{FL}(e_{1},\gamma)\subset\overline{co}\left\{  \overline{fl}\left(
e_{1},\gamma,a\right)  :a\in A^{\gamma,1}\right\}  \cap\left(  \mathcal{%
\mathbb{R}
}_{-}e_{1}\times\mathcal{M}_{+}\left(  E\right)  \times%
\mathbb{R}
\right)  .
\]
For the converse inclusion, we fix $x\in\mathcal{G\smallsetminus}\left\{
O\right\}  $. One begins by noticing that $\overline{FL}(x,\gamma)$ is closed.
Hence, it suffices to prove that
\[
co\left\{  \overline{fl}\left(  x,\gamma,a\right)  :a\in A^{\gamma,1}\right\}
\subset\overline{FL}(x,\gamma).
\]
We consider $\lambda_{i}\geq0,i\in\left\{  1,..,K\right\}  $ such that $%
{\textstyle\sum\limits_{i=1}^{K}}
\lambda_{i}=1$ and $a_{i}\in A^{\gamma,1},$ pour tout $i\in\left\{
1,..,K\right\}  .$ Since $x\in J_{1},$ whenever $t_{n}<\frac{\min\left(
\left\vert x\right\vert ,\left\vert x-e_{1}\right\vert \right)  }{\max\left(
\left\vert f\right\vert _{0},1\right)  },$ an admissible control $\alpha
\in\mathcal{A}_{\gamma,x}$ is obtained by setting $\alpha_{n}\left(  t\right)
=%
{\textstyle\sum\limits_{i=1}^{K}}
a_{i}1_{\left[  \left(
{\textstyle\sum\limits_{j=1}^{i-1}}
\lambda_{j}\right)  t_{n},\left(
{\textstyle\sum\limits_{j=1}^{i}}
\lambda_{j}\right)  t_{n}\right)  }\left(  t\right)  $ and the conclusion follows.
\end{proof}

The family of admissible test functions will be given as in
\cite{AchdouCamilliCutri2013} by $\varphi\in C_{b}\left(  \overline
{\mathcal{G}}\right)  $ for which $\varphi\mid_{\overline{J_{j}}}\in C_{b}%
^{1}\left(  \overline{J_{j}}\right)  ,$ for all $j=1,2,...,N.$ If $x\in
J_{j},$ we recall that%
\[
D\varphi\left(  x;\xi\right)  :=\lim_{t\rightarrow0}\frac{\varphi\left(
x+t\xi\right)  -\varphi\left(  x\right)  }{t},\text{ for all }\xi\in\mathcal{%
\mathbb{R}
}e_{j}.
\]
We also recall that
\[
D\varphi\left(  e_{j};\xi\right)  :=\lim_{t\rightarrow0+}\frac{\varphi\left(
e_{j}+t\xi\right)  -\varphi\left(  x\right)  }{t},\text{ for all }\xi
\in\mathcal{%
\mathbb{R}
}_{-}e_{j}%
\]
and
\[
D\varphi\left(  O;\xi\right)  :=\lim_{t\rightarrow0+}\frac{\varphi\left(
t\xi\right)  -\varphi\left(  O\right)  }{t},\text{ whenever }\xi\in\mathcal{%
\mathbb{R}
}_{+}e_{j}.
\]
If $\varkappa:\left[  0,1\right]  \longrightarrow\mathcal{G}$ is continuous
and $\left(  t_{n}\right)  _{n}\subset\left(  0,1\right]  $ is such that
$\underset{n\rightarrow\infty}{\lim}t_{n}=0$ and
\[
\lim_{n\rightarrow\infty}\frac{\varkappa\left(  t_{n}\right)  }{t_{n}}=\xi,
\]
we have
\[
D\varphi\left(  O;\xi\right)  :=\lim_{n\rightarrow\infty}\frac{\varphi\left(
\varkappa\left(  t_{n}\right)  \right)  -\varphi\left(  O\right)  }{t_{n}}%
\]
and one notes that this limit does not depend on the choice of $\varkappa.$ To
simplify the notations, we will also write $\left\langle \xi,D\varphi\left(
x\right)  \right\rangle $ instead of $D\varphi\left(  x;\xi\right)  $. One
notices easily that the choice of test functions is equivalent to taking a
family of test functions $\varphi_{j}\in C_{b}^{1}\left(  \overline{J_{j}%
}\right)  $ such that $\varphi_{j}\left(  O\right)  =\varphi_{j^{\prime}%
}\left(  O\right)  ,$ for all $1\leq j,j^{\prime}\leq N.$ For further details
of this family of test functions, the reader is referred to \cite[Subsection
3.1]{AchdouCamilliCutri2013}.

We now introduce the definition of the generalized solution of the system
(\ref{HJ}).

\begin{definition}
\label{Def_generalized_solution}A bounded, upper semicontinuous function $V$
is said to be a generalized viscosity subsolution of (\ref{HJ}) if, for every
$\left(  \gamma_{0},x_{0}\right)  \in E\times\mathcal{G}$ whenever $\varphi\in
C_{b}\left(  \overline{\mathcal{G}}\right)  $ for which $\varphi
\mid_{\overline{J_{j}}}\in C_{b}^{1}\left(  \overline{J_{j}}\right)  ,$ for
all $j=1,2,...,N$ is a test function such that $x_{0}\in Argmax\left(
V\left(  \cdot,\gamma_{0}\right)  -\varphi\left(  \cdot\right)  \right)  ,$
one has
\[
\delta V\left(  x_{0},\gamma_{0}\right)  +\sup_{\left(  \xi,\zeta,\eta\right)
\in\overline{FL}\left(  x_{0},\gamma_{0}\right)  }\left\{
\begin{array}
[c]{c}%
-\left\langle D\varphi\left(  x_{0};\xi\right)  \right\rangle -\eta\\
-%
{\textstyle\sum\limits_{\gamma^{\prime}\in E}}
\zeta\left(  \gamma^{\prime}\right)  \left(  V\left(  x_{0},\gamma^{\prime
}\right)  -V\left(  x_{0},\gamma_{0}\right)  \right)
\end{array}
\right\}  \leq0.
\]

A bounded, lower semicontinuous function $V$ is said to be a generalized
viscosity supersolution of (\ref{HJ}) if, for every $\left(  \gamma_{0}%
,x_{0}\right)  \in E\times\overline{\mathcal{G}}$ whenever $\varphi\in
C_{b}\left(  \overline{\mathcal{G}}\right)  $ for which $\varphi
\mid_{\overline{J_{j}}}\in C_{b}^{1}\left(  \overline{J_{j}}\right)  ,$ for
all $j=1,2,...,N$ is a test function such that $x_{0}\in Arg\min\left(
V\left(  \cdot,\gamma_{0}\right)  -\varphi\left(  \cdot\right)  \right)  ,$
one has
\[
\delta V\left(  x_{0},\gamma_{0}\right)  +\sup_{\left(  \xi,\zeta,\eta\right)
\in\overline{FL}\left(  x_{0},\gamma_{0}\right)  }\left\{
\begin{array}
[c]{c}%
-\left\langle D\varphi\left(  x;\xi\right)  \right\rangle -\eta\\
-%
{\textstyle\sum\limits_{\gamma^{\prime}\in E}}
\zeta\left(  \gamma^{\prime}\right)  \left(  V\left(  x_{0},\gamma^{\prime
}\right)  -V\left(  x_{0},\gamma_{0}\right)  \right)
\end{array}
\right\}  \geq0.
\]

\end{definition}

\subsection{(A) Viscosity solution}

We are now able to state and proof the main result of the section.

\begin{theorem}
\label{Th_existence_solution}We assume (\textbf{Aa-Ad, Ac'}) and
(\textbf{A1-A4}) to hold true. Then, the value function $v^{\delta}$ is a
bounded uniformly continuous generalized solution of \ (\ref{HJ}).
\end{theorem}

\begin{proof}
We begin with the proof of the subsolution condition. Let us fix $\left(
\gamma_{0},x_{0}\right)  \in E\times\left(  \mathcal{G}\smallsetminus\left\{
O\right\}  \right)  $ and consider a regular test function $\varphi$ such that
$x_{0}\in Argmax\left(  v^{\delta}\left(  \cdot,\gamma_{0}\right)
-\varphi\left(  \cdot\right)  \right)  $. Then
\[
\varphi\left(  x_{0}\right)  -\varphi\left(  x\right)  \leq v^{\delta}\left(
x_{0},\gamma_{0}\right)  -v^{\delta}\left(  x,\gamma_{0}\right)  ,
\]
for all $x\in\overline{\mathcal{G}}$. We can assume, without loss of
generality, that $\varphi\left(  x_{0}\right)  =v^{\delta}\left(  x_{0}%
,\gamma_{0}\right)  .$ Let us consider $\left(  \xi,\zeta,\eta\right)
\in\overline{FL}\left(  x_{0},\gamma_{0}\right)  .$ Then, there exist $\left(
\alpha_{n}\right)  _{n}\subset\mathcal{A}_{ad},$ $\left(  t_{n}\right)
_{n}\subset%
\mathbb{R}
_{+},$ s.t.$\underset{n\rightarrow\infty}{\lim}t_{n}=0$ and%
\[
\left\{
\begin{array}
[c]{l}%
\underset{n\rightarrow\infty}{\lim}\frac{1}{t_{n}}\int_{0}^{t_{n}}%
f_{\gamma_{0}}\left(  x_{0},\alpha_{n}\left(  s;x_{0},\gamma_{0}\right)
\right)  ds=\xi,\\
\underset{n\rightarrow\infty}{\lim}\frac{1}{t_{n}}\int_{0}^{t_{n}}%
\lambda\left(  x_{0},\gamma_{0},\alpha_{n}\left(  s;x_{0},\gamma_{0}\right)
\right)  Q\left(  x_{0},\gamma_{0},\alpha_{n}\left(  s;x_{0},\gamma
_{0}\right)  \right)  ds=\zeta,\\
\underset{n\rightarrow\infty}{\lim}\frac{1}{t_{n}}\int_{0}^{t_{n}}%
l_{\gamma_{0}}\left(  x_{0},\alpha_{n}\left(  s;x_{0},\gamma_{0}\right)
\right)  ds=\eta.
\end{array}
\right.
\]
We fix, for the time being, $n\in%
\mathbb{N}
$. We let $\tau_{1}^{n}$ be the first jumping time associated to $\alpha
_{n}\left(  \cdot;x,\gamma\right)  $. Using the dynamic programming principle,
one gets
\begin{align*}
0  &  =v^{\delta}\left(  x_{0},\gamma_{0}\right)  -\varphi\left(
x_{0}\right)  \leq\mathbb{E}\left[
\begin{array}
[c]{c}%
\int_{0}^{t_{n}\wedge\tau_{1}^{n}}e^{-\delta s}l_{\gamma_{0}}\left(
y_{\gamma_{0}}\left(  s;x_{0},\alpha_{n}\right)  ,\alpha_{n}\left(
s;x_{0},\gamma_{0}\right)  \right)  ds\\
+e^{-\delta\left(  t_{n}\wedge\tau_{1}^{n}\right)  }v^{\delta}\left(
y_{\gamma_{0}}\left(  t_{n}\wedge\tau_{1}^{n};x_{0},\alpha_{n}\right)
,\Gamma_{t_{n}\wedge\tau_{1}^{n}}^{x_{0},\gamma_{0},\alpha_{n}}\right)
\end{array}
\right]  -\varphi\left(  x_{0}\right) \\
&  \leq\mathbb{E}\left[
\begin{array}
[c]{c}%
\int_{0}^{t_{n}\wedge\tau_{1}^{n}}e^{-\delta s}l_{\gamma_{0}}\left(
x_{0},\alpha_{n}\left(  s;x_{0},\gamma_{0}\right)  \right)  ds+\int_{0}%
^{t_{n}\wedge\tau_{1}^{n}}e^{-\delta s}Lip\left(  l\right)  \left\vert
f\right\vert _{0}sds\\
+e^{-\delta\tau_{1}^{n}}v^{\delta}\left(  y_{\gamma_{0}}\left(  \tau_{1}%
^{n};x_{0},\alpha_{n}\right)  ,\Gamma_{\tau_{1}^{n}}^{x_{0},\gamma_{0}%
,\alpha_{n}}\right)  \mathbf{1}_{\tau_{1}^{n}<t_{n}}+e^{-\delta t_{n}}%
\varphi\left(  y_{\gamma_{0}}\left(  t_{n};x_{0},\alpha_{n}\right)  \right)
\mathbf{1}_{\tau_{1}^{n}\geq t_{n}}%
\end{array}
\right] \\
&  -\varphi\left(  x_{0}\right) \\
&  \leq\left\vert f\right\vert _{0}Lip\left(  l\right)  t_{n}\mathbb{E}\left[
t_{n}\wedge\tau_{1}^{n}\right]  +\left\vert l\right\vert _{0}\left(  \int%
_{0}^{t_{n}}\left(  1-e^{-\delta s}\right)  ds+t_{n}\mathbb{P}\left(  \tau
_{1}^{n}<t_{n}\right)  \right)  +\delta\left\vert \varphi\right\vert _{0}%
t_{n}\mathbb{P}\left(  \tau_{1}^{n}<t_{n}\right) \\
&  \mathbb{E}\left[  \int_{0}^{t_{n}}l_{\gamma_{0}}\left(  x_{0},\alpha
_{n}\left(  s;x_{0},\gamma_{0}\right)  \right)  ds\right]  +e^{-\delta t_{n}%
}\varphi\left(  y_{\gamma}\left(  t_{n};x_{0},\alpha_{n}\right)  \right)
-\varphi\left(  x_{0}\right) \\
&  +\mathbb{E}\left[  e^{-\delta\tau_{1}^{n}}\left(  v^{\delta}\left(
y_{\gamma_{0}}\left(  \tau_{1}^{n};x_{0},\alpha_{n}\right)  ,\Gamma_{\tau
_{1}^{n}}^{x_{0},\gamma_{0},\alpha_{n}}\right)  -\varphi\left(  y_{\gamma_{0}%
}\left(  t_{n};x_{0},\alpha_{n}\right)  \right)  \right)  \mathbf{1}_{\tau
_{1}^{n}<t_{n}}\right]  .
\end{align*}
We set
\[
\lambda\left(  s\right)  :=\lambda\left(  y_{\gamma_{0}}\left(  s;x_{0}%
,\alpha_{n}\right)  ,\gamma_{0},\alpha_{n}\left(  s;x_{0},\gamma_{0}\right)
\right)  \text{ and }\Lambda\left(  s\right)  :=\exp\left(  -\int_{0}%
^{s}\lambda\left(  r\right)  dr\right)
\]
and one gets%

\begin{align}
0=  &  v^{\delta}\left(  x_{0},\gamma_{0}\right)  -\varphi\left(  x_{0}\right)
\nonumber\\
\leq &  \left\vert f\right\vert _{0}Lip\left(  l\right)  t_{n}\mathbb{E}%
\left[  t_{n}\wedge\tau_{1}^{n}\right]  +Lip\left(  \varphi\right)  \left\vert
f\right\vert _{0}t_{n}\mathbb{P}\left(  \tau_{1}^{n}<t_{n}\right)
+\delta\left\vert \varphi\right\vert _{0}t_{n}\mathbb{P}\left(  \tau_{1}%
^{n}<t_{n}\right) \nonumber\\
&  +\left\vert l\right\vert _{0}\left(  \int_{0}^{t_{n}}\left(  1-e^{-\delta
s}\right)  ds+t_{n}\mathbb{P}\left(  \tau_{1}^{n}<t_{n}\right)  \right)
\nonumber\\
&  +e^{-\delta t_{n}}\left(  \varphi\left(  y_{\gamma}\left(  t_{n}%
;x_{0},\alpha_{n}\right)  \right)  -\varphi\left(  x_{0}\right)  \right)
+\left(  e^{-\delta t_{n}}-1\right)  \varphi\left(  x_{0}\right)  +\int%
_{0}^{t_{n}}l_{\gamma_{0}}\left(  x_{0},\alpha_{n}\left(  s;x_{0},\gamma
_{0}\right)  \right)  ds\nonumber\\
&  +\int_{0}^{t_{n}}e^{-\delta s}\lambda\left(  s\right)  \Lambda\left(
s\right)  \left(
{\textstyle\sum\limits_{\gamma^{\prime}\neq\gamma_{0}}}
Q\left(  y_{\gamma_{0}}\left(  s;x_{0},\alpha_{n}\right)  ,\gamma_{0}%
,\gamma^{\prime},\alpha_{n}\left(  x_{0},\gamma_{0},s\right)  \right)  \left(
v^{\delta}\left(  y_{\gamma_{0}}\left(  s;x_{0},a\right)  ,\gamma^{\prime
}\right)  -\varphi\left(  x_{0}\right)  \right)  \right)  ds.
\label{ineq_subsol}%
\end{align}
The reader is invited to note that
\[
\left\{
\begin{array}
[c]{c}%
\left\vert e^{-\delta s}\lambda\left(  s\right)  \Lambda\left(  s\right)
-\lambda\left(  x_{0},\alpha_{n}\left(  s;x_{0},\gamma_{0}\right)  \right)
\right\vert \leq\left(  \left\vert f\right\vert _{0}Lip\left(  \lambda\right)
+\left\vert \lambda\right\vert _{0}\left(  \delta+\left\vert \lambda
\right\vert _{0}\right)  \right)  t_{n},\\
\left\vert Q\left(  y_{\gamma_{0}}\left(  s;x_{0},\alpha_{n}\right)
,\gamma_{0},\gamma^{\prime},a\right)  -Q\left(  x_{0},\gamma_{0}%
,\gamma^{\prime},a\right)  \right\vert \leq\left\vert f\right\vert
_{0}Lip\left(  Q\right)  t_{n},\\
\left\vert v^{\delta}\left(  y_{\gamma_{0}}\left(  s;x_{0},\alpha_{n}\right)
,\gamma^{\prime}\right)  -v^{\delta}\left(  x_{0},\gamma^{\prime}\right)
\right\vert =\omega^{\delta}\left(  \left\vert f\right\vert _{0}t_{n}\right)
,
\end{array}
\right.
\]
whenever $s\leq t_{n},$ where $\omega^{\delta}$ denotes the continuity modulus
of $v^{\delta}.$ Also,
\[
\frac{y_{\gamma}\left(  t_{n};x_{0},\alpha_{n}\right)  -x_{0}}{t_{n}}%
=\frac{\int_{0}^{t_{n}}f_{\gamma}\left(  y_{\gamma}\left(  s;x_{0},\alpha
_{n}\right)  ,\alpha_{n}\left(  s;x_{0},\gamma_{0}\right)  \right)  ds}{t_{n}%
}=\frac{\int_{0}^{t_{n}}f_{\gamma}\left(  x_{0},\alpha_{n}\left(
s;x_{0},\gamma_{0}\right)  \right)  ds}{t_{n}}+\omega\left(  t_{n}\right)  ,
\]
(where $\underset{\varepsilon\rightarrow0}{\lim}$ $\omega\left(
\varepsilon\right)  =0$). We divide (\ref{ineq_subsol}) by $t_{n}$ and allow
$n\rightarrow\infty$ to get%
\[
0\leq\eta-\delta\varphi\left(  x_{0}\right)  +D\varphi\left(  x_{0}%
;\xi\right)  +%
{\textstyle\sum\limits_{\gamma^{\prime}\neq\gamma}}
\zeta\left(  \gamma^{\prime}\right)  \left(  v^{\delta}\left(  x_{0}%
,\gamma^{\prime}\right)  -v^{\delta}\left(  x_{0},\gamma\right)  \right)  .
\]
The conclusion follows by recalling that $\left(  \xi,\zeta,\eta\right)
\in\overline{FL}\left(  x_{0},\gamma_{0}\right)  $ is arbitrary.

To prove that $v^{\delta}$ is a viscosity supersolution of the associated
Hamilton-Jacobi integrodifferential equation, let us fix, for the time being,
$\varepsilon>0.$ We equally fix $\left(  \gamma_{0},x_{0}\right)  \in
E\times\mathcal{G}$ and consider a test function $\varphi$ such that $x_{0}\in
Argmin\left(  v^{\delta}\left(  \cdot,\gamma_{0}\right)  -\varphi\left(
\cdot\right)  \right)  $. Then
\[
\varphi\left(  x_{0}\right)  -\varphi\left(  x\right)  \geq v^{\delta}\left(
x_{0},\gamma_{0}\right)  -v^{\delta}\left(  x,\gamma_{0}\right)  ,
\]
for all $x\in\overline{\mathcal{G}}$. We can assume, without loss of
generality, that $\varphi\left(  x_{0}\right)  =v^{\delta}\left(  x_{0}%
,\gamma_{0}\right)  $. There exists an admissible control $\alpha
^{\varepsilon}$ such that
\[
v^{\delta}\left(  x_{0},\gamma_{0}\right)  +\varepsilon\geq\mathbb{E}\left[
\begin{array}
[c]{c}%
\int_{0}^{\sqrt{\varepsilon}\wedge\tau_{1}}e^{-\delta s}l_{\gamma_{0}}\left(
y_{\gamma_{0}}\left(  s;x_{0},\alpha^{\varepsilon}\right)  ,\mathcal{\alpha
}^{\varepsilon}\left(  s\right)  \right)  ds\\
+e^{-\delta\left(  \sqrt{\varepsilon}\wedge\tau_{1}\right)  }v^{\delta}\left(
y_{\gamma_{0}}\left(  \sqrt{\varepsilon}\wedge\tau_{1};x_{0},\alpha
^{\varepsilon}\right)  ,\Gamma_{\sqrt{\varepsilon}\wedge\tau_{1}}%
^{x_{0},\gamma_{0},\alpha^{\varepsilon}}\right)
\end{array}
\right]  .
\]
(For notation purposes, we have dropped the dependency of $\gamma_{0},x_{0}$
in $\mathcal{\alpha}^{\varepsilon}$). As in the first part of our proof,
$\tau_{1}$ denotes the first jumping time associated to the admissible control
process $\alpha^{\varepsilon}$. Using similar estimates to the first part, one
gets%
\begin{align*}
0  &  =v^{\delta}\left(  x_{0},\gamma_{0}\right)  -\varphi\left(  x_{0}\right)
\\
&  \geq-\varepsilon-\left\vert f\right\vert _{0}Lip\left(  l\right)
\sqrt{\varepsilon}\mathbb{E}\left[  \sqrt{\varepsilon}\wedge\tau_{1}\right]
-Lip\left(  \varphi\right)  \left\vert f\right\vert _{0}\sqrt{\varepsilon
}\mathbb{P}\left(  \tau_{1}<\sqrt{\varepsilon}\right)  -\delta\left\vert
\varphi\right\vert _{0}\sqrt{\varepsilon}\mathbb{P}\left(  \tau_{1}%
<\sqrt{\varepsilon}\right) \\
&  -\left\vert l\right\vert _{0}\left(  \int_{0}^{\sqrt{\varepsilon}}\left(
1-e^{-\delta s}\right)  ds+\sqrt{\varepsilon}\mathbb{P}\left(  \tau_{1}%
<\sqrt{\varepsilon}\right)  \right) \\
&  +e^{-\delta\sqrt{\varepsilon}}\left(  \varphi\left(  y_{\gamma}\left(
\sqrt{\varepsilon};x_{0},\alpha^{\varepsilon}\right)  \right)  -\varphi\left(
x_{0}\right)  \right)  +\left(  e^{-\delta\sqrt{\varepsilon}}-1\right)
\varphi\left(  x_{0}\right) \\
&  +\int_{0}^{\sqrt{\varepsilon}}e^{-\delta s}\lambda\left(  s\right)
\Lambda\left(  s\right)  \left(
{\textstyle\sum\limits_{\gamma^{\prime}\neq\gamma_{0}}}
Q\left(  y_{\gamma_{0}}\left(  s;x_{0},\alpha^{\varepsilon}\right)
,\gamma_{0},\gamma^{\prime},\mathcal{\alpha}^{\varepsilon}\left(  s\right)
\right)  \left(  v^{\delta}\left(  y_{\gamma_{0}}\left(  s;x_{0}%
,\alpha^{\varepsilon}\right)  ,\gamma^{\prime}\right)  -\varphi\left(
x_{0}\right)  \right)  \right)  ds,
\end{align*}
where $\lambda\left(  s\right)  :=\lambda\left(  y_{\gamma_{0}}\left(
s;x_{0},\alpha^{\varepsilon}\right)  ,\gamma_{0},\mathcal{\alpha}%
^{\varepsilon}\left(  s\right)  \right)  $ and $\Lambda\left(  s\right)
:=\exp\left(  -\int_{0}^{s}\lambda\left(  r\right)  dr\right)  .$ We recall
that $f,$ $\lambda$ and $Q$ are Lipschitz-continuous and bounded and
$v^{\delta}$ is uniformly continuous and bounded. The conclusion follows
similarly to the subsolution case by dividing the inequality by $\sqrt
{\varepsilon},$ recalling the definition of $\overline{FL}\left(  x_{0}%
,\gamma_{0}\right)  $ and allowing $\varepsilon$ (or some subsequence) to go
to $0.$
\end{proof}

\section{Extending the intersection and linearizing the value
function\label{Section_Uniqueness}}

\subsection{Additional directions}

Without loss of generality, we assume that $-e_{j}\notin\overline{\mathcal{G}%
},$ for all $j\leq M\leq N$ and $-e_{j}\in\overline{\mathcal{G}},$ for all
$M<j\leq N.$ We define
\[
e_{j}:=-e_{j-N},E_{j}^{active}:=E_{j-N}^{active},E_{j}^{inactive}%
:=E_{j-N}^{inactive},
\]
whenever $N<j\leq M+N.$ For every $\varepsilon>0,$ we complete $\mathcal{G}$
into $\mathcal{G}^{+,\varepsilon}$ by adding $\left[  0,\varepsilon
e_{j}\right)  $ for $N<j\leq M+N$ and $\left(  1,1+\varepsilon\right)  e_{j}$,
for $j\leq N.$
\[%
{\parbox[b]{2.3194in}{\begin{center}
\includegraphics[
height=2.1473in,
width=2.3194in
]%
{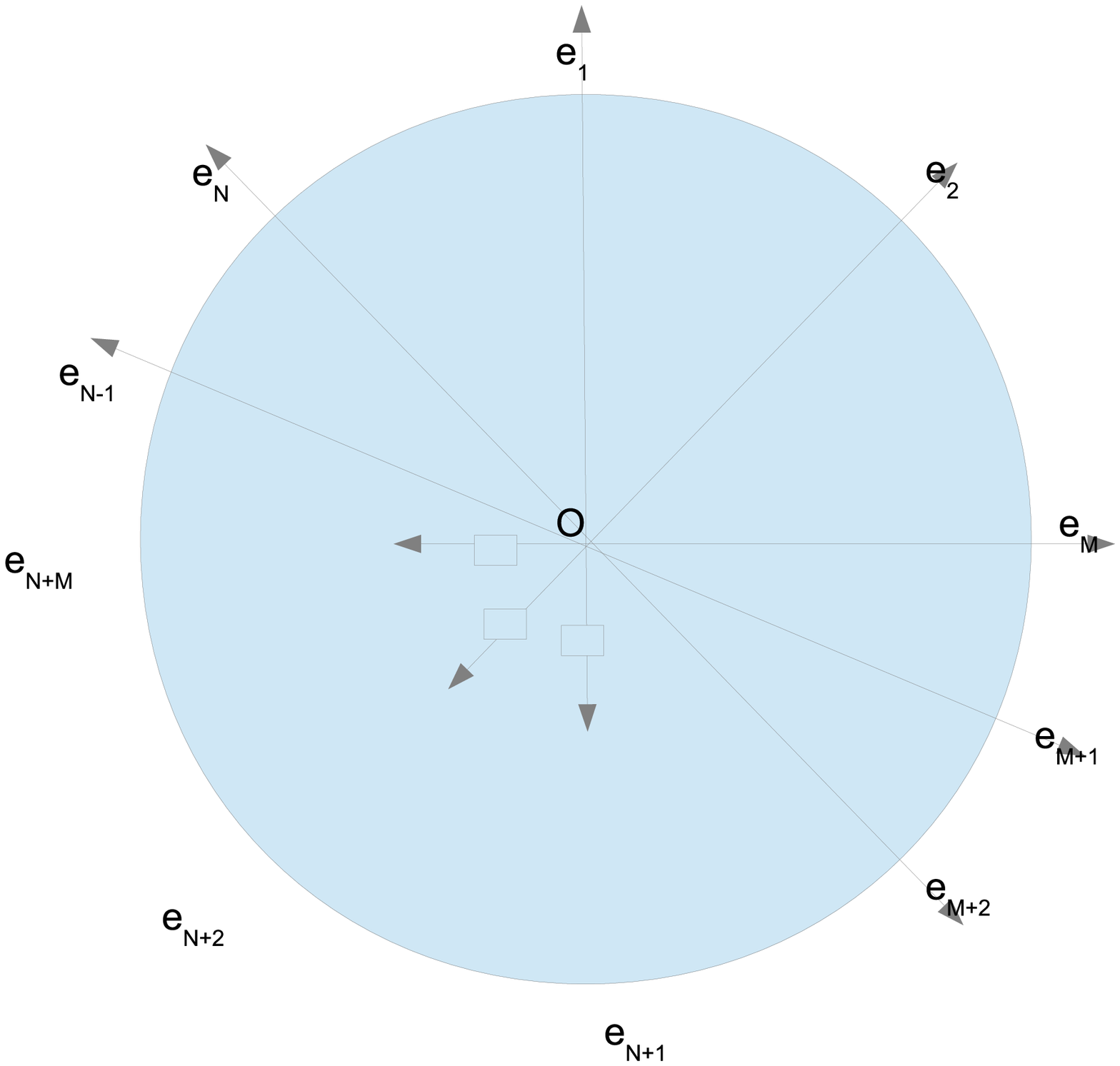}%
\\
Fig 2. The complete intersection
\end{center}}}%
\]

Throughout the remaining of the paper we make the following assumption.

\textbf{(B)} Whenever $M<j,j^{\prime}\leq N$ are such that $e_{j^{\prime}%
}=-e_{j},$ then $A^{\gamma,j}=A^{\gamma,j^{\prime}},$ for all $\gamma\in E.$

\begin{remark}
Roughly speaking, on the roads that cross the intersection (of type $\left(
-1,1\right)  e_{j}$)$,$ the same family of (piecewise constant) controls can
be used both at the entrance and at the exit of the intersection.
\end{remark}

\subsubsection{Inactive roads}

The reader is invited to notice that, if $e_{j},e_{j^{\prime}}=-e_{j}%
\in\overline{\mathcal{G}}$ then, for every $\gamma\in E_{j}^{inactive}\cap
E_{j^{\prime}}^{inactive},$ $f_{\gamma}\left(  O,a\right)  =0,$ for all $a\in
A^{\gamma,j}\cap A^{\gamma,j^{\prime}}.$ This is a simple consequence of the
assumption (\textbf{Ab}) which implies that $\left\langle f_{\gamma}\left(
O,a\right)  ,e_{j}\right\rangle \leq0$ and $\left\langle f_{\gamma}\left(
O,a\right)  ,e_{j^{\prime}}\right\rangle \leq0,$ for all $a\in A^{\gamma
,j}\cap A^{\gamma,j^{\prime}}$. In particular, if (\textbf{B}) holds true,
then $f_{\gamma}\left(  O,a\right)  =0,$ for all $a\in A^{\gamma,j}%
(=A^{\gamma,j^{\prime}})$ whenever $\gamma\in E_{j}^{inactive}\cap
E_{j^{\prime}}^{inactive}.$

Hence, in order to obtain a similar behavior for the completed intersection,
it is natural to strengthen the assumption (\textbf{Ab}). We will assume that,

\textbf{(Ab')} Whenever $\gamma\in E_{j}^{inactive}$ for some $j\leq M,$ then
$f_{\gamma}\left(  O,a\right)  =0,$ for all $a\in A^{\gamma,j}.$

\begin{remark}
This is, of course, less general than the existence of one $a_{\gamma,j}%
^{0}\in A^{\gamma,j}$ guaranteed by (\textbf{Ab}). The assumption states that,
whenever the road $j$ is inactive, a vehicle that needs to go on this road
should wait until it is repaired.
\end{remark}

\subsection{Extending the dynamics}

Unless stated otherwise, we assume the (pseudo-)controllability conditions
(\textbf{Aa, Ab, Ad}), the compatibility at the intersection\textbf{ (Ab',
Ac'), }the regularity of the coefficients and cost functions \textbf{(A1-A4)}
and the compatibility condition (\textbf{B}) to hold true.

We are now able to extend $f$ (and $\lambda,Q)$ $\ $to $\left(
{\textstyle\bigcup\limits_{j=1,2,...,N}}
\mathbb{R}
e_{j}\right)  \times A$ by setting%
\[
f_{\gamma}\left(  x,a\right)  =\left\{
\begin{array}
[c]{l}%
f_{\gamma}\left(  x,a\right)  ,\text{ \ \ \ if }x\in\left(  0,1\right)
e_{j},\text{ }j\leq N,\\
f_{\gamma}\left(  e_{j},a\right)  ,\text{ \ \ \ if }x\in\left[  1,\infty
\right)  e_{j},\text{ }j\leq N,\text{ }\\
-f_{\gamma}\left(  -x,a\right)  ,\text{ if }\gamma\in E_{j}^{inactive},x\in%
\mathbb{R}
_{-}e_{j},\text{ }j\leq M,\\
f_{\gamma}\left(  O,a\right)  ,\text{ \ \ \ if }\gamma\in E_{j}^{active},x\in%
\mathbb{R}
_{-}e_{j},\text{ }j\leq M.
\end{array}
\right.
\]
For the other elements $\left(  \varphi\in\left\{  \lambda,l,Q\right\}
\right)  $, we set
\[
\varphi\left(  x,\gamma,a\right)  =\left\{
\begin{array}
[c]{l}%
\varphi\left(  x,\gamma,a\right)  ,\text{ \ \ \ if }x\in\left(  0,1\right)
e_{j},\text{ }j\leq N,\\
\varphi\left(  e_{j},\gamma,a\right)  ,\text{ \ \ \ if }x\in\left[
1,\infty\right)  e_{j},\text{ }j\leq N,\text{ }\\
\varphi\left(  O,\gamma,a\right)  ,\text{ \ \ otherwise.}%
\end{array}
\right.  .
\]
(by abuse of notation, $l\left(  x,\gamma,a\right)  =l_{\gamma}\left(
x,a\right)  $).

This particular construction for $f$ is needed in order to guarantee that the
assumptions (\textbf{Aa}) and (\textbf{Ab}) hold true for the new system on
$\mathcal{G}^{+,\varepsilon}.$ It basically suggests that in the active case,
the vehicle will continue its road on the extension of the road with the same
speed as in $O$. In the inactive case, the extension of the road is obtained
by looking at the road $j$ using a mirror.

\subsection{Krylov's "shaking the coefficients" method}

We wish to construct a family of regular functions satisfying a suitable
subsolution condition and converging, as $\varepsilon\rightarrow0$ to our
value function. Regularization can be achieved by classical convolution.
However, because of the convolution, the subsolution condition should not only
concern a point $x$ but some neighborhood. This is the reason why, one needs
to introduce a perturbation in the Hamiltonian or, equivalently, in the
coefficients. The method is known as "shaking the coefficients" and has been
introduced, in the framework of Brownian diffusions, in \cite{krylov_00}.

For $r>0,$ we let $B_{r}$ denote the $r$-radius closed ball $B_{r}=\left\{
y\in%
\mathbb{R}
^{m}:\left\vert y\right\vert \leq r\right\}  .$ We set
\[
\left\{
\begin{array}
[c]{c}%
f_{\gamma}^{\rho}\left(  x,a,b\right)  =f_{\gamma}\left(  x+\rho b,a\right)
,\\
\varphi^{\rho}\left(  x,\gamma,a,b\right)  =\varphi\left(  x+\rho
b,\gamma,a\right)  ,\text{ if }\varphi\in\left\{  \lambda,Q,l\right\}  ,
\end{array}
\right.
\]
for all $\left(  x,a,b\right)  \in\underset{j=1,2,...,N}{\cup}\left(  \left[
-\varepsilon,1+\varepsilon\right]  e_{j}\times A\times\left[  -1,1\right]
e_{j}\right)  ,$ and all $\left\vert b\right\vert \leq1.$ Let us fix, for the
time being, $\varepsilon\geq\rho>0$ and consider the control problem on
$\mathcal{G}^{+,\varepsilon}$. We denote by
\[
J_{j}^{\varepsilon,+}:=\left(  0,1+\varepsilon\right)  e_{j},\text{ for all
}j=1,2,...,N,\text{ }J_{j}^{\varepsilon,-}:=\left\{
\begin{array}
[c]{c}%
\left(  -\varepsilon,0\right)  e_{j},\text{ for }j\leq M,\\
\left(  -1-\varepsilon,0\right)  e_{j},\text{ for }M<j\leq N,
\end{array}
\right.  ,J_{j}^{\varepsilon}:=J_{j}^{\varepsilon,+}\cup J_{j}^{\varepsilon
,-}.
\]
for all $j=1,2,...,N.$

We set
\[
\overline{A}:=\underset{j=1,2,...,N}{\cup}\left(  A^{\gamma,j}\times\left[
-1,1\right]  e_{j}\right)  ,\text{ }\overline{A}^{\gamma,j}:=A^{\gamma
,j}\times\left[  -1,1\right]  e_{j}.
\]
For this extended system, we will check that our controllability assumptions
(\textbf{Aa}) and (\textbf{Ab}) hold true for explicit sets of controls.

The reader is invited to notice that the following hold true :%

\begin{equation}
\left.  \overline{A}_{\gamma,x}=\overline{A}^{\gamma,j},\text{ if }x\in
J_{j}^{\varepsilon,+},\text{ }\overline{A}_{\gamma,O}%
=\underset{j=1,2,...,N}{\cup}\overline{A}^{\gamma,j},\text{ }\overline
{A}_{\gamma,\left(  1+\varepsilon\right)  e_{j}}=A_{\gamma,e_{j}}\times\left[
-1,1\right]  e_{j}\text{,}\right.  \tag{$\overline{Aa}$}%
\end{equation}
for all $j=1,2,...,N.$ Let us fix $j\leq M.$

(i) If $\gamma\in E_{j}^{active}$, then%

\[
\overline{A}_{\gamma,-\varepsilon e_{j}}=\left\{  \left(  a,b\right)
\in\overline{A}^{\gamma,j}:f_{\gamma}\left(  O,a\right)  \in%
\mathbb{R}
_{+}e_{j}\right\}  .
\]
The set $\overline{A}_{\gamma,-\varepsilon e_{j}}$ is nonempty. Indeed, the
control $\left(  a_{\gamma,j}^{+},b\right)  $ ($a_{\gamma,j}^{+}$ given by the
assumption (\textbf{Ab}) and $b\in\left[  -1,1\right]  e_{j}$ arbitrary)
belongs to $\overline{A}_{\gamma,-\varepsilon e_{j}}$ and%
\[
\left\langle f_{\gamma}^{\rho}\left(  -\varepsilon e_{j},a_{\gamma,j}%
^{+},b\right)  ,\left(  -e_{j}\right)  \right\rangle =\left\langle f_{\gamma
}\left(  O,a_{\gamma,j}^{+}\right)  ,\left(  -e_{j}\right)  \right\rangle
<-\beta,
\]
for all $b\in\left[  -1,1\right]  e_{j}.$

(ii) If $\gamma\in E_{j}^{inactive}$, then
\[
\overline{A}_{\gamma,-\varepsilon e_{j}}=\overline{A}^{\gamma,j}.
\]
Indeed,
\[
\left\langle f_{\gamma}^{\rho}\left(  -\varepsilon e_{j},a,b\right)  ,\left(
-e_{j}\right)  \right\rangle =\left\langle -f_{\gamma}\left(  \varepsilon
e_{j}-\rho b,a\right)  ,\left(  -e_{j}\right)  \right\rangle =\left\langle
f_{\gamma}\left(  \varepsilon e_{j}-\rho b,a\right)  ,e_{j}\right\rangle
\leq0,
\]
for $\varepsilon$ small enough and all $\left(  a,b\right)  \in\overline
{A}^{\gamma,j}.$

Thus, (\textbf{Aa}) holds true for the system driven by $\left(  f^{\rho
},\lambda^{\rho},Q^{\rho}\right)  .$

Concerning the assumption (\textbf{Ab}), for the already existing branches, it
suffices to take $b=0$ and the controls $a_{\gamma,j}^{+},a_{\gamma,j}%
^{-},a_{\gamma,j}^{0}.$ Let us now fix $j\leq M.$

(i) If $\gamma\in E_{j}^{active}$, then $\gamma\in E_{j+N}^{active}$, by
construction. We recall that $e_{j+N}=-e_{j}$. Moreover we have%
\[
\left\langle f_{\gamma}^{\rho}\left(  O,\left(  a_{\gamma,j}^{+},0\right)
\right)  ,-e_{j}\right\rangle <-\beta\text{ and }\left\langle f_{\gamma}%
^{\rho}\left(  O,\left(  a_{\gamma,j}^{-},0\right)  \right)  ,e_{j}%
\right\rangle >\beta.
\]

(ii) For $\gamma\in E_{j}^{inactive}=E_{j+N}^{active},$
\[
\left\langle f_{\gamma}^{\rho}\left(  x,\left(  a_{\gamma,j}^{-},0\right)
\right)  ,-e_{j}\right\rangle =\left\langle -f_{\gamma}\left(  -x,a_{\gamma
,j}^{-}\right)  ,-e_{j}\right\rangle \leq-\beta\left\langle -x,e_{j}%
\right\rangle ^{\kappa},
\]
for all $x\in\left[  -\varepsilon,0\right]  e_{j}$ and $f_{\gamma}^{\rho
}\left(  O,\left(  a_{\gamma,j}^{0},0\right)  \right)  =0.$

We cannot have
\[
\left\langle f_{\gamma}^{\rho}\left(  x,\left(  a,b\right)  \right)
,-e_{j}\right\rangle \leq0,
\]
for all $\left(  a,b\right)  \in\overline{A}^{\gamma,j}$ and all $x\in J_{j},$
$\left\vert x\right\vert \leq\eta$ (close enough to $O$). Nevertheless, as we
have already hinted before (see Remark \ref{Uniform_remark_one} (ii)), this
condition and the one in (\textbf{Ac}) are no longer necessary since every
control is (locally) admissible at $O$. Thus, the conclusion of Lemma
\ref{Projection_Lemma} holds true and so do all the assertions on the value
functions in this framework.

At this point, we consider the process $\left(  X_{t}^{\rho,x_{0},\gamma
_{0},\overline{\alpha}},\Gamma_{t}^{\rho,x_{0},\gamma_{0},\overline{\alpha}%
}\right)  $ constructed as in Section \ref{constr+ass} using $\left(  f^{\rho
},\lambda^{\rho},Q^{\rho}\right)  $ and controls $\overline{\alpha}$ with
values in $\overline{A}.$ We also let $y^{\rho}$ denote the solution of the
ordinary differential equation driven by $f^{\rho}.$

Then, the value functions
\begin{align*}
&  v^{\delta,\varepsilon,\rho}\left(  x,\gamma\right) \\
&  :=\inf_{\overline{\alpha}\in\overline{\mathcal{A}}_{ad}^{%
\mathbb{N}
}}\mathbb{E}\left[
{\textstyle\sum\limits_{n\geq0}}
\int_{\tau_{n}}^{\tau_{n+1}}e^{-\delta t}l_{\Gamma_{\tau_{n}}^{\rho
,x,\gamma,\overline{\alpha}}}^{\rho}\left(  y_{\Gamma_{\tau_{n}}^{\rho
,\gamma,x,\overline{\alpha}}}^{\rho}\left(  t;X_{\tau_{n}}^{\rho
,x,\gamma,\overline{\alpha}},\overline{\alpha}_{n+1}\right)  ,\overline
{\alpha}_{n+1}\left(  t-\Gamma_{\tau_{n}}^{\rho,x,\gamma,\overline{\alpha}%
};X_{\tau_{n}}^{\rho,x,\gamma,\overline{\alpha}},\Gamma_{\tau_{n}}%
^{\rho,x,\gamma,\overline{\alpha}}\right)  \right)  \right]
\end{align*}
are bounded, uniformly continuous and satisfy, in the generalized sense given
by Definition \ref{Def_generalized_solution} and Theorem
\ref{Th_existence_solution} the Hamilton-Jacobi integrodifferential system%

\begin{equation}
\delta v^{\delta,\varepsilon,\rho}\left(  x,\gamma\right)  +\sup_{\left(
a,b\right)  \in\overline{A}_{\gamma,x}}\left\{
\begin{array}
[c]{c}%
-\left\langle f_{\gamma}\left(  x+\rho b,a\right)  ,Dv^{\delta,\varepsilon
,\rho}\left(  x,\gamma\right)  \right\rangle -l_{\gamma}\left(  x+\rho
b,a\right) \\
-\lambda\left(  x+\rho b,\gamma,a\right)
{\textstyle\sum\limits_{\gamma^{\prime}\in E}}
Q\left(  x+\rho b,\gamma,\gamma^{\prime},a\right)  \left(  v^{\delta
,\varepsilon,\rho}\left(  x,\gamma^{\prime}\right)  -v^{\delta,\varepsilon
,\rho}\left(  x,\gamma\right)  \right)
\end{array}
\right\}  \leq0, \label{HJeps}%
\end{equation}
for all $\left(  x,\gamma\right)  \in\mathcal{G}^{+,\varepsilon}\times E.$

\subsection{Another definition for solutions in the extended intersection}

We define
\[
\overline{co}\overline{fl}^{\rho}\left(  O,\gamma\right)
:=\underset{j=1,2,...,N}{\cup}\overline{co}\left\{  \overline{fl}^{\rho
}\left(  O,\gamma,\left(  a,b\right)  \right)  :\left(  a,b\right)
\in\overline{A}^{\gamma,j}\right\}
\]
and recall that
\[
\overline{FL}^{\rho}(x,\gamma)=\overline{co}\overline{fl}^{\rho}\left(
x,\gamma\right)  \left(  :=\overline{co}\left\{  \overline{fl}^{\rho}\left(
x,\gamma,a\right)  :a\in A_{\gamma,x}\right\}  \right)  ,
\]
for all $x\in\mathcal{G}^{+,\varepsilon}\smallsetminus\left\{  O\right\}  $
and, for every $j\leq N,$%
\begin{align*}
\overline{FL}^{\rho}(\left(  1+\varepsilon\right)  e_{j},\gamma)  &
\subset\overline{co}\overline{fl}^{\rho}\left(  \left(  1+\varepsilon\right)
e_{j},\gamma\right) \\
&  \left(  :=\overline{co}\left\{  \overline{fl}^{\rho}\left(  \left(
1+\varepsilon\right)  e_{j},\gamma,a\right)  :a\in A^{\gamma,e_{j}}\right\}
\cap\left(  \mathcal{%
\mathbb{R}
}_{-}e_{j}\times\mathcal{M}_{+}\left(  E\right)  \times%
\mathbb{R}
\right)  \right)  .
\end{align*}
Also, for every $j\leq M,$%
\begin{align*}
\overline{FL}^{\rho}(-\varepsilon e_{j},\gamma)  &  \subset\overline
{co}\overline{fl}^{\rho}\left(  -\varepsilon e_{j},\gamma\right) \\
&  \left(  :=\overline{co}\left\{  \overline{fl}^{\rho}\left(  -\varepsilon
e_{j},\gamma,a\right)  :a\in A^{\gamma,e_{j}}\right\}  \cap\left(  \mathcal{%
\mathbb{R}
}_{+}e_{j}\times\mathcal{M}_{+}\left(  E\right)  \times%
\mathbb{R}
\right)  \right)
\end{align*}

We consider another definition for viscosity subsolutions by taking more
regular test functions.

\begin{definition}
\label{Def_more_regular_test}A bounded, upper (resp. lower) semicontinuous
function $V$ is said to be a classical constrained viscosity subsolution (resp
subsolution of (\ref{HJeps}) if, for every $\left(  \gamma_{0},x_{0}\right)
\in E\times\mathcal{G}^{+,\varepsilon}$ (resp. $E\times\overline{\mathcal{G}%
}^{+,\varepsilon}$), whenever $\varphi\in C_{b}\left(  \overline{\mathcal{G}%
}^{+,\varepsilon}\right)  $ for which $\varphi\mid_{\overline{J_{j}%
^{\varepsilon}}}\in C_{b}^{1}\left(  \overline{J_{j}^{\varepsilon}}\right)  ,$
for all $j=1,2,...,N$ is a test function such that $x_{0}\in Argmax\left(
V\left(  \cdot,\gamma_{0}\right)  -\varphi\left(  \cdot\right)  \right)  ,$
one has
\[
\delta V\left(  x_{0},\gamma_{0}\right)  +\sup_{\left(  \xi,\zeta,\eta\right)
\in\overline{co}\overline{fl}^{\rho}\left(  x_{0},\gamma_{0}\right)  }\left\{
\begin{array}
[c]{c}%
-\left\langle D\varphi\left(  x_{0};\xi\right)  \right\rangle -\eta\\
-%
{\textstyle\sum\limits_{\gamma^{\prime}\in E}}
\zeta\left(  \gamma^{\prime}\right)  \left(  V\left(  x_{0},\gamma^{\prime
}\right)  -V\left(  x_{0},\gamma_{0}\right)  \right)
\end{array}
\right\}  \leq0,
\]
$(resp.$ $\geq0).$
\end{definition}

We get the following characterization of $v^{\delta,\varepsilon,\rho}.$

\begin{theorem}
\label{ThSolTwo}The bounded uniformly continuous function $v^{\delta
,\varepsilon,\rho}$ is a classical constrained viscosity subsolution of
(\ref{HJeps}). Moreover, it satisfies the supersolution condition on
$E\times\left(  \overline{\mathcal{G}}^{+,\varepsilon}\smallsetminus\left\{
O\right\}  \right)  .$
\end{theorem}

\begin{proof}
The reader is invited to note that the test functions in this case are more
regular than in Definition \ref{Def_generalized_solution}. Thus, the equality
$\overline{FL}^{\rho}(x,\gamma)=\overline{co}\overline{fl}^{\rho}\left(
x,\gamma\right)  $ implies the viscosity sub/super condition at every point
$x\in\mathcal{G}^{+,\varepsilon}\smallsetminus\left\{  O\right\}  .$ The
supersolution condition at $\left(  1+\varepsilon\right)  e_{j}$ (resp.
$-\varepsilon e_{j}$) follows from the inclusion $\overline{FL}^{\rho}(\left(
1+\varepsilon\right)  e_{j},\gamma)\subset\overline{co}\overline{fl}^{\rho
}\left(  \left(  1+\varepsilon\right)  e_{j},\gamma\right)  $ (resp.
$\overline{FL}^{\rho}(-\varepsilon e_{j},\gamma)\subset\overline{co}%
\overline{fl}^{\rho}\left(  -\varepsilon e_{j},\gamma\right)  $)$.$

The constant control $\left(  a,b\right)  \in\overline{A}^{\gamma,1}$ is
locally admissible at $O$ (on the extended graph $\overline{\mathcal{G}%
}^{+,\varepsilon}$)$.$ Hence, reasoning as in the subsolution part of theorem
\ref{Th_existence_solution} (for constant $\overline{\alpha_{n}}=\left(
a,b\right)  $), one proves that if $\varphi$ is a regular test function such
that $O\in Argmax\left(  v^{\delta}\left(  \gamma,\cdot\right)  -\varphi
\left(  \cdot\right)  \right)  $, then
\begin{align*}
0  &  \leq l_{\gamma}^{\rho}\left(  O,\left(  a,b\right)  \right)
-\delta\varphi\left(  O\right)  +\left\langle D\left(  \varphi\mid
_{\overline{J_{1}^{\varepsilon}}}\right)  \left(  O\right)  ,f_{\gamma}\left(
x,\left(  a,b\right)  \right)  \right\rangle +\\
&  \lambda^{\rho}\left(  O,\gamma,\left(  a,b\right)  \right)
{\textstyle\sum\limits_{\gamma^{\prime}\neq\gamma}}
Q^{\rho}\left(  O,\gamma,\gamma^{\prime},\left(  a,b\right)  \right)  \left(
v^{\delta,\varepsilon,\rho}\left(  O,\gamma^{\prime}\right)  -v^{\delta
,\varepsilon}\left(  O,\gamma\right)  \right)  .
\end{align*}
Thus, continuity and convexity arguments imply that
\[
\delta\varphi\left(  O\right)  +\sup_{\left(  \xi,\zeta,\eta\right)
\in\overline{co}\left\{  \overline{fl}^{\rho}\left(  O,\gamma,\left(
a,b\right)  \right)  :\left(  a,b\right)  \in\overline{A}^{\gamma,1}\right\}
}\left\{
\begin{array}
[c]{c}%
-\left\langle D\varphi\left(  O;\xi\right)  \right\rangle -\eta\\
-%
{\textstyle\sum\limits_{\gamma^{\prime}\in E}}
\zeta\left(  \gamma^{\prime}\right)  \left(  v^{\delta,\varepsilon,\rho
}\left(  O,\gamma^{\prime}\right)  -v^{\delta,\varepsilon,\rho}\left(
O,\gamma\right)  \right)
\end{array}
\right\}  \leq0
\]
and the subsolution condition follows.
\end{proof}

\begin{remark}
In order to have (classical) uniqueness, one has to impose further conditions
at the junction $O.$ For example, in the case when $l_{\gamma}\left(
O,a\right)  $ does not depend on $a$ for all $\gamma\in\underset{j\leq N}{%
{\textstyle\bigcup}
}E_{j}^{active},$ one reasons in the same way as in Section 5.2 of
\cite{AchdouCamilliCutri2013}. The arguments are quasi-identical and we prefer
to concentrate on a different approach to uniqueness. Alternatively, one can
impose the analog of the Assumption 2.3 in \cite{AchdouCamilliCutri2013},
i.e.
\[
\left(  \left\{  0\right\}  \times\mathcal{M}_{+}\left(  E\right)
\times\left\{  \inf_{a\in A}l_{\gamma}\left(  O,a\right)  \right\}  \right)
\cap\overline{co}\left\{  \overline{fl}^{0}\left(  O,\gamma,\left(
a,b\right)  \right)  :\left(  a,b\right)  \in\overline{A}^{\gamma,j}\right\}
\neq\emptyset,
\]
for all $j$ such that $\gamma\in E_{j}^{active}.$
\end{remark}

\subsection{Convergence to the initial value function}

Unless stated otherwise, we assume the controllability conditions (\textbf{Aa,
Ab, Ad}), the compatibility at the intersection\textbf{ (Ab', Ac'), }the
regularity of the coefficients and cost functions \textbf{(A1-A4)} and the
compatibility condition (\textbf{B}) to hold true.

\textbf{(C)} Throughout the subsection, we also assume that $l$ does not
depend on the control at $O$ and the nodes $e_{j}$.

This "projection long-run compatibility condition" will allow to change the
control process around the "critical" points in order to obtain, from
admissible controls on $\overline{\mathcal{G}}^{+,\varepsilon}$ an admissible
control keeping the trajectory in $\overline{\mathcal{G}}.$ This assumption
\textbf{(C)} is only needed to prove Lemma \ref{Projection_Lemma_eps} in its
full generality. We have chosen to give a deeper result in Lemma
\ref{Projection_Lemma_eps} for further developments on the subject.

Let us fix $\varepsilon>0$ small enough. We introduce the following
notations:
\begin{align*}
t_{\varepsilon}  &  :=-\frac{1}{\delta}\ln\left(  \frac{\varepsilon\delta
}{2\left\vert f\right\vert _{0}}\right)  ,\text{ }\rho_{\varepsilon}%
:=-\frac{\varepsilon^{1+\frac{2Lip(f)}{(1-\kappa)\delta}}}{\ln(\varepsilon
)},r_{\varepsilon}^{\prime}\leq\frac{\rho_{\varepsilon}}{2},\text{ }\\
\omega_{\varepsilon}(t;r)  &  :=e^{Lip\left(  f\right)  t}\left(
r+(2\rho_{\varepsilon}\vee4r_{\varepsilon}^{\prime})Lip\left(  f\right)
t\right)  ,\text{ }t\geq0,r\geq0,\text{ }\Phi(\varepsilon):=\left(
\frac{|f|_{0}}{(1-\kappa)\beta}+1\right)  \left(  \omega_{\varepsilon}\left(
t_{\varepsilon};r_{\varepsilon}^{\prime}\right)  \right)  ^{1-\kappa}.\text{ }%
\end{align*}
The reader is invited to note that
\[
\omega_{\varepsilon}(t;\omega_{\varepsilon}(t^{\ast};r))\leq\omega
_{\varepsilon}(t^{\ast}+t;r),
\]
for all $t,t^{\ast},r\geq0.$ To get the best approximation and simplify the
proof of Lemma \ref{Projection_Lemma_eps}, we also strengthen (\textbf{A1)
}and ask that the restriction of $f_{\gamma}$ to $\left[  0,1\right]  e_{j}$
be Lipschitz-continuous for $\gamma\in E_{j}^{active}$. We emphasize that this
only affects the definition of $\rho_{\varepsilon}$ in Lemma
\ref{Projection_Lemma_eps} but not Theorem \ref{ThConvergence}.

With these notations, we establish.

\begin{lemma}
\label{Projection_Lemma_eps}Whenever $\gamma\in E$, $x\in J_{1}^{\varepsilon}$
and $\overline{\alpha}=\left(  \alpha,\beta\right)  \in\overline{\mathcal{A}%
}_{\gamma,x}$, there exists $\mathcal{P}_{x}^{\varepsilon}\left(
\alpha\right)  $ (also depending on $\gamma$) such that $\left(
\mathcal{P}_{x}^{\varepsilon}\left(  \alpha\right)  ,0\right)  \in
\overline{\mathcal{A}}_{\gamma,x}$ such that
\begin{equation}
\left\vert y_{\gamma}^{\rho_{\varepsilon}}\left(  t;x,\left(  \mathcal{P}%
_{x}\left(  \overline{\alpha}\right)  ,0\right)  \right)  -y_{\gamma}%
^{\rho_{\varepsilon}}\left(  t;x,\overline{\alpha}\right)  \right\vert
\leq\omega_{\varepsilon}(t_{\varepsilon};\Phi(\varepsilon)),
\end{equation}
for $t\leq t_{\varepsilon}$. Moreover, when (\textbf{C}) holds true,
\begin{equation}
\lim_{\varepsilon\rightarrow0}\sup_{t\leq t_{\varepsilon}}\left\vert
\begin{array}
[c]{c}%
\int_{0}^{t}e^{-\delta s}l_{\gamma}^{\rho_{\varepsilon}}\left(  y_{\gamma
}^{\rho_{\varepsilon}}\left(  t;x,\left(  \mathcal{P}_{x}^{\varepsilon}\left(
\alpha\right)  ,0\right)  \right)  ,\left(  \mathcal{P}_{x}^{\varepsilon
}\left(  \alpha\right)  \left(  s\right)  ,0\right)  \right)  ds\\
-\int_{0}^{t}e^{-\delta s}l_{\gamma}\left(  y_{\gamma}\left(  s;x,\overline
{\alpha}\right)  ,\overline{\alpha}\left(  s\right)  \right)  ds
\end{array}
\right\vert =0.
\end{equation}

(ii) Moreover, if $\overline{\alpha}=\left(  \alpha,\beta\right)  \in
\overline{\mathcal{A}}_{ad},$ then, for every $\varepsilon>0$ there exists
$\left(  \mathcal{P}^{\varepsilon}\left(  \overline{\alpha}\right)  ,0\right)
\in\overline{\mathcal{A}}_{ad}$ such that the previous inequalities are
satisfied with $\mathcal{P}^{\varepsilon}\left(  \overline{\alpha}\right)
\left(  \cdot,x,\gamma\right)  $ replacing $\mathcal{P}_{x}^{\varepsilon
}\left(  \overline{\alpha}\right)  .$
\end{lemma}

We postpone the proof of this Lemma to the Appendix. We emphasize that
whenever $\left(  \alpha,0\right)  \in\overline{\mathcal{A}}_{ad},$ one has
$y_{\gamma}^{\rho_{\varepsilon}}\left(  t;x,\left(  \mathcal{\alpha},0\right)
\right)  =y_{\gamma}\left(  t;x,\mathcal{\alpha}\right)  $ (and similar for
$l_{\gamma}^{\rho_{\varepsilon}},Q_{\gamma}^{\rho_{\varepsilon}}%
,\lambda_{\gamma}^{\rho_{\varepsilon}}$)$,$ even though $\alpha$ may not
belong to $\mathcal{A}_{\gamma,x}$. The second argument takes care of this
later issue.

\begin{lemma}
\label{Projection_Lemma_eps_Partii} Let us consider $T>0.$ Then, there exists
a decreasing function $\omega:%
\mathbb{R}
_{+}\longrightarrow%
\mathbb{R}
_{+}$ such $\omega\left(  0\right)  =\omega\left(  0+\right)  =0$ and whenever
$\gamma\in E$, $x\in\overline{\mathcal{G}},$ and $\left(  \alpha,0\right)
\in\overline{\mathcal{A}}_{\gamma,x},$ there exists $\mathcal{P}_{\gamma
,x}\left(  \alpha\right)  \in\mathcal{A}_{\gamma,x}$ such that
\[
\left.
\begin{array}
[c]{l}%
\left\vert y_{\gamma}\left(  t;x,\mathcal{\alpha}\right)  -y_{\gamma}\left(
t;x,\mathcal{P}_{\gamma,x}\left(  \alpha\right)  \right)  \right\vert
\leq\omega(\varepsilon),\\
\sup_{t\leq T}%
\begin{array}
[c]{c}%
\int_{0}^{t}e^{-\delta s}l_{\gamma}\left(  y_{\gamma}\left(  t;x,\mathcal{P}%
_{\gamma,x}\left(  \alpha\right)  \right)  ,\mathcal{P}_{\gamma,x}\left(
\alpha\right)  \left(  s\right)  \right)  ds\\
-\int_{0}^{t}e^{-\delta s}l_{\gamma}\left(  y_{\gamma}\left(
t;x,\mathcal{\alpha}\right)  ,\alpha\left(  s\right)  \right)  ds
\end{array}
\leq\omega\left(  \varepsilon\right)  ,
\end{array}
\right.
\]
and%
\[
\sup_{s\leq T}\left.
\begin{array}
[c]{c}%
\left\vert Q\left(  y_{\gamma}\left(  s;x,\mathcal{P}_{\gamma,x}\left(
\alpha\right)  \right)  ,\gamma,\gamma^{\prime},\mathcal{P}_{\gamma,x}\left(
\alpha\right)  \left(  s\right)  \right)  -Q\left(  y_{\gamma}\left(
s;x,\alpha\right)  ,\gamma,\gamma^{\prime},\alpha\left(  s\right)  \right)
\right\vert \\
+\left\vert \lambda\left(  y_{\gamma}\left(  s;x,\mathcal{P}_{\gamma,x}\left(
\alpha\right)  \right)  ,\gamma,\mathcal{P}_{\gamma,x}\left(  \alpha\right)
\left(  s\right)  \right)  -\lambda\left(  y_{\gamma}\left(  s;x,\alpha
\right)  ,\gamma,\alpha\left(  s\right)  \right)  \right\vert
\end{array}
\right.  \leq\omega\left(  \varepsilon\right)  ,
\]
for all $\gamma^{\prime}\in E.$

(ii) Moreover, if $\left(  \alpha,0\right)  \in\overline{\mathcal{A}}_{ad},$
then, for every $\varepsilon>0$ there exists $\mathcal{P}\left(
\alpha\right)  \in\mathcal{A}_{ad}$ such that the previous inequalities are
satisfied with $\mathcal{P}\left(  \alpha\right)  \left(  \cdot,x,\gamma
\right)  $ replacing $\mathcal{P}_{\gamma,x}\left(  \alpha\right)  \left(
\cdot\right)  .$
\end{lemma}

Although the approach is rather obvious (when looking at the proofs of Lemmas
\ref{Projection_Lemma} or \ref{Projection_Lemma_eps}), hints on the proof are
given in the Appendix. We wish to emphasize that, although the trajectories
can be kept close up to a fixed $T$ due to the proximity of $\overline
{\mathcal{G}}^{+,\varepsilon}$ and $\overline{\mathcal{G}}$, we cannot do
better then $\varepsilon$. Thus, we are unable to give the same kind of
estimates up to $t_{\varepsilon}$.

The main result of the subsection is the following convergence theorem.

\begin{theorem}
\label{ThConvergence}Under the assumption (C), the following convergence holds
true
\[
\lim_{\varepsilon\rightarrow0}\sup_{x\in\overline{\mathcal{G}},\gamma\in
E}\left\vert v^{\delta,\varepsilon,\rho_{\varepsilon}}\left(  x,\gamma\right)
-v^{\delta}\left(  x,\gamma\right)  \right\vert =0.
\]

\end{theorem}

\begin{proof}
The definition of our value functions yields $v^{\delta,\varepsilon
,\rho_{\varepsilon}}\leq v^{\delta}$ on $\overline{\mathcal{G}}\times E.$
Hence, we only need to prove the converse inequality. The proof is very
similar to that of Theorem 15 in \cite{G8}. Let us fix $\left(  x,\gamma
\right)  \in\overline{\mathcal{G}}\times E$, $T>0$ and (for the time being,)
$\varepsilon>0.$ Then using the dynamic programming principle for
$v^{\delta,\varepsilon,\rho_{\varepsilon}}$ one gets the existence of some
admissible control process $\overline{\alpha}$ such that
\begin{equation}
v^{\delta,\varepsilon,\rho_{\varepsilon}}\left(  x,\gamma\right)
\geq\mathbb{E}\left[
\begin{array}
[c]{c}%
\int_{0}^{T\wedge\tau_{1}}e^{-\delta t}l_{\gamma}^{\rho_{\varepsilon}}\left(
y_{\gamma}^{\rho_{\varepsilon}}\left(  t;x,\alpha\right)  ,\overline{\alpha
}\left(  t;x,\gamma\right)  \right)  dt\\
+e^{-\delta\left(  T\wedge\tau_{1}\right)  }v^{\delta}\left(  y_{\gamma}%
^{\rho_{\varepsilon}}\left(  T\wedge\tau_{1};x,\alpha\right)  ,\Gamma
_{T\wedge\tau_{1}}^{\rho_{\varepsilon},x,\gamma,\overline{\alpha}}\right)
\end{array}
\right]  -\varepsilon. \label{choicealphabar}%
\end{equation}
For simplicity, we let $\mathcal{P}$ and $\mathcal{P}^{\varepsilon}$ denote
the two projectors of the previous lemmas and introduce the following
notations:
\begin{align*}
\overline{\alpha}_{t}  &  =\overline{\alpha}\left(  t;x,\gamma\right)  ,\text{
}\alpha_{t}=\mathcal{P}\left(  \mathcal{P}^{\varepsilon}\left(  \overline
{\alpha}\right)  \right)  \left(  t;x,\gamma\right)  ,\\
\overline{\lambda}\left(  t\right)   &  =\lambda\left(  y_{\gamma}%
^{\rho_{\varepsilon}}\left(  t;x,\overline{\alpha}\right)  ,\gamma
,\overline{\alpha}_{t}\right)  ,\text{ }\overline{\Lambda}\left(  t\right)
=\exp\left(  -\int_{0}^{t}\overline{\lambda}\left(  s\right)  ds\right) \\
\lambda\left(  t\right)   &  =\lambda\left(  y_{\gamma}\left(  t;x,\alpha
\right)  ,\gamma,\alpha_{t}\right)  ,\text{ }\Lambda\left(  t\right)
=\exp\left(  -\int_{0}^{t}\lambda\left(  s\right)  ds\right)  ,
\end{align*}
for all $t\geq0$. We denote the right-hand member of the inequality
(\ref{choicealphabar}) by $I.$ Then, $I$ is explicitly given by
\begin{align*}
I  &  =\int_{0}^{T}\overline{\lambda}(t)\overline{\Lambda}\left(  t\right)
\int_{0}^{t}e^{-\delta s}l_{\gamma}^{\rho_{\varepsilon}}\left(  y_{\gamma
}^{\rho_{\varepsilon}}\left(  s;x,\overline{\alpha}\right)  ,\overline{\alpha
}_{s}\right)  dsdt\\
&  +\int_{0}^{T}\overline{\lambda}(t)\overline{\Lambda}\left(  t\right)
e^{-\delta t}\underset{\gamma^{\prime}\in E}{%
{\textstyle\sum}
}v^{\delta,\varepsilon,\rho_{\varepsilon}}\left(  y_{\gamma}^{\rho
_{\varepsilon}}\left(  t;x,\overline{\alpha}\right)  ,\gamma^{\prime}\right)
Q^{\rho_{\varepsilon}}\left(  y_{\gamma}^{\rho_{\varepsilon}}\left(
t;x,\overline{\alpha}\right)  ,\gamma,\gamma^{\prime},\overline{\alpha}%
_{t}\right)  dt\\
&  +\overline{\Lambda}\left(  T\right)  \int_{0}^{T}e^{-\delta t}l_{\gamma
}^{\rho_{\varepsilon}}\left(  y_{\gamma}^{\rho_{\varepsilon}}\left(
t;x,\overline{\alpha}\right)  ,\overline{\alpha}_{t}\right)  dt+\overline
{\Lambda}\left(  T\right)  e^{-\delta T}v^{\delta,\varepsilon,\rho
_{\varepsilon}}\left(  y_{\gamma}^{\rho_{\varepsilon}}\left(  T;x,\overline
{\alpha}\right)  ,\gamma\right) \\
&  =I_{1}+I_{2}+I_{3}+I_{4}.
\end{align*}
The conclusion follows using the Lemmas \ref{Projection_Lemma_eps} and
\ref{Projection_Lemma_eps_Partii}. These estimates are tailor-made to allow
substituting $\overline{\lambda}$, $\overline{\Lambda},$ $l_{\gamma}%
^{\rho_{\varepsilon}}$and $y_{\gamma}^{\rho_{\varepsilon}}$with $\lambda
,\Lambda,l_{\gamma}$ and $y_{\gamma}$ and the error is some (generic)
$\omega\left(  \varepsilon\right)  \underset{\varepsilon\rightarrow
0}{\rightarrow}0$ (the reader may also want to take a glance at the proof of
Theorem 15 in \cite{G8}). In the following, this function $\omega$ may change
from one line to another. Let us recall (see Remark \ref{Uniform_remark_three}%
) that $v^{\delta,\varepsilon,\rho_{\varepsilon}}$ have the same continuity
modulus (denoted $\omega^{\delta}$ and independent of $\varepsilon$). Then,
$v^{\delta,\varepsilon,\rho_{\varepsilon}}\left(  y_{\gamma}^{\rho
_{\varepsilon}}\left(  t;x,\overline{\alpha}\right)  ,\gamma^{\prime}\right)
$ can be replaced by $v^{\delta,\varepsilon,\rho_{\varepsilon}}\left(
y_{\gamma}\left(  t;x,\overline{\alpha}\right)  ,\gamma^{\prime}\right)  $
with an error $\omega^{\delta}\left(  \left\vert y_{\gamma}^{\rho
_{\varepsilon}}\left(  t;x,\overline{\alpha}\right)  -y_{\gamma}\left(
t;x,\alpha\right)  \right\vert \right)  ,$ hence, again some $\omega\left(
\varepsilon\right)  $. The only interesting terms in $I$ are $I_{2}$ and
$I_{4}$. For the term $I_{2},$ one writes%
\begin{align}
I_{2}  &  \geq\int_{0}^{T}\lambda(t)\Lambda\left(  t\right)  e^{-\delta
t}\underset{\gamma^{\prime}\in E}{%
{\textstyle\sum}
}v^{\delta,\varepsilon,\rho_{\varepsilon}}\left(  y_{\gamma}^{\rho
_{\varepsilon}}\left(  t;x,\overline{\alpha}\right)  ,\gamma^{\prime}\right)
Q^{\rho_{\varepsilon}}\left(  y_{\gamma}^{\rho_{\varepsilon}}\left(
t;x,\overline{\alpha}\right)  ,\gamma,\gamma^{\prime},\overline{\alpha}%
_{t}\right)  dt+\omega\left(  \varepsilon\right) \nonumber\\
&  \geq\int_{0}^{T}\lambda(t)\Lambda\left(  t\right)  e^{-\delta
t}\underset{\gamma^{\prime}\in E}{%
{\textstyle\sum}
}v^{\delta,\varepsilon,\rho_{\varepsilon}}\left(  y_{\gamma}\left(
t;x,\alpha\right)  ,\gamma^{\prime}\right)  Q\left(  y_{\gamma}\left(
t;x,\alpha\right)  ,\gamma,\gamma^{\prime},\alpha_{t}\right)  dt+\omega\left(
\varepsilon\right) \nonumber\\
&  \geq\int_{0}^{T}\lambda(t)\Lambda\left(  t\right)  e^{-\delta
t}\underset{\gamma^{\prime}\in E}{%
{\textstyle\sum}
}v^{\delta}\left(  y_{\gamma}\left(  t;x,\alpha\right)  ,\gamma^{\prime
}\right)  Q\left(  y_{\gamma}\left(  t;x,\alpha\right)  ,\gamma,\gamma
^{\prime},\alpha_{t}\right)  dt\nonumber\\
&  -\int_{0}^{T}\lambda(t)\Lambda\left(  t\right)  e^{-\delta t}dt\sup
_{\gamma^{\prime}\in E,z\in\overline{\mathcal{G}}}\left\vert v^{\delta
,\varepsilon,\rho_{\varepsilon}}(z,\gamma^{\prime})-v^{\delta}(z,\gamma
^{\prime})\right\vert +\omega\left(  \varepsilon\right)  \label{I3}%
\end{align}
Similar,
\begin{equation}
I_{4}\geq\Lambda\left(  T\right)  e^{-\delta T}v^{\delta}\left(  y_{\gamma
}\left(  T;x,\alpha\right)  \right)  -\Lambda\left(  T\right)  e^{-\delta
T}\sup_{\gamma^{\prime}\in E,z\in\overline{\mathcal{G}}}\left\vert
v^{\delta,\varepsilon,\rho_{\varepsilon}}(z,\gamma^{\prime})-v^{\delta
}(z,\gamma^{\prime})\right\vert +\omega\left(  \varepsilon\right)  .
\label{I4}%
\end{equation}
Hence, using (\ref{I3}, \ref{I4}), one gets
\begin{align*}
I  &  \geq\int_{0}^{T}\lambda(t)\Lambda\left(  t\right)  \int_{0}%
^{t}e^{-\delta s}l_{\gamma}\left(  y_{\gamma}\left(  s;x,\alpha\right)
,\alpha_{s}\right)  dsdt\\
&  +\int_{0}^{T}\lambda(t)\Lambda\left(  t\right)  e^{-\delta t}%
\underset{\gamma^{\prime}\in E}{%
{\textstyle\sum}
}v^{\delta}\left(  y_{\gamma}\left(  t;x,\alpha\right)  ,\gamma^{\prime
}\right)  Q\left(  y_{\gamma}\left(  t;x,\alpha\right)  ,\gamma,\gamma
^{\prime},\alpha_{t}\right)  dt\\
&  +\Lambda\left(  T\right)  \int_{0}^{T}e^{-\delta t}l_{\gamma}\left(
y_{\gamma}\left(  t;x,\alpha\right)  ,\alpha_{t}\right)  dt+\Lambda\left(
T\right)  e^{-\delta T}v^{\delta}\left(  y_{\gamma}\left(  T;x,\alpha\right)
,\gamma\right) \\
&  -\left[  \int_{0}^{T}\lambda(t)\Lambda\left(  t\right)  e^{-\delta
t}dt+\Lambda\left(  T\right)  e^{-\delta T}\right]  \sup_{\gamma^{\prime}\in
E,z\in\overline{\mathcal{G}}}\left\vert v^{\delta,\varepsilon,\rho
_{\varepsilon}}(z,\gamma^{\prime})-v^{\delta}(z,\gamma^{\prime})\right\vert
+\omega\left(  \varepsilon\right)  .
\end{align*}
Then, using the dynamic programming principle for $v^{\delta}$ and
(\ref{choicealphabar})$,$ one gets
\begin{align*}
v^{\delta,\varepsilon,\rho_{\varepsilon}}\left(  x,\gamma\right)   &  \geq
v^{\delta}(x,\gamma)-\left[  1-\delta\int_{0}^{T}\Lambda\left(  t\right)
e^{-\delta t}dt\right]  \sup_{\gamma^{\prime}\in E,z\in\overline{\mathcal{G}}%
}\left\vert v^{\delta,\varepsilon,\rho_{\varepsilon}}(z,\gamma^{\prime
})-v^{\delta}(z,\gamma^{\prime})\right\vert +\omega\left(  \varepsilon\right)
\\
&  \geq v^{\delta}(x,\gamma)-\left[  1-\delta\int_{0}^{T}e^{-\left(
\delta+\left\vert \lambda\right\vert _{0}\right)  t}dt\right]  \sup
_{\gamma^{\prime}\in E,z\in\overline{\mathcal{G}}}\left\vert v^{\delta
,\varepsilon,\rho_{\varepsilon}}(z,\gamma^{\prime})-v^{\delta}(z,\gamma
^{\prime})\right\vert +\omega\left(  \varepsilon\right)
\end{align*}
Thus,
\[
(0\leq)v^{\delta}(x,\gamma)-v^{\delta,\varepsilon,\rho_{\varepsilon}}\left(
x,\gamma\right)  \leq\left[  1-\delta\int_{0}^{T}e^{-\left(  \delta+\left\vert
\lambda\right\vert _{0}\right)  t}dt\right]  \sup_{\gamma^{\prime}\in
E,z\in\overline{\mathcal{G}}}\left\vert v^{\delta,\varepsilon,\rho
_{\varepsilon}}(z,\gamma^{\prime})-v^{\delta}(z,\gamma^{\prime})\right\vert
+\omega\left(  \varepsilon\right)  .
\]
The conclusion follows by taking the supremum over $x\in\overline{\mathcal{G}%
}$ and $\gamma\in E$ and allowing $\varepsilon\rightarrow0.$
\end{proof}

\begin{remark}
We recall (cf. Remark \ref{Uniform_remark_three}) that $v^{\delta
,\varepsilon,\rho_{\varepsilon}}$ have the same continuity modulus
(independent of $\varepsilon$). Moreover, $v^{\delta,\varepsilon
,\rho_{\varepsilon}}\left(  \cdot\right)  \leq\frac{\left\vert l\right\vert
_{0}}{\delta}$. Therefore, applying Arzela-Ascoli Theorem, there exists
$\lim_{\varepsilon\rightarrow0}\left(  v^{\delta,\varepsilon,\rho
_{\varepsilon}}\mid_{\overline{\mathcal{G}}}\right)  $ and this limit is
uniformly continuous. It would have sufficed, therefore, to prove that
$\lim_{\varepsilon\rightarrow0}v^{\delta,\varepsilon,\rho_{\varepsilon}%
}\left(  x;\gamma\right)  =v^{\delta}\left(  x,\gamma\right)  $ for all
$x\in\underset{i=1,2,...,N}{\cup}\left(  0,1\right)  e_{i}.$
\end{remark}

\subsection{Linearizing the problem}

We assume the (pseudo-)controllability conditions (\textbf{Aa, Ab, Ad}), the
compatibility at the intersection\textbf{ (Ab', Ac'), }the regularity of the
coefficients and cost functions \textbf{(A1-A4)}, the compatibility condition
(\textbf{B}) and the projection compatibility condition \textbf{(C}) to hold true.

\subsubsection{Smooth subsolutions\label{Subsubsection5.6.1}}

Starting from $v^{\delta,\varepsilon,\rho_{\varepsilon}},$ we will construct a
family of smooth subsolutions of (\ref{HJeps}) (with $\varepsilon=\rho=0$)
that converge to the value function $v^{\delta}.$ To this purpose, we
regularize the functions $v^{\delta,\varepsilon,\rho_{\varepsilon}}$ in each
direction given by $e_{j},$ for $j=1,2,...,N.$ Finally, we conveniently modify
the value at the junction point $O.$

We begin by picking $\left(  \psi_{\epsilon}\right)  _{\epsilon}$ to be a
sequence of standard mollifiers $\psi_{\epsilon}\left(  y\right)  =\frac
{1}{\epsilon}\psi\left(  \frac{y}{\epsilon}\right)  ,$ $y\in%
\mathbb{R}
,$ $\epsilon>0,$ where $\psi\in C^{\infty}\left(
\mathbb{R}
\right)  $ is a positive function such that
\[
Supp(\psi)\subset\left[  -1,1\right]  \text{ and }\int_{%
\mathbb{R}
}\psi(y)dy=1.
\]
For every $\varepsilon>0$ and every $0<\epsilon\leq\rho_{\varepsilon},$ one
can define regular functions $v_{\varepsilon,\epsilon}^{\delta,j}$ by setting
\[
v_{\varepsilon,\epsilon}^{\delta,j}\left(  x,\gamma\right)  =\int%
_{-\varepsilon}^{\varepsilon}v^{\delta,\varepsilon,\rho_{\varepsilon}}\left(
x-ye_{j},\gamma\right)  \psi_{\epsilon}\left(  y\right)  dy,
\]
for all $x\in\left(  -\varepsilon,1+\varepsilon\right)  e_{j},$ $j=1,M,$ or
$x\in\left(  -1-\varepsilon,1+\varepsilon\right)  e_{j},$ if $M<j\leq N.$
Using the same methods as those employed in \cite{G7}, Appendix (see also
\cite{G8}, Appendix A2 or \cite{krylov_00} or \cite{barles_jakobsen_02}, Lemma
2.7), it is easy to prove that
\begin{equation}
\delta v_{\varepsilon,\epsilon}^{\delta,j}\left(  x,\gamma\right)  +\left\{
\begin{array}
[c]{c}%
-\left\langle f_{\gamma}\left(  x,a\right)  ,Dv_{\varepsilon,\epsilon}%
^{\delta,j}\left(  x,\gamma\right)  \right\rangle -l_{\gamma}\left(
x,a\right) \\
-\lambda\left(  x,\gamma,a\right)
{\textstyle\sum\limits_{\gamma^{\prime}\in E}}
Q\left(  x,\gamma,\gamma^{\prime},a\right)  \left(  v_{\varepsilon,\epsilon
}^{\delta,j}\left(  x,\gamma^{\prime}\right)  -v_{\varepsilon,\epsilon
}^{\delta,j}\left(  x,\gamma\right)  \right)
\end{array}
\right\}  \leq0,
\end{equation}
for all $x\in\left[  0,1\right]  e_{j},$ $j\leq N$ and all $a\in A^{\gamma
,j}.$ Also, we note that
\[
\left\vert v_{\varepsilon,\epsilon}^{\delta,j}\left(  x,\gamma\right)
-v^{\delta}\left(  x,\gamma\right)  \right\vert \leq\left\vert v^{\delta
,\varepsilon,\rho_{\varepsilon}}-v^{\delta}\right\vert _{0}+\omega^{\delta
}\left(  \epsilon\right)  =:\omega\left(  \varepsilon,\epsilon\right)  ,
\]
for all $x,\gamma\in\overline{\mathcal{G}}\times E$, where $\omega^{\delta}$
is the continuity modulus of $v^{\delta}$ (with respect to the space
component). Theorem \ref{ThConvergence} yields
\[
\lim_{\varepsilon,\epsilon\rightarrow0}\omega\left(  \varepsilon
,\epsilon\right)  =0.
\]
We define an admissible test function by setting
\[
v_{\varepsilon}^{\delta}\left(  x,\gamma\right)  =v_{\varepsilon
,\rho_{\varepsilon}}^{\delta,j}\left(  x,\gamma\right)  -v_{\varepsilon
,\rho_{\varepsilon}}^{\delta,j}\left(  O,\gamma\right)  +\min_{j^{\prime
}=1,2,...,N}v_{\varepsilon,\rho_{\varepsilon}}^{\delta,j^{\prime}}\left(
O,\gamma\right)  -4\frac{\left\vert \lambda\right\vert _{0}}{\delta}%
\omega\left(  \varepsilon,\rho_{\varepsilon}\right)  ,
\]
for $x\in\left[  0,1\right]  e_{j},$ $1\leq j\leq N$ and $\gamma\in E.$ Then
$v_{\varepsilon}^{\delta}$ is a regular test function (continuous at $O$)
which satisfies
\begin{equation}
\left\{
\begin{array}
[c]{l}%
\left(
\begin{array}
[c]{c}%
\delta v_{\varepsilon}^{\delta}\left(  x,\gamma\right)  -\left\langle
f_{\gamma}\left(  x,a\right)  ,Dv_{\varepsilon}^{\delta}\left(  x,\gamma
\right)  \right\rangle -l_{\gamma}\left(  x,a\right) \\
-\lambda\left(  x,\gamma,a\right)
{\textstyle\sum\limits_{\gamma^{\prime}\in E}}
Q\left(  x,\gamma,\gamma^{\prime},a\right)  \left(  v_{\varepsilon}^{\delta
}\left(  x,\gamma^{\prime}\right)  -v_{\varepsilon}^{\delta}\left(
x,\gamma\right)  \right)
\end{array}
\right)  \leq0,\text{ and }\\
\lim_{\varepsilon\rightarrow0}\left\vert v_{\varepsilon}^{\delta}-v^{\delta
}\right\vert _{0}=0,
\end{array}
\right.  \label{estimatesEta}%
\end{equation}
for all $\left(  x,\gamma,a\right)  $ such that $x\in\left[  0,1\right]
e_{j},$ $\gamma\in E,$ $a\in A^{\gamma,j},$ $j\leq N$. These functions are
Lipschitz continuous on $\overline{\mathcal{G}}.$ (In fact, the reader can
check rather easily that the Lipschitz constant of $v_{\varepsilon}^{\delta}$
does not exceed $\sqrt{2}\frac{\underset{1\leq j\leq N}{\max}\left\vert
D\left(  v_{\varepsilon}^{\delta}\mid_{\left[  -1,1\right]  e_{j}}\right)
\right\vert _{0}}{\sqrt{1-\underset{i^{\prime},j^{\prime}\in\left\{
1,...,M+\frac{N-M}{2}\right\}  ,\text{ }i^{\prime}\neq j^{\prime}}{\max}%
\cos\left(  e_{i^{\prime}},e_{j^{\prime}}\right)  }}$). Hence, (using
Kirszbraun's Theorem,) one can find an extension (explicitly given by
\[
\widetilde{v}_{\varepsilon}^{\delta}\left(  x,\gamma\right)  :=\inf
_{y\in\overline{\mathcal{G}}}\left(  v_{\varepsilon}^{\delta}\left(
y,\gamma\right)  +Lip\left(  v_{\varepsilon}^{\delta}\right)  \left\vert
x-y\right\vert \right)  \text{)}%
\]
which is Lipschitz continuous on $%
\mathbb{R}
^{m}.$ As a by-product, this function (identified with $v_{\varepsilon
}^{\delta}\left(  \cdot,\gamma\right)  $ whenever no confusion is at risk) is
absolutely continuous on $%
\mathbb{R}
^{m}$ $\left(  AC\left(
\mathbb{R}
^{m}\right)  \right)  .$

\subsubsection{Occupation measures and embedding}

To every admissible control $\alpha\in\mathcal{A}_{ad}^{%
\mathbb{N}
}$ and $\gamma\in E,x\in\overline{\mathcal{G}}$, we can associate a
probability measure $\mu^{x,\gamma,\alpha}\in\mathcal{P}\left(
\mathbb{R}
^{m}\times E\times A\right)  $ by setting
\[
\mu^{x,\gamma,\alpha}\left(  A\times B\times C\right)  =\delta\mathbb{E}%
\left[  \int_{0}^{\infty}e^{-\delta t}1_{A\times B\times C}\left(
X_{t}^{x,\gamma,\alpha},\Gamma_{t}^{x,\gamma,\alpha},\alpha_{t}\right)
\right]  ,
\]
for all Borel sets $A\times B\times C\subset%
\mathbb{R}
^{m}\times E\times A$. As before, if $\left(  \tau_{i}\right)  _{i\geq0}$
denote the switch times, then $\alpha_{t}=\alpha_{i+1}\left(  t-\tau
_{i},X_{\tau_{i}}^{x,\gamma,\alpha},\Gamma_{\tau_{i}}^{x,\gamma,\alpha
}\right)  $ on $t\in\left[  \tau_{i},\tau_{i+1}\right)  $. Obviously, the
choice of admissible controls (under constraints) yields
\[
Supp\left(  \mu^{x,\gamma,\alpha}\right)  \subset\widehat{\overline
{\mathcal{G}}\times E\times A}:=\left\{  \left(  y,\gamma^{\prime},a\right)
\in\overline{\mathcal{G}}\times E\times A:a\in A^{\gamma^{\prime},j}\text{
whenever }y\in\overline{J_{j}}\right\}  .
\]
We note that the set $\widehat{\overline{\mathcal{G}}\times E\times A}$ is compact.

We denote by $BAC\left(
\mathbb{R}
^{m}\times E;%
\mathbb{R}
\right)  $ the set of all bounded functions $\varphi:%
\mathbb{R}
^{m}\times E\longrightarrow%
\mathbb{R}
$ such that $\varphi\left(  \cdot,\gamma^{\prime}\right)  \in AC\left(
\mathbb{R}
^{m}\right)  $ for all $\gamma^{\prime}\in E.$ Then, It\^{o}'s formula (see
\cite[Theorem 31.3 ]{davis_93}) yields
\begin{align}
&  \delta e^{-\delta T}\mathbb{E}\left[  \varphi\left(  X_{T}^{x,\gamma
,\alpha},\Gamma_{T}^{x,\gamma,\alpha}\right)  \right] \nonumber\\
&  =\delta\varphi\left(  x,\gamma\right)  +\mathbb{E}\int_{0}^{T}\delta
e^{-\delta s}\left[  -\delta\varphi\left(  X_{t}^{x,\gamma,\alpha},\Gamma
_{t}^{x,\gamma,\alpha}\right)  +\mathcal{U}^{\alpha_{t}}\varphi\left(
X_{t}^{x,\gamma,\alpha},\Gamma_{t}^{x,\gamma,\alpha}\right)  \right]  dt.
\label{ItoT}%
\end{align}
Here,
\[
\mathcal{U}^{a}\varphi\left(  y,\gamma^{\prime}\right)  =\left\langle
f_{\gamma}\left(  y,a\right)  ,D\varphi\left(  y,\gamma^{\prime}\right)
\right\rangle +\lambda\left(  y,\gamma^{\prime},a\right)
{\textstyle\sum\limits_{\gamma^{\prime\prime}\in E}}
Q\left(  y,\gamma^{\prime},\gamma^{\prime\prime},a\right)  \left(
\varphi\left(  \gamma^{\prime\prime},x\right)  -\varphi\left(  y,\gamma
^{\prime}\right)  \right)  ,
\]
for regular $\varphi\left(  \cdot,\gamma\right)  \in C_{b}^{1}\left(
\mathbb{R}
^{m}\right)  $ is the classical generator of the PDMP. We recall that the
extended domain of $\mathcal{U}^{a}$ includes functions such that
$\varphi\left(  \cdot,\gamma^{\prime}\right)  \in AC\left(
\mathbb{R}
^{m}\right)  $ (cf. Theorem 31.3 in \cite{davis_93}). Hence, passing to the
limit as $T\rightarrow\infty$ in (\ref{ItoT}) (and recalling that $\varphi$ is
bounded), one gets%
\[
\int_{%
\mathbb{R}
^{m}\times E\times A}\left[  \mathcal{U}^{a}\varphi\left(  y,\gamma^{\prime
}\right)  -\delta\left[  \varphi\left(  y,\gamma^{\prime}\right)
-\varphi\left(  x,\gamma\right)  \right]  \right]  \mu^{x,\gamma,\alpha
}\left(  dyd\gamma^{\prime}da\right)  =0.
\]
We set
\begin{equation}
\left.
\begin{array}
[c]{l}%
\Theta_{\overline{\mathcal{G}}}^{0}\left(  x,\gamma\right)  :=\left\{
\mu^{x,\gamma,\alpha}:\alpha\in\mathcal{A}_{ad}^{%
\mathbb{N}
}\right\}  \text{ and}\\
\Theta_{\overline{\mathcal{G}}}\left(  x,\gamma\right)  :=\left\{
\begin{array}
[c]{c}%
\mu^{x,\gamma,\alpha}\in\mathcal{P}\left(  \widehat{\overline{\mathcal{G}%
}\times E\times A}\right)  :\forall\varphi\in BAC\left(
\mathbb{R}
^{m}\times E;%
\mathbb{R}
\right) \\
\int_{%
\mathbb{R}
^{m}\times E\times A}\left[  -\mathcal{U}^{a}\varphi\left(  y,\gamma^{\prime
}\right)  +\delta\left[  \varphi\left(  y,\gamma^{\prime}\right)
-\varphi\left(  x,\gamma\right)  \right]  \right]  \mu\left(  dyd\gamma
^{\prime}da\right)  =0.
\end{array}
\right\}
\end{array}
\right.  \label{ThetaG}%
\end{equation}
We are now able to state (and prove) the main linearization result.

\begin{theorem}
\label{TH_Lin_Network}The following equalities hold true%
\begin{align*}
&  \delta v^{\delta}\left(  x,\gamma\right) \\
&  =\Lambda^{\delta}\left(  x,\gamma\right)  :=\inf_{\mu\in\Theta
_{\overline{\mathcal{G}}}\left(  x,\gamma\right)  }\int_{%
\mathbb{R}
^{m}\times E\times A}l_{\gamma^{\prime}}\left(  y,a\right)  \mu\left(
dyd\gamma^{\prime}da\right) \\
&  =\Lambda^{\delta,\ast}\left(  x,\gamma\right)  :=\sup\left\{
\begin{array}
[c]{c}%
\eta\in%
\mathbb{R}
:\exists\varphi\in BAC\left(
\mathbb{R}
^{m}\times E;%
\mathbb{R}
\right)  ,\text{ for all }\left(  y,\gamma^{\prime},a\right)  \in
\widehat{\overline{\mathcal{G}}\times E\times A},\\
\eta\leq\mathcal{U}^{a}\varphi\left(  y,\gamma^{\prime}\right)  +l_{\gamma
^{\prime}}\left(  y,a\right)  -\delta\left[  \varphi\left(  y,\gamma^{\prime
}\right)  -\varphi\left(  x,\gamma\right)  \right]  .
\end{array}
\right\}  ,
\end{align*}
for all $\left(  x,\gamma\right)  \in\overline{\mathcal{G}}\times E$.
\end{theorem}

\begin{proof}
Let us fix $\left(  x,\gamma\right)  \in\overline{\mathcal{G}}\times E$. It is
clear that
\[
\delta v^{\delta}\left(  x,\gamma\right)  \geq\inf_{\mu\in\Theta
_{\overline{\mathcal{G}}}\left(  x,\gamma\right)  }\int_{%
\mathbb{R}
^{m}\times E\times A}l_{\gamma^{\prime}}\left(  y,a\right)  \mu\left(
dyd\gamma^{\prime}da\right)
\]
since $\Theta_{\overline{\mathcal{G}}}^{0}\left(  x,\gamma\right)
\subset\Theta_{\overline{\mathcal{G}}}\left(  x,\gamma\right)  .$ Next, if
$\eta\leq\mathcal{U}^{a}\varphi\left(  y,\gamma^{\prime}\right)
+l_{\gamma^{\prime}}\left(  y,a\right)  -\delta\left[  \varphi\left(
y,\gamma^{\prime}\right)  -\varphi\left(  x,\gamma\right)  \right]  ,$ for all
$\left(  y,\gamma^{\prime},a\right)  \in\widehat{\overline{\mathcal{G}}\times
E\times A},$ then, due to the definition of $\Theta_{\overline{\mathcal{G}}%
}\left(  x,\gamma\right)  ,$ if $\mu\in\Theta_{\overline{\mathcal{G}}}\left(
x,\gamma\right)  ,$ by integrating the inequality w.r.t. $\mu$, it follows
that
\[
\int_{%
\mathbb{R}
^{m}\times E\times A}l_{\gamma^{\prime}}\left(  y,a\right)  \mu\left(
dyd\gamma^{\prime}da\right)  \geq\eta.
\]
Hence, $\Lambda^{\delta}\left(  x,\gamma\right)  \geq\Lambda^{\delta,\ast
}\left(  x,\gamma\right)  .$ To complete the proof, one needs to prove
$\Lambda^{\delta,\ast}\left(  x,\gamma\right)  \geq\delta v^{\delta}\left(
x,\gamma\right)  $. We use $v_{\varepsilon}^{\delta}$ given in Subsubsection
\ref{Subsubsection5.6.1} to infer%
\[
\delta v_{\varepsilon}^{\delta}\left(  x,\gamma\right)  \leq\mathcal{U}%
^{a}v_{\varepsilon}^{\delta}\left(  y,\gamma^{\prime}\right)  +l_{\gamma
^{\prime}}\left(  y,a\right)  -\delta\left[  v_{\varepsilon}^{\delta}\left(
y,\gamma^{\prime}\right)  -v_{\varepsilon}^{\delta}\left(  x,\gamma\right)
\right]  ,
\]
for all $\left(  y,\gamma^{\prime},a\right)  \in\widehat{\overline
{\mathcal{G}}\times E\times A}$. Hence, $\delta v_{\varepsilon}^{\delta
}\left(  x,\gamma\right)  \leq\Lambda^{\delta,\ast}\left(  x,\gamma\right)  .$
The proof is completed by taking the limit as $\varepsilon\rightarrow0$ and
recalling that (\ref{estimatesEta}) holds true.
\end{proof}

\subsection{Conclusion and comments}

The previous result can be interpreted in connection to Perron's method. If
$\varphi$ is a regular subsolution of (\ref{HJeps}) for $\rho=0,$
$\varepsilon=0$ on $\overline{\mathcal{G}}$ (i.e. such that
\[
\mathcal{U}^{a}\varphi\left(  y,\gamma^{\prime}\right)  +l_{\gamma^{\prime}%
}\left(  y,a\right)  -\delta\varphi\left(  y,\gamma^{\prime}\right)  \geq0,
\]
for all $\left(  y,\gamma^{\prime},a\right)  \in\widehat{\overline
{\mathcal{G}}\times E\times A}$), then $\delta\varphi\left(  x,\gamma\right)
\leq\Lambda^{\delta,\ast}\left(  x,\gamma\right)  =\delta v^{\delta}\left(
x,\gamma\right)  .$ Since we have exhibited a family ($\left(  v_{\varepsilon
}^{\delta}\left(  x,\gamma\right)  \right)  _{\varepsilon>0}$) converging to
$v^{\delta}\left(  x,\gamma\right)  ,$ it follows that $v^{\delta}$ is the
pointwise supremum over such regular subsolutions, hence giving Perron's
solution to the Hamilton-Jacobi integrodifferential system.

This implies a weak form of uniqueness for our solution. This approach has a
couple of advantages. First, it provides an approximating scheme for the value
function $v^{\delta}$ in the spirit of \cite{krylov_00},
\cite{barles_jakobsen_02} or \cite{BiswasJakobsenKaelsen2010}. However, the
speed of convergence is given by the estimates in Lemma
\ref{Projection_Lemma_eps} and are less explicit than the H\"{o}lder ones
exhibited in the cited papers. Second, having stated the equivalent problem on
a linear space of measures should prove useful for optimality issues (see
\cite{finlay_gaitsgory_lebedev_07} or, more recently,
\cite{Dufour_Stockbridge_12} in a general Markovian framework or \cite{G28} in
a Brownian one).

In a deterministic framework, there is an increasing literature dealing with
stronger forms of uniqueness based on comparison principles. Some of the
papers deal with frameworks similar to ours (e.g.
\cite{AchdouCamilliCutri2013}) and use the geodetic distance in the doubling
variable approach . These results have been generalized and simplified in
\cite{AchdouOudetTchou2013}. Another approach consists in introducing "vertex
test" functions. This allows to treat a generalized quasi-convex case in the
recent preprint \cite{ImbertMoneau2015_Flux}. A nice comparison between the
different notions of solution (corresponding to \cite{AchdouCamilliCutri2013},
\cite{Imbert_Monneau_Zidani} and \cite{Schieborn:2013aa}) is also provided in
\cite{CamilliMarchi2013}. Finally, let us note that, in our setting, the test
functions only need to be substituted in the gradient and, hence, adapting the
comparison methods of the previous papers should work rather smoothly.

With the assumptions of this section (in particular \textbf{(C)), }it follows
that any regular subsolution in the sense of Definition
\ref{Def_more_regular_test} is also a regular subsolution in the sense of
Definition \ref{Def_generalized_solution}. It follows that $v^{\delta}$ cannot
exceed the supremum over regular subsolutions in the sense of Definition
\ref{Def_generalized_solution}. Equality is obtained whenever a classical
comparison principle is available.

\section{ Appendix}

\subsection{Proof of Lemma \ref{Projection_Lemma}}

\begin{proof}
[Proof of Lemma \ref{Projection_Lemma}]We will consider several cases and
prove (i) in each case. We provide the construction for (ii) only in the first
case (a) and hint what is needed for the remaining cases.\qquad

(a) (i) Let us assume that $x=O$. If $y=O,$ then $\mathcal{P}_{O,y}\left(
\alpha\right)  =\alpha.$ Otherwise, we let $t_{y,O}:=\inf\left\{
t\geq0:y_{\gamma}\left(  t;y,a_{\gamma,1}^{-}\right)  =O\right\}  .$
Obviously,
\[
t_{y,O}\leq\frac{\left\vert y\right\vert ^{1-\kappa}}{\left(  1-\kappa\right)
\beta}\leq\frac{\rho_{\varepsilon}^{2}}{\left(  1-\kappa\right)  \beta}.
\]
(These estimates are for the "inactive" case; for the "active" one, one can
consider $\kappa=0$)$.$ For $\varepsilon$ small enough, one can assume,
without loss of generality that $\frac{\rho_{\varepsilon}}{\left(
1-\kappa\right)  \beta}<t_{\varepsilon}.$ We define
\[
\mathcal{P}_{O,y}\left(  \alpha\right)  \left(  t\right)  :=a_{\gamma,1}%
^{-}\mathbf{1}_{\left[  0,t_{y,O}\right]  }\left(  t\right)  +\alpha\left(
t-t_{y,O}\right)  \mathbf{1}_{(t_{y,O},\infty)}\left(  t\right)  ,
\]
for all $t\geq0.$ Then, one gets%
\begin{align*}
\left\vert y_{\gamma}\left(  t;y,\mathcal{P}_{O,y}\left(  \alpha\right)
\right)  -y_{\gamma}\left(  t;O,\alpha\right)  \right\vert  &  \leq\left\vert
y_{\gamma}\left(  t;y,\mathcal{P}_{O,y}\left(  \alpha\right)  \right)
-y\right\vert +\left\vert y\right\vert +\left\vert y\left(  t;O,\alpha\right)
\right\vert \\
&  \leq\left(  \frac{2\left\vert f\right\vert _{0}}{\left(  1-\kappa\right)
\beta}+1\right)  \left\vert y\right\vert ^{1-\kappa}\leq\left(  \frac
{2\left\vert f\right\vert _{0}}{\left(  1-\kappa\right)  \beta}+1\right)
\rho_{\varepsilon}^{2},
\end{align*}
if $t\in\left[  0,t_{y,O}\right]  $ and
\begin{align*}
\left\vert y_{\gamma}\left(  t;y,\mathcal{P}_{O,y}\left(  \alpha\right)
\right)  -y_{\gamma}\left(  t;O,\alpha\right)  \right\vert  &  =\left\vert
y_{\gamma}\left(  t;O,\alpha\right)  -y_{\gamma}\left(  t-t_{y,O}%
;O,\alpha\right)  \right\vert \\
&  \leq\frac{\left\vert f\right\vert _{0}}{\left(  1-\kappa\right)  \beta
}\left\vert y\right\vert ^{1-\kappa}\leq\frac{\left\vert f\right\vert _{0}%
}{\left(  1-\kappa\right)  \beta}\rho_{\varepsilon}^{2},
\end{align*}
if $t>t_{y,O}.$ Moreover, for every $T\geq0,$%
\begin{align*}
&  \left\vert \int_{0}^{T}e^{-\delta t}l_{\gamma}\left(  y_{\gamma}\left(
t;y,\mathcal{P}_{O,y}\left(  \alpha\right)  \right)  ,\mathcal{P}_{O,y}\left(
\alpha\right)  \left(  t\right)  \right)  dt-\int_{0}^{T}e^{-\delta
t}l_{\gamma}\left(  y_{\gamma}\left(  t;O,\alpha\right)  ,\alpha\left(
t\right)  \right)  dt\right\vert \\
&  \leq\int_{0}^{t_{y,O}}e^{-\delta t}\left\vert l_{\gamma}\left(  y_{\gamma
}\left(  t;y,\mathcal{P}_{O,y}\left(  \alpha\right)  \right)  ,\mathcal{P}%
_{O,y}\left(  \alpha\right)  \left(  t\right)  \right)  \right\vert
dt+\int_{0}^{t_{y,O}}e^{-\delta t}\left\vert l_{\gamma}\left(  y_{\gamma
}\left(  t;O,\alpha\right)  ,\alpha\left(  t\right)  \right)  \right\vert dt\\
&  +\mathbf{1}_{T>t_{y,O}}\left(  1-e^{-\delta t_{y,O}}\right)  \int%
_{0}^{T-t_{y,O}}e^{-\delta t}\left\vert l_{\gamma}\left(  y_{\gamma}\left(
t;O,\alpha\right)  ,\alpha\left(  t\right)  \right)  \right\vert dt\\
&  +\mathbf{1}_{T>t_{y,O}}\int_{T-t_{y,O}}^{T}e^{-\delta t}\left\vert
l_{\gamma}\left(  y_{\gamma}\left(  t;O,\alpha\right)  ,\alpha\left(
t\right)  \right)  \right\vert dt\\
&  \leq2\left\vert l\right\vert _{0}\frac{\left\vert y\right\vert ^{1-\kappa}%
}{\left(  1-\kappa\right)  \beta}+\frac{1}{\delta}\left\vert l\right\vert
_{0}\left(  1-e^{-\delta\frac{\left\vert y\right\vert ^{1-\kappa}}{\left(
1-\kappa\right)  \beta}}\right)  +\left\vert l\right\vert _{0}\frac{\left\vert
y\right\vert ^{1-\kappa}}{\left(  1-\kappa\right)  \beta}\\
&  \leq4\left\vert l\right\vert _{0}\frac{\left\vert y\right\vert ^{1-\kappa}%
}{\left(  1-\kappa\right)  \beta}\leq\frac{4\left\vert l\right\vert _{0}%
}{\left(  1-\kappa\right)  \beta}\rho_{\varepsilon}^{2}.
\end{align*}

(ii) If $\alpha\in\mathcal{A}_{ad},$ then we set
\[
\mathcal{P}_{\left(  O,\gamma\right)  }\left(  \alpha\right)  \left(
t;y,\eta\right)  =\left\{
\begin{array}
[c]{c}%
\mathcal{P}_{O,y}\left(  \alpha\left(  t;O,\gamma\right)  \right)  \text{ if
}\eta=\gamma,\left\vert y\right\vert \leq\rho_{\varepsilon}^{\frac{2}%
{1-\kappa}},\\
\alpha\left(  t;y,\eta\right)  ,\text{ otherwise.}%
\end{array}
\right.
\]
One only needs to notice that $y\mapsto t_{y,O}$ is Borel measurable to deduce
that $\mathcal{P}_{O,\gamma}\left(  \alpha\right)  \in\mathcal{A}_{ad}.$ In
the other cases, the construction is similar. We will just hint the
measurability properties needed to insure that the constructed function
$\mathcal{P}_{\left(  x,\gamma\right)  }\left(  \alpha\right)  $ is Borel
measurable in $\left(  t,y,\eta\right)  $.

(b) If $y=O$, we distinguish two cases :

(b1) The road is "inactive". Then, we introduce $t_{x,O}\left(  \alpha\right)
:=\inf\left\{  t>0:y_{\gamma}\left(  t;x,\alpha\right)  =O\right\}  $ and
define, if it is finite%
\[
\mathcal{P}_{x,y}\left(  \alpha\right)  \left(  t\right)  :=a_{\gamma,1}%
^{0}\mathbf{1}_{\left[  0,t_{x,O}\left(  \alpha\right)  \right]  }\left(
t\right)  +\alpha\left(  t\right)  \mathbf{1}_{(t_{x,O}\left(  \alpha\right)
,\infty)}\left(  t\right)  ,
\]
where $a_{\gamma,1}^{0}$ is given by (\textbf{Ab})$.$ Then, due to
(\textbf{Ab}), it is clear that
\[
\left\vert y_{\gamma}\left(  t;y,\mathcal{P}_{x,y}\left(  \alpha\right)
\right)  -y_{\gamma}\left(  t;x,\alpha\right)  \right\vert \leq\left\vert
x-y\right\vert \leq\rho_{\varepsilon}^{2},
\]
if $t\leq t_{x,O}\left(  \alpha\right)  $ and
\[
\left\vert y_{\gamma}\left(  t;y,\mathcal{P}_{x,y}\left(  \alpha\right)
\right)  -y_{\gamma}\left(  t;x,\alpha\right)  \right\vert =0,
\]
otherwise. We note that $y_{\gamma}\left(  t;y,\mathcal{P}_{x,y}\left(
\alpha\right)  \right)  =O$, for $t\leq t_{x,O}\left(  \alpha\right)  .$ Thus,
the assumption (\textbf{Ac}) yields
\begin{align*}
&  \left\vert \int_{0}^{T}e^{-\delta t}l\left(  y_{\gamma}\left(
t;y,\mathcal{P}_{x,y}\left(  \alpha\right)  \right)  ,\mathcal{P}_{x,y}\left(
\alpha\right)  \left(  t\right)  \right)  dt-\int_{0}^{T}e^{-\delta t}l\left(
y_{\gamma}\left(  t;x,\alpha\right)  ,\alpha\left(  t\right)  \right)
dt\right\vert \\
&  \leq\int_{0}^{T}e^{-\delta t}Lip(l)\left\vert x-y\right\vert dt\leq
\frac{Lip(l)}{\delta}\left\vert x-y\right\vert \leq\frac{Lip(l)}{\delta}%
\rho_{\varepsilon}^{2}.
\end{align*}

(b2) The road is "active". Then, we introduce $t_{y,x}:=\inf\left\{
t>0:y_{\gamma}\left(  t;y,a_{\gamma,1}^{+}\right)  =x\right\}  .$ Similar to
(a), one easily proves that $t_{y,x}\leq\frac{\rho_{\varepsilon}^{2}}{\beta}.$
In this case, we define%
\[
\mathcal{P}_{x,y}\left(  \alpha\right)  \left(  t\right)  :=a_{\gamma,1}%
^{+}\mathbf{1}_{\left[  0,t_{y,x}\right]  }\left(  t\right)  +\alpha\left(
t-t_{y,x}\right)  \mathbf{1}_{(t_{y,x},\infty)}\left(  t\right)  ,
\]
and get the same kind of estimates as in (a).

(c) We assume that $x\in J_{1}\cup\left\{  e_{1}\right\}  $ and $y\in J_{1}.$
Then, $\alpha\in\mathcal{A}_{\gamma,x}$ is admissible for $y$ (at least for
some small time). We define $t_{y}^{\ast}\left(  \alpha\right)  =\inf\left\{
t>0:y_{\gamma}\left(  t;y,\alpha\right)  \in\partial J_{1}\right\}  \wedge
\inf\left\{  t>0:y_{\gamma}\left(  t;x,\alpha\right)  =0\right\}  \wedge
t_{\varepsilon}.$ One notices, as before, that $y\mapsto t_{y}^{\ast}\left(
\alpha\right)  $ is Borel measurable.

(c1) If $t_{y}^{\ast}\left(  \alpha\right)  \geq t_{\varepsilon}$, then we let
$\mathcal{P}_{x,y}\left(  \alpha\right)  (t):=\alpha(t)\mathbf{1}_{\left[
0,t_{\varepsilon}\right)  }\left(  t\right)  +\alpha_{0}\left(  t;y_{\gamma
}\left(  t_{\varepsilon};y,\alpha\right)  ,\gamma\right)  \mathbf{1}_{\left[
t_{\varepsilon},\infty\right)  }\left(  t\right)  ,$ where $\alpha_{0}%
\in\mathcal{A}_{ad}$ and have%
\[
\left\vert y_{\gamma}\left(  t;y,\mathcal{P}_{x,y}\left(  \alpha\right)
\right)  -y_{\gamma}\left(  t;x,\alpha\right)  \right\vert \leq e^{Lip(f)t}%
\left\vert x-y\right\vert \leq\sqrt{\left\vert x-y\right\vert }\leq
\rho_{\varepsilon}^{\frac{1}{1-\kappa}},
\]
for all $t\leq t_{\varepsilon}.$ Also, one easily gets, for every $T\leq
t_{\varepsilon},$
\begin{align*}
&  \left\vert \int_{0}^{T}e^{-\delta t}l\left(  y_{\gamma}\left(
t;y,\mathcal{P}_{x,y}\left(  \alpha\right)  \right)  ,\mathcal{P}_{x,y}\left(
\alpha\right)  \left(  t\right)  \right)  dt-\int_{0}^{T}e^{-\delta t}l\left(
y_{\gamma}\left(  t;x,\alpha\right)  ,\alpha\left(  t\right)  \right)
dt\right\vert \\
&  \leq\frac{Lip\left(  l\right)  }{\delta}\sqrt{\left\vert x-y\right\vert
}\leq\frac{Lip\left(  l\right)  }{\delta}\rho_{\varepsilon}^{\frac{1}%
{1-\kappa}}.
\end{align*}
Since $\alpha_{0}\in\mathcal{A}_{ad},$ it follows that $\left(  t,y\right)
\mapsto\mathcal{P}_{x,y}\left(  \alpha\right)  (t)1_{t_{y}^{\ast}\left(
\alpha\right)  \geq t_{\varepsilon}}$ is Borel-measurable.

(c2) If $t_{y}^{\ast}\left(  \alpha\right)  <t_{\varepsilon}$ and $y_{\gamma
}\left(  t_{y}^{\ast}\left(  \alpha\right)  ;y,\alpha\right)  =e_{1}$, then,
in particular, $\left\vert y_{\gamma}\left(  t_{y}^{\ast}\left(
\alpha\right)  ;x,\alpha\right)  -e_{1}\right\vert <\sqrt{\left\vert
x-y\right\vert }\leq\rho_{\varepsilon}^{\frac{1}{1-\kappa}}.$ Of course, this
case is only interesting if $\alpha$ is no longer admissible. In particular,
when $A_{\gamma,e_{1}}\neq A^{\gamma,1}.$ Then, we introduce $t_{e_{1}%
,y_{\gamma}\left(  t_{y}^{\ast}\left(  \alpha\right)  ;x,\alpha\right)
}:=\inf\left\{  t\geq0:y_{\gamma}\left(  t;e_{1},a_{\gamma,1}\right)
=y_{\gamma}\left(  t_{y}^{\ast}\left(  \alpha\right)  ;x,\alpha\right)
\right\}  .$ One has $t_{e_{1},y_{\gamma}\left(  t_{y}^{\ast}\left(
\alpha\right)  ;x,\alpha\right)  }\leq\frac{\sqrt{\left\vert x-y\right\vert }%
}{\beta}.$ We define
\begin{align*}
\mathcal{P}_{x,y}\left(  \alpha\right)  \left(  t\right)   &  :=\alpha\left(
t\right)  \mathbf{1}_{\left[  0,t_{y}^{\ast}\left(  \alpha\right)  \right)
}\left(  t\right)  +a_{\gamma,1}\mathbf{1}_{\left[  t_{y}^{\ast}\left(
\alpha\right)  ,t_{y}^{\ast}\left(  \alpha\right)  +t_{e_{1},y_{\gamma}\left(
t_{y}^{\ast}\left(  \alpha\right)  ;x,\alpha\right)  }\right]  }\left(
t\right) \\
&  +\alpha\left(  t-t_{e_{1},y_{\gamma}\left(  t_{y}^{\ast}\left(
\alpha\right)  ;x,\alpha\right)  }\right)  \mathbf{1}_{\left(  t_{y}^{\ast
}\left(  \alpha\right)  +t_{e_{1},y_{\gamma}\left(  t_{y}^{\ast}\left(
\alpha\right)  ;x,\alpha\right)  },\infty\right)  }\left(  t\right)  .
\end{align*}
The functions $y\mapsto t_{y}^{\ast}\left(  \alpha\right)  ,$ $y\mapsto
y_{\gamma}\left(  t_{y}^{\ast}\left(  \alpha\right)  ;y,\alpha\right)  $ are
Borel measurable. Hence, so is $y\mapsto t_{e_{1},y_{\gamma}\left(
t_{y}^{\ast}\left(  \alpha\right)  ;x,\alpha\right)  }$. It follows that
\[
\left(  t,y\right)  \mapsto\mathcal{P}_{x,y}\left(  \alpha\right)
(t)\mathbf{1}_{t_{y}^{\ast}\left(  \alpha\right)  <t_{\varepsilon},\text{
}y_{\gamma}\left(  t_{y}^{\ast}\left(  \alpha\right)  ;y,\alpha\right)
=e_{1}}%
\]
is also Borel-measurable. One has%
\[
\left\vert y_{\gamma}\left(  t;y,\mathcal{P}_{x,y}\left(  \alpha\right)
\right)  -y_{\gamma}\left(  t;x,\alpha\right)  \right\vert \leq\sqrt
{\left\vert x-y\right\vert },
\]
if $t\leq t_{y}^{\ast}\left(  \alpha\right)  ,$%
\begin{align*}
\left\vert y_{\gamma}\left(  t;y,\mathcal{P}_{x,y}\left(  \alpha\right)
\right)  -y_{\gamma}\left(  t;x,\alpha\right)  \right\vert  &  \leq\left\vert
y_{\gamma}\left(  t-t_{y}^{\ast}\left(  \alpha\right)  ;e_{1},a_{\gamma
,1}\right)  -e_{1}\right\vert +\left\vert e_{1}-y_{\gamma}\left(  t_{y}^{\ast
}\left(  \alpha\right)  ;x,\alpha\right)  \right\vert \\
&  +\left\vert y_{\gamma}\left(  t_{y}^{\ast}\left(  \alpha\right)
;x,\alpha\right)  -y_{\gamma}\left(  t;x,\alpha\right)  \right\vert \\
&  \leq\left(  \frac{2\left\vert f\right\vert _{0}}{\beta}+1\right)
\sqrt{\left\vert x-y\right\vert },
\end{align*}
if $t\in\left[  t_{y}^{\ast}\left(  \alpha\right)  ,t_{y}^{\ast}\left(
\alpha\right)  +t_{e_{1},y_{\gamma}\left(  t_{y}^{\ast}\left(  \alpha\right)
;x,\alpha\right)  }\right]  .$ Finally, if $t>t_{y}^{\ast}\left(
\alpha\right)  +t_{e_{1},y_{\gamma}\left(  t_{y}^{\ast}\left(  \alpha\right)
;x,\alpha\right)  },$ then
\begin{align*}
&  \left\vert y_{\gamma}\left(  t;y,\mathcal{P}_{x,y}\left(  \alpha\right)
\right)  -y_{\gamma}\left(  t;x,\alpha\right)  \right\vert \\
&  =\left\vert
\begin{array}
[c]{c}%
y_{\gamma}\left(  t-t_{y}^{\ast}\left(  \alpha\right)  +t_{e_{1},y_{\gamma
}\left(  t_{y}^{\ast}\left(  \alpha\right)  ;x,\alpha\right)  };y_{\gamma
}\left(  t_{y}^{\ast}\left(  \alpha\right)  ;x,\alpha\right)  ,\mathcal{\alpha
}\left(  t_{y}^{\ast}\left(  \alpha\right)  +\cdot\right)  \right) \\
-y_{\gamma}\left(  t-t_{y}^{\ast}\left(  \alpha\right)  ;y_{\gamma}\left(
t_{y}^{\ast}\left(  \alpha\right)  ;x,\alpha\right)  ,\alpha\left(
t_{y}^{\ast}\left(  \alpha\right)  +\cdot\right)  \right)
\end{array}
\right\vert \\
&  \leq\left\vert f\right\vert _{0}\frac{\sqrt{\left\vert x-y\right\vert }%
}{\beta}.
\end{align*}
Moreover, if $T\leq t_{\varepsilon},$ one gets (similar to (a)),
\begin{align*}
&  \left\vert \int_{0}^{T}e^{-\delta t}l\left(  y_{\gamma}\left(
t;y,\mathcal{P}_{x,y}\left(  \alpha\right)  \right)  ,\mathcal{P}_{x,y}\left(
\alpha\right)  \left(  t\right)  \right)  dt-\int_{0}^{T}e^{-\delta t}l\left(
y_{\gamma}\left(  t;x,\alpha\right)  ,\alpha\left(  t\right)  \right)
dt\right\vert \\
&  \leq\int_{0}^{t_{y}^{\ast}\left(  \alpha\right)  }e^{-\delta t}Lip\left(
l\right)  \sqrt{\left\vert x-y\right\vert }dt+\frac{4\left\vert l\right\vert
_{0}}{\beta}\sqrt{\left\vert x-y\right\vert }.
\end{align*}

(c3) The case $t_{y}^{\ast}\left(  \alpha\right)  <t_{\varepsilon}$ and
$y_{\gamma}\left(  t_{y}^{\ast}\left(  \alpha\right)  ;y,\alpha\right)  =O:$
In particular, one gets $\left\vert y_{\gamma}\left(  t_{y}^{\ast}\left(
\alpha\right)  ;x,\alpha\right)  \right\vert \leq\sqrt{\left\vert
x-y\right\vert }\leq\rho_{\varepsilon}^{\frac{1}{1-\kappa}}.$

(c3.1) In the "active case", we consider $t_{O,y_{\gamma}\left(  t_{y}^{\ast
}\left(  \alpha\right)  ;x,\alpha\right)  }=\inf\left\{  t>0:y_{\gamma}\left(
t;O,a_{\gamma,1}^{+}\right)  =y_{\gamma}\left(  t_{y}^{\ast}\left(
\alpha\right)  ;x,\alpha\right)  \right\}  $ and define
\begin{align*}
\mathcal{P}_{x,y}\left(  \alpha\right)  \left(  t\right)   &  :=\alpha\left(
t\right)  \mathbf{1}_{\left[  0,t_{y}^{\ast}\left(  \alpha\right)  \right)
}\left(  t\right)  +a_{\gamma,1}^{+}\mathbf{1}_{\left[  t_{y}^{\ast}\left(
\alpha\right)  ,t_{y}^{\ast}\left(  \alpha\right)  +t_{O,y_{\gamma}\left(
t_{y}^{\ast}\left(  \alpha\right)  ;x,\alpha\right)  }\right]  }\left(
t\right) \\
&  +\alpha\left(  t-t_{O,y_{\gamma}\left(  t_{y}^{\ast}\left(  \alpha\right)
;x,\alpha\right)  }\right)  \mathbf{1}_{\left(  t_{y}^{\ast}\left(
\alpha\right)  +t_{O,y_{\gamma}\left(  t_{y}^{\ast}\left(  \alpha\right)
;x,\alpha\right)  },\infty\right)  }\left(  t\right)  .
\end{align*}
One gets the same estimates (and measurability properties) as in (c2).

(c3.2) The "inactive case" is similar to (b1\textbf{).} We consider
\begin{align*}
\mathcal{P}_{x,y}\left(  \alpha\right)  \left(  t\right)   &  :=\alpha\left(
t\right)  \mathbf{1}_{\left[  0,t_{y}^{\ast}\left(  \alpha\right)  \right)
}\left(  t\right)  +a_{\gamma,1}^{0}\mathbf{1}_{\left[  t_{y}^{\ast}\left(
\alpha\right)  ,t_{y}^{\ast}\left(  \alpha\right)  +t_{y_{\gamma}\left(
t_{y}^{\ast}\left(  \alpha\right)  ;x,\alpha\right)  ,O}\right]  }\left(
t\right) \\
&  +\alpha\left(  t-t_{y_{\gamma}\left(  t_{y}^{\ast}\left(  \alpha\right)
;x,\alpha\right)  ,O}\right)  \mathbf{1}_{\left(  t_{y}^{\ast}\left(
\alpha\right)  +t_{y_{\gamma}\left(  t_{y}^{\ast}\left(  \alpha\right)
;x,\alpha\right)  ,O},\infty\right)  }\left(  t\right)  ,
\end{align*}
for all $t\geq0.$ The functions $y\mapsto t_{y}^{\ast}\left(  \alpha\right)
,$ $y\mapsto y_{\gamma}\left(  t_{y}^{\ast}\left(  \alpha\right)
;x,\alpha\right)  $ are Borel measurable. Hence, so is $y\mapsto t_{y_{\gamma
}\left(  t_{y}^{\ast}\left(  \alpha\right)  ;x,\alpha\right)  ,O}\left(
a_{\gamma,1}^{0}\right)  $. It follows that
\[
\left(  t,y\right)  \mapsto\mathcal{P}_{x,y}\left(  \alpha\right)
(t)\mathbf{1}_{t_{y}^{\ast}\left(  \alpha\right)  <t_{\varepsilon},\text{
}y_{\gamma}\left(  t_{y}^{\ast}\left(  \alpha\right)  ;y,\alpha\right)  =O}%
\]
is also Borel-measurable.

One easily notices that
\[
\left\vert y_{\gamma}\left(  t;y,\mathcal{P}_{x,y}\left(  \alpha\right)
\right)  -y_{\gamma}\left(  t;x,\alpha\right)  \right\vert \leq\sqrt
{\left\vert x-y\right\vert }\leq\rho_{\varepsilon},\text{ if }0\leq t\leq
t_{y}^{\ast}\left(  \alpha\right)  +t_{y_{\gamma}\left(  t_{y}^{\ast}\left(
\alpha\right)  ;x,\alpha\right)  ,O},
\]
and $y_{\gamma}\left(  t;y,\mathcal{P}_{x,y}\left(  \alpha\right)  \right)
=y_{\gamma}\left(  t;x,\alpha\right)  $ if $t>t_{y}^{\ast}\left(
\alpha\right)  +t_{y_{\gamma}\left(  t_{y}^{\ast}\left(  \alpha\right)
;x,\alpha\right)  ,O}.$ Using the assumption (\textbf{Ac}) on $\left[
t_{y}^{\ast}\left(  \alpha\right)  ,t_{y}^{\ast}\left(  \alpha\right)
+t_{y_{\gamma}\left(  t_{y}^{\ast}\left(  \alpha\right)  ;x,\alpha\right)
,O}\right]  $, one gets%
\begin{align*}
&  \left\vert \int_{0}^{T}e^{-\delta t}l\left(  y_{\gamma}\left(
t;y,\mathcal{P}_{x,y}\left(  \alpha\right)  \right)  ,\mathcal{P}_{x,y}\left(
\alpha\right)  \left(  t\right)  \right)  dt-\int_{0}^{T}e^{-\delta t}l\left(
y_{\gamma}\left(  t;x,\alpha\right)  ,\alpha\left(  t\right)  \right)
dt\right\vert \\
&  \leq\int_{0}^{t_{y}^{\ast}\left(  \alpha\right)  }e^{-\delta t}Lip\left(
l\right)  \sqrt{\left\vert x-y\right\vert }dt+\int_{t_{y}^{\ast}\left(
\alpha\right)  }^{\left(  t_{y}^{\ast}\left(  \alpha\right)  +t_{y_{\gamma
}\left(  t_{y}^{\ast}\left(  \alpha\right)  ;x,\alpha\right)  ,O}\right)
\wedge T}e^{-\delta t}Lip(l)\sqrt{\left\vert x-y\right\vert }dt\\
&  \leq\frac{1}{\delta}Lip(l)\sqrt{\left\vert x-y\right\vert }.
\end{align*}

(c4) If $t_{y}^{\ast}\left(  \alpha\right)  <t_{\varepsilon}$ and $y_{\gamma
}\left(  t_{y}^{\ast}\left(  \alpha\right)  ;x,\alpha\right)  =O,$ then we
proceed as in (a). We let
\[
t_{y_{\gamma}\left(  t_{y}^{\ast}\left(  \alpha\right)  ;y,\alpha\right)
,O}:=\inf\left\{  t\geq0:y_{\gamma}\left(  t;y_{\gamma}\left(  t_{y}^{\ast
}\left(  \alpha\right)  ;y,\alpha\right)  ,a_{\gamma,1}^{-}\right)
=O\right\}  .
\]
Obviously, $t_{y_{\gamma}\left(  t_{y}^{\ast}\left(  \alpha\right)
;y,\alpha\right)  ,O}\leq\frac{\sqrt{\left\vert x-y\right\vert }^{1-\kappa}%
}{\left(  1-\kappa\right)  \beta}.$ We set
\begin{align*}
\mathcal{P}_{x,y}\left(  \alpha\right)  \left(  t\right)   &  :=\alpha\left(
t\right)  \mathbf{1}_{\left[  0,t_{y}^{\ast}\left(  \alpha\right)  \right)
}\left(  t\right)  +a_{\gamma,1}^{-}\mathbf{1}_{\left[  t_{y}^{\ast}\left(
\alpha\right)  ,t_{y}^{\ast}\left(  \alpha\right)  +t_{y_{\gamma}\left(
t_{y}^{\ast}\left(  \alpha\right)  ;y,\alpha\right)  ,O}\right]  }\left(
t\right) \\
&  +\alpha\left(  t-t_{y_{\gamma}\left(  t_{y}^{\ast}\left(  \alpha\right)
;y,\alpha\right)  ,O}\right)  \mathbf{1}_{\left(  t_{y}^{\ast}\left(
\alpha\right)  +t_{y_{\gamma}\left(  t_{y}^{\ast}\left(  \alpha\right)
;y,\alpha\right)  ,O},\infty\right)  }\left(  t\right)  ,
\end{align*}
for all $t\geq0$ and the estimates follow. The measurability properties hold
as before.

(d) Finally, we assume that $y=e_{1}.$ Again, we only modify $\alpha$ if
$A_{\gamma,e_{1}}\neq A^{\gamma,1}$. In this eventuality, we define
$t_{e_{1},x}:=\inf\left\{  t\geq0:y_{\gamma}\left(  t;y,a_{\gamma,1}\right)
=x\right\}  ,$ where $a_{\gamma,1}$ appears in (\textbf{Aa})$.$ Then
$t_{e_{1},x}\leq\frac{\left\vert x-y\right\vert }{\beta}.$ We let%
\[
\mathcal{P}_{x,e_{1}}\left(  \alpha\right)  \left(  t\right)  :=a_{\gamma
,1}\mathbf{1}_{\left[  0,t_{e_{1},x}\right]  }\left(  t\right)  +\alpha\left(
t-t_{e_{1},x}\right)  \mathbf{1}_{\left(  t_{e_{1},x},\infty\right)  }\left(
t\right)  .
\]
and get the conclusion.

The proof of our lemma is now complete.
\end{proof}

\subsection{Proof of Lemma \ref{Projection_Lemma_eps}.}

For any $y\in\lbrack O,(1+\varepsilon)e_{i}]$ (with $\gamma\in E_{i}^{active}%
$), we set
\[
a_{\gamma,i}^{\mathrm{opt},+}(y)=\underset{a\in A_{\gamma,y}}{\mathrm{argmax}%
}\langle f_{\gamma}\left(  y,a\right)  ,e_{i}\rangle.
\]
It is clear that
\begin{equation}
\left.
\begin{array}
[c]{l}%
\left\langle f_{\gamma}\left(  y^{\prime},a\right)  -f_{\gamma}\left(
y,a_{\gamma,i}^{\mathrm{opt},+}(y)\right)  ,e_{i}\right\rangle \leq
\underset{a^{\prime}\in A_{\gamma,e_{i}}}{\sup}\left\vert f_{\gamma}\left(
y^{\prime},a^{\prime}\right)  -f_{\gamma}\left(  y,a^{\prime}\right)
\right\vert \leq Lip\left(  f\right)  \left\vert y^{\prime}-y\right\vert ,\\
\left\langle f_{\gamma}\left(  y,a\right)  -f_{\gamma}\left(  y,a_{\gamma
,i}^{\mathrm{opt},+}(y)\right)  ,e_{i}\right\rangle \leq0,
\end{array}
\right.  \label{estimatesArgmax}%
\end{equation}
for all $y,y^{\prime}\in\lbrack O,(1+\varepsilon)e_{i}]$. We also let
\[
d_{geo}\left(  x,y\right)  :=\left\{
\begin{array}
[c]{c}%
\left\vert x-y\right\vert ,\text{ if }x,y\in\left[  -1-\varepsilon
,1+\varepsilon\right]  e_{i},\\
\left\vert x\right\vert +\left\vert y\right\vert ,\text{ if }x\in\left[
-1-\varepsilon,1+\varepsilon\right]  e_{i},y\in\left[  -1-\varepsilon
,1+\varepsilon\right]  e_{j},\text{ }i\neq j
\end{array}
\right.  .
\]

\begin{proof}
[Proof of Lemma \ref{Projection_Lemma_eps}.]We will prove only the estimates
on the trajectory. The estimates on the partial cost follow from the
construction $\mathcal{P}_{x}\left(  \alpha\right)  $ which coincides with
$\alpha$ except at the end points (where (\textbf{C}) applies; see also the
similar condition (Ac) and the proof of Lemma \ref{Projection_Lemma}). The
assertion (ii) follows similar patterns to Lemma \ref{Projection_Lemma}.

We aim at constructing $\tilde{\alpha}:=\mathcal{P}_{x}\left(  \alpha\right)
.$ We let $r_{0}\leq\varepsilon$ (to be specified later on). We can assume,
without loss of generality, that $x\neq O.$ (Should this not be the case, see
Case 3). Then $\alpha$ is locally admissible. We set
\[
\tau_{0}:=\inf\left\{  t\geq0~:~d_{geo}\left(  y_{\gamma}(t;x,\alpha
),{y_{\gamma}^{\rho_{\varepsilon}}\left(  t;x,\overline{\alpha}\right)
}\right)  \geq r_{0}\right\}  .
\]
If $\tau_{0}\geq t_{\varepsilon},$ the conclusion follows. Otherwise, the time
where $y_{\gamma}$ meets again our target $y^{\rho_{\varepsilon}}$ will be
referred to as \textquotedblleft renewal time\textquotedblright. We give the
construction of $\tilde{\alpha}$ on $[\tau_{0},t_{\varepsilon}]$ prior to
renewal time. We let $\tau_{O}^{\varepsilon}$ be the exit time of the target
from the branch,
\[
\tau_{O}^{\varepsilon}=\inf\left\{  t\geq\tau_{0}~:~{y_{\gamma}^{\rho
_{\varepsilon}}\left(  t;x,\overline{\alpha}\right)  }=O\right\}  .
\]
(Hence, $\tau_{O}^{\varepsilon}>\tau_{0}$). Let us assume that $\tilde{\alpha
}$ has been constructed up to some time $\tau_{0}\leq t^{\ast}\leq\tau
_{O}^{\varepsilon}$ before the renewal time such that
\begin{equation}
d_{geo}\left(  y_{\gamma}^{\ast},y_{\gamma}^{\rho_{\varepsilon},\ast}\right)
\leq\omega_{\varepsilon}\left(  t^{\ast},r_{0}\right)  , \tag{R}\label{R}%
\end{equation}
where we used the notation $y_{\gamma}^{\ast}=y_{\gamma}\left(  t^{\ast
};x,\widetilde{\alpha}\right)  $ and $y_{\gamma}^{\rho_{\varepsilon},\ast
}=y_{\gamma}^{\rho_{\varepsilon}}\left(  t^{\ast};x,\overline{\alpha}\right)
$. Even if this is not crucial for the rest of the proof, remark that renewal
cannot occur before $\tau_{0}+\frac{r_{0}}{2|f|_{0}}$, so that this iterative
procedure will be applied only a finite number of times.

\textbf{Case 1: } $y_{\gamma}$ and $y_{\gamma}^{\rho_{\varepsilon}}$ are on
the same branch (say $\left[  O,(1+\varepsilon)e_{1}\right]  ;$ the case when
$y_{\gamma}$ and $y^{\rho_{\varepsilon}}$ are on a "new" branch $\left[
O,-\varepsilon e_{1}\right]  $ is similar), and $y_{\gamma}$ lies between the
junction $O$ and $y_{\gamma}^{\rho_{\varepsilon}}$ (i.e. $0\leq\langle
y_{\gamma}^{\ast},e_{1}\rangle<\langle y_{\gamma}^{\rho_{\varepsilon},\ast
},e_{1}\rangle$). We let
\begin{align*}
t_{out}  &  =\inf\left\{  t\geq0:y_{\gamma}\left(  t;y_{\gamma}^{\ast}%
,\alpha(t^{\ast}+\cdot)\right)  =\left(  1+\varepsilon\right)  e_{1}\right\}
,\;\;t_{out}^{\rho_{\varepsilon}}=\inf\left\{  t\geq0:y_{\gamma}%
^{\rho_{\varepsilon}}\left(  t;y_{\gamma}^{{\rho_{\varepsilon}},\ast
},\overline{\alpha}(t^{\ast}+\cdot)\right)  =\left(  1+\varepsilon\right)
e_{1}\right\}  ,\\
t_{0}^{\rho_{\varepsilon}}  &  =\inf\left\{  t\geq0:y_{\gamma}^{\rho
_{\varepsilon}}\left(  t;y_{\gamma}^{\rho_{\varepsilon,\ast}},\overline
{\alpha}(t^{\ast}+\cdot)\right)  =O\right\}  ,\;\;t_{0}=\inf\left\{
t\geq0:y_{\gamma}\left(  t;y_{\gamma}^{\ast},\alpha(t^{\ast}+\cdot)\right)
=O\right\}  .
\end{align*}
Let us introduce $t_{act}=\min\left(  t_{out},t_{out}^{\rho_{\varepsilon}%
},t_{0},t_{0}^{\rho_{\varepsilon}}\right)  $. Obviously, prior to the renewal
time$,$ only $t_{0}$ is relevant (since $t_{out},t_{0}^{\rho_{\varepsilon}}$
cannot occur without renewal and if $t_{out}^{\rho_{\varepsilon}}<t_{0},$ then
$\alpha$ is still locally admissible for the follower $y_{\gamma}$). We
distinguish between the cases

(a1) If $t_{act}>0$, we extend $\tilde{\alpha}$ by setting $\tilde{\alpha
}(t)=\alpha\left(  t\right)  $, if $t^{\ast}<t\leq t^{\ast}+t_{act}.$
Gronwall's inequality yields
\[
\left\vert y_{\gamma}\left(  t;x,\tilde{\alpha}\right)  -y_{\gamma}%
^{\rho_{\varepsilon}}\left(  t;x,\overline{\alpha}\right)  \right\vert
\leq\omega_{\varepsilon}\left(  t-t^{\ast};|y_{\gamma}^{\ast}-y_{\gamma}%
^{\rho_{\varepsilon},\ast}|\right)  ,
\]
for all $t^{\ast}<t\leq t^{\ast}+t_{act}$.

(a2) If $t_{act}=t_{0}=0$, then we necessarily have that $t_{0}^{\rho
_{\varepsilon}}>0$. In this case $y_{\gamma}^{\ast}=O$ and $\langle y_{\gamma
}^{\rho_{\varepsilon}}\left(  t^{\ast};x,\overline{\alpha}\right)
,e_{1}\rangle>0$.

(a2.1) The active case (by far the most complicated) $\gamma\in E_{1}%
^{active}$. In order to simplify our notations, denote, in this case,
$a_{\gamma,O}^{+}=a_{\gamma,1}^{\mathrm{opt},+}(O)$. We introduce
\begin{align*}
{t}_{control}  &  =\inf\{t>0~:~y_{\gamma}\left(  t;,y_{\gamma}^{\ast
},a_{\gamma,O}^{+}\right)  =r_{\varepsilon}^{\prime}e_{1}\}\\
{t}_{collision}  &  =\inf\{t>0~:~y_{\gamma}\left(  t;,y_{\gamma}^{\ast
},a_{\gamma,O}^{+}\right)  =y_{\gamma}^{\rho_{\varepsilon}}\left(  t^{\ast
}+t;,y_{\gamma}^{\rho_{\varepsilon},\ast},\overline{\alpha}(t^{\ast}%
+\cdot)\right)  \}
\end{align*}
Note that because of the continuity of the trajectories and since
$r_{\varepsilon}^{\prime}>0$, we have ${t}_{control}>0$ and ${t}%
_{collision}>0$. We extend naturally $\tilde{\alpha}$ by setting
\[
\tilde{\alpha}\left(  t+t^{\ast}\right)  =a_{\gamma,O}^{+}\text{, if }%
t\in\left(  0,{t}_{collision}\wedge{t}_{control}\right]  .
\]
With this extension, our assumptions guarantee that $\langle y_{\gamma}\left(
t+t^{\ast};x,\tilde{\alpha}\right)  ,e_{1}\rangle\geq Lip(f)\beta>0$ and the
junction $O$ is now a reflecting barrier for $t\mapsto y_{\gamma}(t;y_{\gamma
}^{\ast},a_{\gamma,O}^{+}).$ Note also that for any $t\leq{t}_{collision}%
\wedge{t}_{control}$, we have $\langle y_{\gamma}^{\rho_{\varepsilon}}\left(
t+t^{\ast};x,\overline{\alpha}\right)  ,e_{1}\rangle>0$. For every $0<t\leq
t_{collision}\wedge t_{control}$, one uses (\ref{estimatesArgmax}) to get
\begin{align*}
&  \left\vert y_{\gamma}^{\rho_{\varepsilon}}\left(  t+t^{\ast};x,\overline
{\alpha}\right)  -y_{\gamma}\left(  t+t^{\ast};x,\tilde{\alpha}\right)
\right\vert =\left\langle y_{\gamma}^{\rho_{\varepsilon}}\left(  t;y_{\gamma
}^{\rho_{\varepsilon},\ast},\overline{\alpha}\left(  t^{\ast}+\cdot\right)
\right)  -y_{\gamma}\left(  t;y_{\gamma}^{\ast},a_{\gamma,O}^{+}\right)
,e_{1}\right\rangle \\
&  =\langle(y_{\gamma}^{\rho_{\varepsilon},\ast}-y_{\gamma}^{\ast}%
),e_{1}\rangle+\int_{0}^{t}\left\langle f_{\gamma}^{\rho_{\varepsilon}}\left(
y_{\gamma}^{\rho_{\varepsilon}}\left(  s;y_{\gamma}^{\rho_{\varepsilon},\ast
},\overline{\alpha}\left(  t^{\ast}+\cdot\right)  \right)  ,\overline{\alpha
}\left(  t^{\ast}+\cdot\right)  \right)  -f_{\gamma}\left(  y_{\gamma}\left(
s;y_{\gamma}^{\ast},a_{\gamma,O}^{+}\right)  ,a_{\gamma,O}^{+}\right)
,e_{1}\right\rangle ds\\
&  \leq\langle(y_{\gamma}^{\rho_{\varepsilon},\ast}-y_{\gamma}^{\ast}%
),e_{1}\rangle+\int_{0}^{t}Lip\left(  f\right)  \left(  \rho_{\varepsilon
}+\left\vert y_{\gamma}^{\rho_{\varepsilon}}\left(  s+t^{\ast};x,\overline
{\alpha}\right)  -y_{\gamma}\left(  s+t^{\ast};x,\tilde{\alpha}\right)
\right\vert \right)  ds\\
&  +\int_{0}^{t}\left[  \left\langle
\begin{array}
[c]{c}%
f_{\gamma}\left(  y_{\gamma}\left(  s;y_{\gamma}^{\ast},a_{\gamma,O}%
^{+}\right)  ,\alpha\left(  t^{\ast}+\cdot\right)  \right)  -f_{\gamma}\left(
O,a_{\gamma,O}^{+}\right) \\
+f_{\gamma}\left(  O,a_{\gamma,O}^{+}\right)  -f_{\gamma}\left(  y_{\gamma
}\left(  s;y_{\gamma}^{\ast},a_{\gamma,O}^{+}\right)  ,a_{\gamma,O}%
^{+}\right)
\end{array}
,e_{1}\right\rangle \right]  ds\\
&  \leq\left\vert y_{\gamma}^{\rho_{\varepsilon},\ast}-y_{\gamma}^{\ast
}\right\vert +Lip\left(  f\right)  \left[  \left(  \rho_{\varepsilon
}+2r_{\varepsilon}^{\prime}\right)  t+\int_{0}^{t}\left\vert y_{\gamma}%
^{\rho_{\varepsilon}}\left(  s+t^{\ast};x,\overline{\alpha}\right)
-y_{\gamma}\left(  s+t^{\ast};x,\tilde{\alpha}\right)  \right\vert ds\right]
.
\end{align*}
Using Gronwall's inequality and our assumptions on $r_{\varepsilon}^{\prime}$,
we deduce that for any $0<t\leq t_{control}\wedge t_{collision}$,
\[
\left\vert y_{\gamma}\left(  t+t^{\ast};x,\tilde{\alpha}\right)  -y_{\gamma
}^{\rho_{\varepsilon}}\left(  t+t^{\ast};x,\overline{\alpha}\right)
\right\vert \leq\omega_{\varepsilon}\left(  t;\left\vert y_{\gamma}%
^{\rho_{\varepsilon},\ast}-y_{\gamma}^{\ast}\right\vert \right)  .
\]
Thus, we have constructed an extension of $t\mapsto\tilde{\alpha}(t)$
satisfying $(R)$ during an increment of some strictly positive time
$t_{control}\wedge t_{collision}$.

(a2.2) In the inactive case, it suffices to continue with the control $\alpha$
(since, in this case, $f_{\gamma}\left(  O,a\right)  =0,$ for all $a\in
A^{\gamma,1}$) up till $t_{collision}$ (or $t_{\varepsilon}$).

\textbf{Case 2}~:~ We use the same notations as in the first case and aim at
giving the control when $\tilde{\alpha}$ has been constructed up to some time
$\tau_{0}\leq t^{\ast}\leq\tau_{O}^{\varepsilon}$ such that renewal does not
occur at $t^{\ast}$ and both motions are at time $t^{\ast}$ on the same active
branch (say $\left[  O,(1+\varepsilon)e_{1}\right]  $). Contrary to Case 1, in
this case we are assuming that $0<\langle y_{\gamma}^{\rho_{\varepsilon},\ast
},e_{1}\rangle<\langle y_{\gamma}^{\ast},e_{1}\rangle$. We distinguish the
following cases

(b1) If $t_{act}>0$. In this case we proceed exactly as in case (a1) and get
the same conclusion.

(b2) If $t_{act}=t_{out}=0$ then $y_{\gamma}^{\ast}=(1+\varepsilon)e_{1}$ and
we have $t_{out}^{\rho_{\varepsilon}}>0$. This case is completely symmetric to
case (a2.1) but with motions starting at $t^{\ast}$ near $(1+\varepsilon
)e_{1}$. The conclusion is similar.

(The case when $y_{\gamma}^{\ast}=-\varepsilon e_{1}$ is similar to (a2.1) if
$\gamma\in E_{1}^{active}$ and to (a2.2) in the inactive case.)

\textbf{Case 3}~:~control when $y_{\gamma}^{\rho_{\varepsilon}}\left(
t^{\ast};x,\overline{\alpha}\right)  \in\left[  O,(1+\varepsilon)e_{j}\right]
$ and $y_{\gamma}\left(  t^{\ast};x,\alpha^{\ast}\right)  \in\left[
O,(1+\varepsilon)e_{i}\right]  $ with $i\neq j$. In particular, the two points
may be at the intersection or the target is at the intersection and the
follower is not. We can assume, without loss of generality, that $\gamma\in
E_{j}^{active}$. (Otherwise, recalling that we start at the same initial
point, this situation can only happen if $y_{\gamma}^{\rho_{\varepsilon},\ast
}=O$ and no active branch exists. Then, whatever the control, $y_{\gamma}$ can
only get closer to $O.$) In this case, we introduce
\begin{align*}
\hat{t}_{O}  &  =\inf\{t>0~:~y_{\gamma}\left(  t;y_{\gamma}^{\ast}%
,a_{\gamma,i}^{-}\right)  =O\}\\
\hat{t}_{collision}  &  =\inf\{t>0~:~y_{\gamma}\left(  t;y_{\gamma}^{\ast
},a_{\gamma,i}^{-}\right)  =y_{\gamma}^{\rho_{\varepsilon}}\left(  t^{\ast
}+t;x,\overline{\alpha}\right)  \}
\end{align*}
and we extend $t\mapsto\tilde{\alpha}\left(  t\right)  $ up to time $t^{\ast
}+\hat{t}_{O}\wedge\hat{t}_{collision}$ by setting
\[
\tilde{\alpha}(t)=a_{\gamma,i}^{-}\text{, for }t^{\ast}<t<t^{\ast}+\hat{t}%
_{O}\wedge\hat{t}_{collision}.
\]
Since by assumption $d_{geo}\left(  y_{\gamma}^{\ast},y_{\gamma}%
^{\rho_{\varepsilon},\ast}\right)  \leq\omega_{\varepsilon}(t^{\ast};r_{0})$,
we have that
\[
0<\hat{t}_{O}\wedge\hat{t}_{collision}\leq\frac{\left(  \omega_{\varepsilon
}(t^{\ast};r_{0})\right)  ^{1-\kappa}}{\left(  1-\kappa\right)  \beta}.
\]
Hence, with such a construction we have that
\[
d_{geo}\left(  y_{\gamma}\left(  t;y_{\gamma}^{\ast},\tilde{\alpha}\right)
,y_{\gamma}^{\rho_{\varepsilon}}\left(  t;y_{\gamma}^{\rho_{\varepsilon},\ast
},\overline{\alpha}\right)  \right)  |\leq\left(  \frac{|f|_{0}}{\left(
1-\kappa\right)  \beta}+1\right)  \left(  \omega_{\varepsilon}(t^{\ast}%
;r_{0})\right)  ^{1-\kappa},
\]
for all $t<\hat{t}_{O}\wedge\hat{t}_{collision}$. If $\hat{t}_{O}=\hat{t}%
_{O}\wedge\hat{t}_{collision}$, we arrive at $y_{\gamma}\left(  \hat{t}%
_{O};y_{\gamma}^{\ast},\tilde{\alpha}\right)  =O$. If every road is inactive,
we continue to stay at $O.$

(c1) If $y_{\gamma}^{\rho_{\varepsilon}}\left(  \hat{t}_{O};y_{\gamma}%
^{\rho_{\varepsilon},\ast},\overline{\alpha}\right)  \neq O$ we are back to
case 1 but with $r_{0}$ now replaced by $r_{0}^{\prime}$ lower than $\left(
\frac{|f|_{0}}{\left(  1-\kappa\right)  \beta}+1\right)  \left(
\omega_{\varepsilon}(t^{\ast};r_{0})\right)  ^{1-\kappa}$~:~even if there has
been a deterioration of the distance between $y_{\gamma}$ and $y^{\rho
_{\varepsilon}}$ (not exceeding $\left(  \frac{|f|_{0}}{\left(  1-\kappa
\right)  \beta}+1\right)  \left(  \omega_{\varepsilon}(t^{\ast};r_{0})\right)
^{1-\kappa}$ because we are back to case 1, the situation of case 3 (and also
the situation of (b2)) will never happen before some renewal time occurs.
Consequently, in the situation of case 3 we are always allowed to take in
$(R)$ the same value for $r_{0}$ (and we choose $r_{0}=r_{\varepsilon}%
^{\prime}$).

(c2) Finally, we assume $y_{\gamma}^{\rho_{\varepsilon}}\left(  \hat{t}%
_{O};y_{\gamma}^{\rho_{\varepsilon},\ast},\overline{\alpha}\right)  =O$. If
every road is inactive, then $y_{\gamma}$ stays at $O$ and $y_{\gamma}%
^{\rho_{\varepsilon}}$ cannot go further than $\rho_{\varepsilon}$. Otherwise,
let us assume that some $j^{\prime}$ is active. Then, we take $\tilde{\alpha
}\left(  t\right)  =a_{\gamma,j^{\prime}}^{+}$ for some very small (yet
strictly positive) time $t^{\ast}+\hat{t}_{O}<t\leq t^{\ast}+\hat{t}_{O}%
+\frac{r_{\varepsilon}^{\prime}}{2\left\vert f\right\vert _{0}}$ and get
\[
d_{geo}\left(  y_{\gamma}^{\rho_{\varepsilon}}\left(  t;x,\overline{\alpha
}\right)  ,y_{\gamma}\left(  t;x,\widetilde{\alpha}\right)  \right)  \leq
r_{\varepsilon}^{\prime},
\]
which allows one to iterate.

\textbf{Conclusion} Gathering all these results together, the constructed
strategy $\tilde{\alpha}$ is such that
\[
\left\vert y_{\gamma}\left(  t;x,\tilde{\alpha}\right)  -y_{\gamma}%
^{\rho_{\varepsilon}}\left(  t;x,\overline{\alpha}\right)  \right\vert
\leq\omega_{\varepsilon}(t_{\varepsilon};\Phi(\varepsilon)).
\]
for any $t\leq t_{\varepsilon}$ and the lemma is proved.
\end{proof}

\subsection{Some hints on the proof of Lemma \ref{Projection_Lemma_eps_Partii}%
}

The reader is invited to note that, if (C) holds true, then $l\left(
y,a\right)  =l\left(  \Pi_{\overline{\mathcal{G}}}\left(  y\right)  ,a\right)
,$for all $y\in\overline{\mathcal{G}}^{+,\varepsilon}$. Hence, the same kind
of cost can be reached by :

- hurrying to $O$ when the target is at $O,$ then wait for collision by

\qquad- staying at $O$ when the target enters a fictive road from the
intersection if a control $a$ such that $f\left(  O,a\right)  =0$ exists (for
example, in the inactive case).

\qquad- or mimic staying at $O$ by making very small trips (see case (c2) of
the previous Lemma);

- at $e_{1}:$

\qquad- if $\left\langle f\left(  e_{1},a\right)  ,e_{1}\right\rangle \leq0,$
for all $a,$ we are done, since the target will never enter $\left(
1,1+\varepsilon\right]  e_{1}$ (recall we start from $\overline{\mathcal{G}%
}).$

\qquad- otherwise, there exists $\left\langle f\left(  e_{1},\widetilde{a}%
\right)  ,e_{1}\right\rangle >\beta^{\prime}>0$ and, by our assumption, we
also have $\left\langle f\left(  e_{1},a_{\gamma,1}\right)  ,e_{1}%
\right\rangle <-\beta.$ Then, again, we mimic staying at $e_{1}$ by making
very small trips until collision.

The same kind of assertion are valid for $\lambda$ and $Q$ (notice the
definition of these terms on "fictive" roads). The trajectories around $O$ are
close due to the $\varepsilon$ distance from $\overline{\mathcal{G}%
}^{+,\varepsilon}$ to $\overline{\mathcal{G}}$ and as in the previous
argument, coming around the intersection can only occur once before collision.

\bibliographystyle{plain}
\bibliography{bibliografie_06092014}

\end{document}